\documentclass[11pt,usenames,dvipsnames]{amsart}

\newskip\stdskip
\usepackage{amsfonts}
\usepackage{amsmath}
\usepackage{amssymb}
\usepackage{amstext}
\usepackage{amsthm}
\usepackage{tikz}
\usepackage{subcaption}
\usepackage{slashed}

\usetikzlibrary{matrix}
\usetikzlibrary{shapes,calc,decorations.markings}

\usepackage[all]{xy}

\usepackage[dvips]{epsfig}
\usepackage{pinlabel}
\usepackage{graphicx}
\usepackage{enumitem}
\usepackage{mathtools}

\usepackage{hyperref}

\newcommand{\C}{\mathbb C}
\newcommand{\R}{\mathbb R}

\newcommand{\Z}{\mathbb Z}
\newcommand{\N}{\mathbb N}
\newcommand{\F}{\mathbb F}

\newcommand{\HH}{\mathbb H}
\newcommand{\kk}{\mathbf k}
\newcommand{\id}{\operatorname{Id}}

\mathchardef\mhyphen="2D

\newcommand{\op}{\operatorname}
\newcommand{\bs}{\boldsymbol}

\newtheorem{thm}{Theorem}[section]
\newtheorem{lemma}[thm]{Lemma}
\newtheorem{cor}[thm]{Corollary}
\newtheorem{prop}[thm]{Proposition}

\theoremstyle{definition}
\newtheorem{dfn}[thm]{Definition}
\newtheorem{rem}[thm]{Remark}

\mathchardef\mhyphen="2D

\begin{document}

\author[B. Chantraine]{Baptiste Chantraine}
\author[G. Dimitroglou Rizell]{Georgios Dimitroglou Rizell}
\author[P. Ghiggini]{Paolo Ghiggini}

\address{Nantes Universit\'e, CNRS, Laboratoire de Math\'ematiques Jean Leray, LMJL,
F-44000 Nantes, France.}
\email{baptiste.chantraine@univ-nantes.fr}

\address{Uppsala University, Sweden}
\email{georgios.dimitroglou@math.uu.se}

\address{Universit\'e Grenoble Alpes, France.}
\email{paolo.ghiggini@univ-grenoble-alpes.fr}


\title[Representations from closed exact Lagrangians I]{Representations of the Chekanov-Eliashberg algebra from closed exact Lagrangians I}
\maketitle

\begin{abstract}
  This is the first of a series of two articles aiming at relating the compact Fukaya category of a Weinstein manifold to the derived category of finite dimensional representations of the Chekanov-Eliashberg differential graded algebra of the attaching spheres of the critical handles. In this first article we associate a finite dimensional representation $V_L$ to any compact exact Lagrangian submanifold $L$ and prove that for two any such Lagrangian submanifolds $L_0$ and $L_1$ the isomorphism
  $$HF(L_0, L_1) \cong H^*R\hom_{\mathcal A}(V_{L_0}, V_{L_1})$$
  holds. This generalises a previous result of Ekholm and Lekili, but out techniques are different since we use an extension of the Floer theory for Lagrangian cobordisms with negative ends that we developed in collaboration with Roman Golovko.
\end{abstract}
\tableofcontents
\section{Introduction}
We fix the following notation: $W$ denotes a Weinstein domain obtained by attaching critical Weinstein handles to a subcritical Weinstein domain $W^{sc}$ along a link of Legendrian spheres $\mathbf{S} \subset \partial W^{sc}$,  $\F$ is a field of characteristic two, and ${\mathcal A}_{\mathbf{S}}$ denotes the Chekanov-Eliashberg differential graded algebra of $\mathbf{S}$ over the ring $\mathbf{k}_{\mathbf{S}}=H^0(\mathbf{S}; \F)$ with idempotents corresponding to the connected components of $\mathbf{S}$. One can also use fields with arbitrary characteristics, but in that case additional data is needed, such as spin structures, and we do not give any details here.
 This article is the first in a series of two whose final goal it to prove the following statement.

\begin{thm}\label{thm: main of part 2}
  There is a cohomologically full and faithful $A_\infty$-functor from the compact Fukaya category ${\mathcal F}_c(W)$ of $W$, whose objects are closed exact Lagrangian submanifolds, to the dg-category ${\mathcal D}_{\op{dg}}({\mathcal A}_{\mathbf{S}})$ whose objects are left $dg$-modules over ${\mathcal A}_{\mathbf{S}}$ and whose morphisms are
  $$\hom_{{\mathcal D}_{\op{dg}}({\mathcal A_{\mathbf{S}}})}(V_0, V_1)= \op{Rhom}_{{\mathcal A}_{\mathbf{S}}}(V_0, V_1).$$
  Moreover the objects of ${\mathcal D}_{\op{dg}}({\mathcal A}_{\mathbf{S}})$ in the image of this functor have finite rank over $\mathbf{k}_{\mathbf{S}}$.
\end{thm}

In this first article, however, we will content ourselves with the more modest goal of proving the following statement concerning objects and morphism spaces at the cohomological level, leaving the categorical construction at the chain level for the sequel.  We introduce some more notation: we denote by $S_\sigma$ a connected component of $\mathbf{S}$, by $\sigma$ the corresponding idempotent in $\kk_{\mathbf{S}}$, by $D_\sigma$ the cocore of the critical Weinstein handle attached along $S_\sigma$, by $\bullet$ the intersection pairing between homology and relative homology and by $\chi$ the Euler characteristic.

  \begin{thm}\label{thm: main}
  To any closed exact Lagrangian submanifold $L \subset W$ which intersects all cocores of the critical Weinstein handles transversely we associate a differential graded ${\mathcal A}_{\mathbf{S}}$-module $V_L$ such that
\begin{itemize}
   \item $\dim_{\F} (\sigma \cdot V_{L})=|L \cap D_\sigma|$, and
  \item $\chi(\sigma \cdot V_{L})=L \bullet D_\sigma$ when $L$ is oriented,
    
  \end{itemize}
  for all $\sigma \in \pi_0(\mathbf{S})$.
Moreover, given two closed exact Lagrangian submanifolds $L_0$ and $L_1$ as above, the isomorphism
  $$HF^*(L_0, L_1) \cong H^*\op{Rhom}_{{\mathcal A}_{\mathbf{S}}}(V_{L_0}, V_{L_1})$$
  holds. 
\end{thm}

These results are of course not new, since Theorem \ref{thm: main of part 2} could be derived from available technology as follows: we denote by $\mathbf{D}$ the union of the Liouville completions of the Lagrangian cocores of the critical handles in the Liouville completion of $W$, and we consider the functor from ${\mathcal F}_c(W)$ to the category of $A_\infty$ modules over the $A_\infty$ algebra $CW^*(\mathbf{D}, \mathbf{D})$ defined by $L \mapsto CW^*(L, \mathbf{D})$. This functor is cohomologically full and faithful because the completed cocores of the critical handles generate the wrapped Fukaya category of the completion of $W$ by \cite{generation} and \cite{GanatraPardonShende2}. Then the surgery formula from \cite{EffectLegendrian} gives a quasi-isomorphism between $CW^*(\mathbf{D}, \mathbf{D})$ and ${\mathcal A}_{\mathbf{S}}$, and therefore an equivalence between the category of modules over $CW^*(\mathbf{D}, \mathbf{D})$ and that of modules over ${\mathcal A}_{\mathbf{S}}$. Hence this paper is one of those cases when one embarks on a journey not for the sake of reaching the final destination, but to admire the landscape one could meet during the trip. Indeed, in order to prove Theorem \ref{thm: main}, we proved several intermediate results which might have an independent interest: we extend the Floer theory for Lagrangian cobordisms from \cite{Floer_cob} (also known as Cthulhu homology) to immersed exact Lagrangian cobordisms, we prove a relative exact sequence for the Cthulhu homology of a concatenation of cobordisms, we relate the Chekanov-Eliashberg differential graded algebra of a Legendrian sphere with the one of multiple copies of that sphere, and show that in certain cases an augmentation of the parallel copies induces a higher dimensional representation of the original sphere. 

When $L$ intersects each Lagrangian cocore transversely in at most a single point, a quasi-isomorphism between the Floer complex and the linearized Legendrian cohomology complex induced by the representation of the Chekanov-Eliashberg differential graded algebra was shown by \cite{ekholm-lekili}. This hypothesis is satisfied in particular when $L$ is a subset of the skeleton $W^{sk} \subset W$ of the Weinstein domain $W$. It was conjectured in \cite{Flexible} that any exact closed Lagrangian $L \subset (W,d\theta)$ of a Weinstein manifold can be realised as a subset of the skeleton for an appropriate choice of Weinstein structure; this is the so-called regular Lagrangian conjecture.

We currently lack the technology for proving the regularity conjecture, but when some of the geometric  intersection numbers of $L$ with the cocores are greater than one, we can at least make the $L$ coincide with perturbations of core discs inside $W^{crit}$, if we allow regular exact Lagrangian homotopies that introduce double points.  We will use contact topological techniques to define these regular homotopies and to show that they do not change the Floer theoretical properties of the objects involved.

More precisely, the strategy of the proof of Theorem \ref{thm: main} is the following. Given a closed exact Lagrangian $L \subset W$, we produced an immersed exact Lagrangian $\overline{L}$ whose Legendrian lift to the contactisation $W \times \R$ of $W$  is Legendrian isotopic to the  Legendrian lift of $L$, and which decomposes as $\overline{L}=\mathbf{C} \cup \Sigma$, where $\mathbf{C}$ is union of  Lagrangian cores perturbed by a small Hamiltonian diffeomorphism, and $\Sigma$ is an immersed exact Lagrangian filling of the  Legendrian link $\bs{\Lambda}= \partial \mathbf{C}$ consisting of several parallel copies of the attaching spheres of the critical handles. Then we prove that $\Sigma$ induces an augmentation $\varepsilon_\Sigma$ of a differential graded algebra ${\mathcal A}_{\mathbf{C}}$ associated to $\mathbf{C}$ extending the Chekanov-Eliashberg algebra of $\bs{\Lambda}$, and that $\varepsilon_\Sigma$ induces a finite dimensional dg module $V_L$ over the Chekanov-Eliashberg algebra ${\mathcal A}_{\mathbf{S}}$ of the attaching link of the critical handles. The latter result in some sense generalises to higher dimension a relationship between augmentations and higher dimensional representations that was first established by Ng and Rutherford \cite{NgSatellites} in the case of Legendrian knots and their satellites. 

If $L_0$ and $L_1$ are two closed exact Lagrangians and $\overline{L}_i= \mathbf{C}_i \cup \Sigma_i$ as above, then $HF^*(L_0, L_1)$ is isomorphic to $HF^*(\overline{L}_0, \overline{L}_1)$, for an appropriated bounding cochain, and we use a neck-stretching argument to show  that $HF^*(\overline{L}_0, \overline{L}_1)$ is isomorphic to the homology of the Cthulhu complex $\op{Cth}^*_{\varepsilon_{\Sigma_0}, \varepsilon_{\Sigma_1}}(\mathbf{C}_0, \mathbf{C}_1)$. Finally we use Morse-flow trees techniques to show that $\op{Cth}^*_{\varepsilon_{\Sigma_0}, \varepsilon_{\Sigma_1}}(\mathbf{C}_0, \mathbf{C}_1)$ is quasi-isomorphic to a chain complex computing $\op{Rhom}_{\mathcal A}(V_{L_0}, V_{L_1})$.

Finally, under some rather strong assumptions on the quasi-isomorphism class of $\mathcal{A}_{\mathbf{S}}$, we can deduce that the homology class of all exact Lagrangians are primitive and intersect each co-core with intersection number at most one for the given handle-decomposition.
\begin{cor}
  \label{cor:main}
  Assume that the dg-algebra $\mathcal{A}_{\mathbf{S}}$ for some choice of handle-decomposition of $W$ is $\Z$-graded and satisfies the property that the inclusion $\kk_{\mathbf{S}} \subset \mathcal{A}_{\mathbf{S}}$ induces an isomorphism on homology in non-positive degrees, i.e.~$H_{\le 0}(\mathcal{A}_{\mathbf{S}})=\kk_{\mathbf{S}}$. If $L$ is a closed connected exact Lagrangian with vanishing Maslov class, which thus gives rise to a finite-dimensional $\Z$-graded $\mathcal{A}_{\mathbf{S}}$-module $V_L$, then $H(V_L) \neq 0$ is supported in a single degree, and $H(\sigma \cdot V_L)$ is at most one-dimensional for any idempotent $\sigma \in \kk_{\mathbf{S}}$. In particular, the intersection numbers of $L$ and the co-cores satisfy $D_\sigma \bullet L \in \{ \pm 1,0\}$.
  \end{cor}

The article is organised as follows. In Section \ref{sec: idempotents and dga} we describe some algebraic operations on semi-projective differential graded algebras over rings of idempotents that will be used. In Section \ref{sec:short-resolution} we recall a particularly simple resolution of the diagonal biomodule over a semi-projective differential graded algebra and describe how derived morphisms of dg-modules are affected by the algebraic operations introduced in Section \ref{sec: idempotents and dga}. In Section \ref{ChekAlg} we recall the definition of the Chekanov-Eliashberg algebra and justify a frequently used limit procedure for computing them. In Section \ref{sec: immersed cobordisms} we associate a differential graded algebra to an immersed exacty Lagrangian cobordism, and study its functoriality properties. In Section \ref{sec:cthulhu-complex-with} we recall, and slightly generalise, the Cthulhu complex for Lagrangian cobordisms. In Section \ref{sec: relative exact sequence} we prove an exact triangle for the Cthulhu homology of a concatenation of cobordisms which is similar in spirit to the exact sequence for relative homology. In Section \ref{sec:geom-constr} we prove that every closed exact Lagrangian submanifold in a Weinstein manifold $W$ is regular homotopic through immersed exact Lagrangians to the concatenation of an immersed Lagrangian filling in the subcritical part of the Weinstein manifold with an immersed exact Lagrangian cap consisting of parallel copies of cores of the critical handles. In Section \ref{sec: cap algebra} we prove that the differential graded algebra of the cap is quasi-isomorphic to a differential graded algebra obtained from the Chekanov-Eliashberg algebra of the attaching link of the critical handles of $W$ by applying the algebraic operations described in Section \ref{sec: idempotents and dga}. In Section \ref{Cth(C,C)} we show that the Cthulhu complex of a cap is quasi-isomorphic to a resolution described in Section \ref{sec:short-resolution}. Section \ref{sec: morphisms of rep} puts everything together to prove Theorem \ref{thm: main}. Appendix \ref{sec:reduc-cont-stopp} explains how many technical difficulties can be circumvented by stopping the Reeb vector field in the boundary the subcritical part of $W$, at the cost, however, of losing the invariance of the Chekanov-Eliashberg algebra. Appendix \ref{appendix:morse} justifies the use of Morse-flow trees to compute some holomorphic curves in exact Lagrangian cobordisms. Finally, Appendix \ref{appendix: uffa} relates the holomorphic curves with boundary on a Lagrangian cylinder in a symplectisation to some holomorphic curves with boundary on multiple parallel copies of that Lagrangian cylinder with only a minimal use of Morse-Bott techniques (and in particular without Morse-Bott gluing). This appendix provides the counts of holomorphic curves we need in Section \ref{sec: cap algebra}, but we believe that it can be of independent interest.  

\subsection*{Acknowledgments} Chantraine was supported by the ANR COSY grant (ANR-21-CE40-0002), the ANR COSYDY grant (ANR-CE40-0014) and the Labex Centre Henri Lebesgue, ANR-11-LABX-0020-01. Dimitroglou Rizell was supported by the Knut and Alice Wallenberg grants nr KAW 2021.0191, KAW 2021.0300, KAW 2023.0294, and Swedish Research Council through the project grant nr 2020-04426, as well as grant nr 2022-06593: the Centre of Excellence in Geometry and Physics at Uppsala University. Ghiggini was supported by the ANR COSY grant nr ANR-21-CE40-0002, the Knut and Alice Wallenberg guest professorship grant nr KAW 2019.0531, and CNRS through the International Emerging Action “Catégories $A_\infty$ pour les cobordismes symplectiques”. Ghiggini would also like to thank the Mittag-Leffler institute for its hospitality in Autumn 2020, and Uppsala University for his stay during the Winter and Spring of 2021. Finally, all three authours would like to thank Tobias Ekholm for useful conversations.

\section{Differential graded algebras over idempotent rings}\label{sec: idempotents and dga}
In this section and in the next one we present some of the purely algebraic constructions and results that we will need in the remaining sections.  For these two sections we will fix the convention that vector spaces, tensor products, morphism and endomorphism spaces are to be understood over the ground field $\F$ if  they are undecorated. 

Given a finite set $\mathbf{S}$, we define a commutative algebra $\kk_{\mathbf{S}}$ with underlying vector space
$$\kk_{\mathbf{S}}=\bigoplus_{s \in \mathbf{S}} \F s$$
and multiplication induced by $s^2=s$ for every $s \in \mathbf{S}$ and $s \cdot t =0$ if $s, t \in \mathbf{S}$ and $s \ne t$. The unit of $\kk_{\mathbf{S}}$, denoted by $\mathbb{I}_S$, is the sum of all elements of $\mathbf{S}$. We will call $\kk_{\mathbf{S}}$ a {\em ring of idempotents}, or {\em idempotent ring}, over $\mathbf{S}$.

\begin{lemma}\label{tutti proiettivi}
  Every $\kk_{\mathbf S}$-bimodule is projective.
\end{lemma}
\begin{proof}
  We recall that a bimodule is projective if and only if it is a direct summand of a free bimodule. 
  We define linear maps $\iota \colon {\mathcal C} \to \kk_{\mathbf{S}} \otimes {\mathcal C} \otimes \kk_{\mathbf{S}}$ by $\iota(x) = \sum \limits_{s,t \in \mathbf{S}} s\otimes sxt \otimes t$ and $\pi \colon \kk_{\mathbf{S}} \otimes {\mathcal C} \otimes \kk_{\mathbf{S}} \to {\mathcal C}$ by $\pi(s \otimes y \otimes t)= syt$. It is easy to verify that $\iota$ and $\pi$ are bimodule maps; moreover $\pi \circ \iota$ is the identity on ${\mathcal C}$ because for every element $x \in {\mathcal C}$ we have $x = \sum \limits_{s,t \in \mathbf{S}} sxt$ (recall that the sum of all elements of $\mathbf{S}$ is the identity element of $\kk_{\mathbf{S}}$). Then the image of $\iota$ is a summand of the free $\kk_{\mathbf{S}}$-bimodule  $\kk_{\mathbf{S}} \otimes {\mathcal C} \otimes \kk_{\mathbf{S}}$ and is isomorphic to ${\mathcal C}$.
\end{proof}

A {\em pure basis} of a $\kk_{\mathbf{S}}$-bimodule ${\mathcal C}$ is a basis as vector space which has the property that, for every basis element $x$, there are idempotents $s_\pm(x) \in \mathbf{S}$ 
such that $s_-(x) x s_+(x)=x$ and $sxt=0$ if $(s,t) \ne (s_-(x), s_+(x)) \in \mathbf{S}^2$.
\begin{lemma}\label{pure bases exist}
  Every $\kk_{\mathbf{S}}$-bimodule admits a pure basis.\footnote{We prove this for completeness, but in the application the geometry will always provide us with a pure basis} 
\end{lemma}
\begin{proof}
  Let ${\mathcal C}$ be a $\kk_{\mathbf{S}}$-module. In the proof of the previous lemma we showed that
  $${\mathcal C} = \bigoplus_{s,t \in \mathbf{S}} s {\mathcal C} t,$$
  so it is enough to choose a basis of every summand.
\end{proof}
One of the main objects of study in this article will be a differential graded\footnote{The grading is allowed to take values in a cyclic group and will be largely ignored throughout the article.} algebra ${\mathcal A}_{\mathbf{S}}$ over $\kk_{\mathbf{S}}$ (with differential $\partial_{\mathbf{S}}$) whose underlying graded algebra is a tensor algebra of the form
 $${\mathcal A}_{\mathbf{S}}= \bigoplus_{k=0}^\infty \underbrace{{\mathcal C} \otimes_{\kk_{\mathbf{S}}} \cdots \otimes_{\kk_{\mathbf{S}}} {\mathcal C}}_k$$
 where $\mathcal{C}$ is a graded $\kk_{\mathbf{S}}$-bimodule. Moreover, we require that there is an increasing filtration of $\kk_{\mathbf{S}}$-bimodules
 $$\emptyset = {\mathcal F}^0 \mathcal{C} \subset {\mathcal F}^1 \mathcal{C} \subset \cdots \subset
 {\mathcal F}^i\mathcal{C} \subset \ldots \subset \mathcal{C}$$
 which satisfies
 $$\bigcup_{i=0}^{\infty} {\mathcal F}^i \mathcal{C}= \mathcal{C},$$
and such the differential is strictly decreasing, i.e.
$$\partial_{\mathbf{S}}({\mathcal F}^i\mathcal{C}) \subset {\mathcal F}^{i-1}{\mathcal A}_{\mathbf{S}},$$
where
$$ \kk_{\mathbf{S}} = {\mathcal F}^0{\mathcal A}_{\mathbf{S}} \subset {\mathcal F}^1 {\mathcal A}_{\mathbf{S}} \subset \cdots \subset
{\mathcal F}^i {\mathcal A}_{\mathbf{S}} \subset \ldots \subset  {\mathcal A}_{\mathbf{S}}$$
is the induced filtration of tensor algebras. We call differential algebras of this form {\em semi-projective} over $\kk_{\mathbf{S}}$. In the applications the filtration will be induced by a geometrically defined action. 

Let $\lambda \colon \kk_{\mathbf{S}} \to \F$ be the linear function such that $\lambda(s)=1$ for every $s \in \mathbf{S}$. We define an algebra structure on the rank one free $\kk_{\mathbf{S}}$-bimodule $\kk_{\mathbf{S}}^e = \kk_{\mathbf{S}} \otimes \kk_{\mathbf{S}}$ by
$$(s \otimes s') \star (t \otimes t') = \lambda(s't) s \otimes t'$$
for all $s,s',t,t' \in \kk_{\mathbf{S}}$. There is nothing new going on here: with this algebra structure $\kk_{\mathbf{S}}^e$ is isomorphic to $\op{End}(\kk_{\mathbf{S}})$. The operation $\star$ is compatible with the $\kk_{\mathbf{S}}$-bimodule structure on $\kk_{\mathbf{S}}^e$ in the following sense.
\begin{lemma}\label{silentium}
  For every $\mathbf{f}, \mathbf{g} \in \kk_{\mathbf{S}}^e$ and $s \in \mathbf{S}$ we have $(s \mathbf{f}) \star \mathbf{g}= s(\mathbf{f} \star \mathbf{g})$,
  $\mathbf{f} \star (\mathbf{g} s)= (\mathbf{f} \star \mathbf{g})s$, and $(\mathbf{f}s) \star \mathbf{g}= \mathbf{f} \star (s \mathbf{g})$.
\end{lemma}
\begin{proof}
  The first two equalities are trivial. The third one too, but here is the proof. It is enough to verify it on elements of the form $\mathbf{f}=u \otimes u'$ and $v \otimes v'$, where it becomes:
  $$(u \otimes u's) \star (v \otimes v')= \lambda(u'sv) u \otimes v' = (u \otimes u') \star (sv \otimes v').$$ 
\end{proof}
This observation has the following corollary.
\begin{cor} \label{from bimodule morphisms to algebra morphisms}
  Every (graded) bimodule map $\varepsilon \colon {\mathcal C} \to \kk_{\mathbf{S}}^e$ extends to a (graded) algebra morphism $\epsilon \colon {\mathcal A}_{\mathbf{S}} \to
  \kk_{\mathbf{S}}^e$. 
\end{cor}
\begin{proof}
  We define $\varepsilon(x_1 \otimes_{\kk_{\mathbf{S}}} \cdots \otimes_{\kk_{\mathbf{S}}} x_n)= \varepsilon(x_1) \star \cdots \star \varepsilon(x_n)$. This is well defined by Lemma \ref{silentium}.
\end{proof}

 We introduce three algebraic operations on semi-projective differential graded algebras that we call {\em minimal morsification}, {\em omission of the idempotents} and {\em expansion of idempotents}. 

 We start by explaining \emph{minimal morsification}, which is so-called because it loosely corresponds to adding a minimum of a Morse function on a Legendrian to the Chekanov-Eliashberg algebra.  We extend ${\mathcal A}_{\mathbf{S}}$ to a differential graded algebra ${\mathcal A}^+_{\mathbf{S}}$ as follows. We define the $\kk_{\mathbf{S}}$-bimodule ${\mathcal C}^+= \kk_{\mathbf{S}} \oplus {\mathcal C}$ and the algebra
   $${\mathcal A}_{\mathbf{S}}^+= \bigoplus_{k=0}^\infty \underbrace{{\mathcal C}^+ \otimes_{\kk_{\mathbf{S}}} \cdots \otimes_{\kk_{\mathbf{S}}} {\mathcal C}^+}_k.$$
   Elements $s \in \mathbf{S}$, when viewed as generators of the $\kk_{\mathbf{S}}$ summand of ${\mathcal C}^+$, will be denoted by $e_s$. We will also denote $e = \sum \limits _{s \in \mathbf{S}} e_s$, and with this notation in hand, we  define a differential $\partial_{\mathbf{S}}^+$ on ${\mathcal A}_{\mathbf{S}}^+$  by
   \begin{equation}\label{differential plus}
     \begin{cases}
        \partial_{\mathbf{S}}^+e_s=e_s^2 & \text{if } s \in {\mathbf{S},} \text{ and} \\
       \partial_{\mathbf{S}}^+ x = \partial_{\mathbf{S}} x + ex + xe & \text{ if } x \in {\mathcal C}. \\
      \end{cases}
   \end{equation}
    The dga $\mathcal{A}^+_{\mathbf{S}}$ is never semi-projective. In addition, note that ${\mathcal A}_{\mathbf{S}}$ sits inside ${\mathcal A}_{\mathbf{S}}^+$ as a sub-algebra, but not as a sub-complex, and therefore is not a differential graded sub-algebra. There is however a surgective dga morphism ${\mathcal A}_{\mathbf{S}}^+ \to {\mathcal A}_{\mathbf{S}}$ mapping $e_s$ to zero for every $s \in S$. 
  
   \begin{lemma}\label{the meaning of e}
   
     Giving a graded vector space $V$ the structure of a dg-module over ${\mathcal A}_{\mathbf{S}}$ is equivalent to producing an augmentation  $\rho \colon {\mathcal A}_{\mathbf{S}}^+ \to \op{End}(V)$ (i.e.\ a morphism of differential graded algebras 
     where $\op{End}(V)$ has trivial differential).
 \end{lemma}
 \begin{proof}
   We need to produce a differential $d$ on $V$ and an action of ${\mathcal A}_{\mathbf{S}}$ satisfying the Leibniz rule. For the differential we define $d(v)= \rho(e)v$ for every $v \in V$ and for the action we define $x \cdot v = \rho(x)v$ for every $x \in {\mathcal A}_{\mathbf S}$. The required relations are a consequence of Equation \eqref{differential plus} and $\rho(\partial_{\mathbf{S}}^+ x)=0$ for every $x \in {\mathcal A}_{\mathbf{S}}$.
 \end{proof}

 We proceed by describing the \emph{omission of the idempotents}.
 Let ${\mathcal A}_{\slashed{\mathbf{S}}}$ the differential graded algebra which, as an algebra is
 $${\mathcal A}_{\slashed{\mathbf{S}}} = \bigoplus_{k = 0}^\infty \underbrace{{\mathcal C} \otimes \cdots \otimes {\mathcal C}}_k$$
 and whose differential is induced by the differential of ${\mathcal A}_{\mathbf{S}}$. More precisely, ${\mathcal A}_{\mathbf{S}}$ is generated, as a vector space, by $s \in \mathbf{S}$ and by composable words of length at least one in the elements of a pure basis of ${\mathcal C}$, i.e. words $x_1 \cdots x_n$ where each $x_i$ is a basis element and $s_+(x_i) = s_-(x_{i+1})$.
 On the other hand, ${\mathcal A}_{\slashed{\mathbf{S}}}$ is generated as a vector space by $1$ and by words of length at least one in the basis element, without further restrictions. So we can define a linear map $i \colon {\mathcal A}_{\mathbf{S}} \to {\mathcal A}_{\slashed{\mathbf{S}}}$ by $i(s) =1$ for all $s \in {\mathbf{S}}$ and $i(x_1 \cdots x_n)= x_1 \cdots x_n$ for every composable word. (Note that this is not a map of algebras!) With this notation at hand, we define $\partial_{\slashed{\mathbf{S}}}x = i(\partial_{\mathbf{S}}x)$ for every $x \in {\mathcal C}$ and extend it to a derivation of ${\mathcal A}_{\slashed{\mathbf{S}}}$ via the Leibniz rule.
 \begin{lemma}
   The map $\partial_{\slashed{\mathbf{S}}}$ satisfies $\partial_{\slashed{\mathbf{S}}}^2=0$.
 \end{lemma}
 \begin{proof}
    We observe that ${\mathcal A}_{\mathbf{S}}$ decomposes as chain complex as ${\mathcal A}_{\mathbf{S}} = \bigoplus \limits_{s,t \in \mathbf{S}} s {\mathcal A}_{\mathbf{S}} t$, and moreover $i$ is injective when restricted to each summand $s {\mathcal A}_{\mathbf{S}} t$. Thus $\partial_{\mathbf{S}}^2 x =0$ for every $x$ in a pure basis of ${\mathcal C}$ implies that $\partial_{\slashed{\mathbf{S}}}^2 x =0$.
  \end{proof}

  We consider the differential graded algebra ${\mathcal A}_{\slashed{\mathbf{S}}}^{\mathit{ext}} =  {\mathcal A}_{\slashed{\mathbf{S}}} \otimes \kk_{\mathbf{S}}^e$ where the multiplication is component by component and the differential on $\kk_{\mathbf{S}}^e$ is trivial. Equivalently, we can regard ${\mathcal A}_{\slashed{\mathbf{S}}}^{\mathit{ext}}$ as the vector space generated by words $s \otimes x_1 \cdots x_n \otimes t$ where $s,t \in \mathbf{S}$ and $x_i$ are element of a pure basis of ${\mathcal C}$; then the multiplication is defined by
  $$(s \otimes x_1 \cdots x_n \otimes t)(s' \otimes y_1 \cdots y_m \otimes t')= \lambda(ts') s\otimes x_1 \cdots x_n y_1 \cdots y_m \otimes t'$$
  and the differential by $\partial_{\slashed{\mathbf{S}}}^{\textit{ext}}(s\otimes x \otimes t)= s \otimes \partial_{\slashed{\mathbf{S}}} x \otimes t$. The following lemma is immediate from the definitions. 

  \begin{lemma}
    The map $j \colon {\mathcal A}_{\mathbf{S}} \to {\mathcal A}_{\slashed{\mathbf{S}}}^{\mathit{ext}}$ which is defined on a pure basis of ${\mathcal C}$ by $j(x)= s_-(x) \otimes x \otimes s_+(x)$ is an inclusion of differential graded algebras.
  \end{lemma}
  
  \begin{lemma} \label{from slashed to normal}
   Every augmentation $\slashed{\varepsilon} \colon {\mathcal A}_{\slashed{\mathbf{S}}} \to \F$ induces an augmentation $\varepsilon \colon {\mathcal A}_{\mathbf{S}} \to \kk_{\mathbf{S}}^e$ which, on ${\mathcal C}$, is defined by
   \begin{equation} \label{lyon}
     \varepsilon(x) = \sum \limits_{s_-, s_+ \in \mathbf{S}} \slashed{\varepsilon}(s_- x s_+) s_- \otimes s_+.
   \end{equation}
 \end{lemma}
 \begin{proof}
We extend $\slashed{\varepsilon}$ to an augmentation  $\slashed{\varepsilon}^{\mathit{ext}}= \slashed{\varepsilon} \otimes \op{Id} \colon {\mathcal A}_{\slashed{\mathbf{S}}}^{\mathit{ext}} \to \kk^e_{\mathbf{S}}$. Then $\varepsilon =  \slashed{\varepsilon}^{\mathit{ext}} \circ j$.
 \end{proof}

 Finally we describe the \emph{expansion of idempotents}.
 If $\mathbf{C}$ is a finite set and $i \colon \mathbf{C} \to \mathbf{S}$ is a map, there is a non-unital algebra map $i^* \colon \kk_{\mathbf{S}} \to \kk_{\mathbf{C}}$ which is defined by
 $$i^*(s)= \sum \limits_{c \in i^{-1}(s)} c$$
 for every $s \in \mathbf{S}$. Here non-unital means ``not necessarily unital''; unitality fails when $i$ is not surjective. 
 We define the $\kk_{\mathbf{C}}$-bimodule
 $${\mathcal C}_{\mathbf{C}}= \kk_{\mathbf{C}} \otimes_{\kk_{\mathbf{S}}} {\mathcal C} \otimes_{\kk_{\mathbf{S}}} \kk_{\mathbf{C}}$$
and the $\kk_{\mathbf{C}}$-algebra 
$${\mathcal A}_{\mathbf{C}} = \bigoplus_{k=0}^\infty \underbrace{{\mathcal C}_{\mathbf{C}} \otimes_{\kk_{\mathbf{C}}} \cdots \otimes_{\kk_{\mathbf{C}}} {\mathcal C}_{\mathbf{C}}}_k.$$

Of course ${\mathcal A}_{\mathbf{C}}$ is {\em a fortiori} a $\kk_{\mathbf{S}}$-algebra and we define a $\kk_{\mathbf{S}}$-algebra morphism $T \colon {\mathcal A}_{\mathbf{S}} \to {\mathcal A}_{\mathbf{C}}$ by extending the $\kk_{\mathbf{S}}$-bimodule map
$${\mathcal C} \to {\mathcal C}_{\mathbf{C}}, \quad x \mapsto \mathbb{I}_{\mathbf{C}} \otimes_{\kk_{\mathbf{S}}} x \otimes_{\kk_{\mathbf{S}}} \mathbb{I}_{\mathbf{C}}.$$

We define a derivation  $\partial_{\mathbf{C}}$ on ${\mathcal A}_{\mathbf{C}}$ by extending
\begin{equation} \label{pizza}
  \partial_{\mathbf{C}}(c^- \otimes_{\kk_{\mathbf{S}}} x \otimes_{\kk_{\mathbf{S}}} c^+) = c^- \cdot T(\partial_{\mathbf{S}} x) \cdot c^+
\end{equation}
on every $x \in {\mathcal C}$.

\begin{lemma}
  The map $\partial_{\mathbf{C}}$ is a differential which makes $T$ a (non-unital) morphism of differential graded algebras.
\end{lemma}
\begin{proof}
 By Equation \eqref{pizza} one can see that $\partial_{\mathbf{C}} \circ T = T \circ \partial_{\mathbf{S}}$, and therefore $\partial_{\mathbf C}^2(\mathbb{I}_{\mathbf C} \otimes_{\kk_{\mathbf{S}}} x \otimes_{\kk_{\mathbf{S}}} \mathbb{I}_{\mathbf C})=0$. Then $\partial_{\mathbf C}^2(c_- \otimes_{\kk_{\mathbf{S}}} x \otimes_{\kk_{\mathbf{S}}} \otimes c_+)=0$ because $\partial_{\mathbf{C}}$ is $\kk_{\mathbf{C}}$-linear.
\end{proof}
\begin{cor} \label{from aug to rep 2}
 An augmentation $\varepsilon \colon {\mathcal A}_{\mathbf{C}} \to \kk_{\mathbf{C}}^e$ induces a differential algebra morphism $\rho \colon {\mathcal A}_{\mathbf{S}} \to \kk_{\mathbf{C}}^e$ by $\rho = \varepsilon \circ T$. 
\end{cor} 

Let ${\mathcal A}^+_{\slashed{\mathbf{C}}}$ be the differential graded algebra obtained from ${\mathcal A}_{\mathbf{S}}$ and a map $i \colon \mathbf{C} \to \mathbf{S}$ by the following operations:
\begin{enumerate}
  \item a minimal morsification to ${\mathcal A}_{\mathbf{S}}$ to obtain ${\mathcal A}_{\mathbf{S}}^+$,
  \item an expansion of idempotents to ${\mathcal A}_{\mathbf{S}}^+$ to obtain ${\mathcal A}^+_{{\mathbf{C}}}$, and finally
  \item an omission of idempotents to ${\mathcal A}^+_{{\mathbf{C}}}$ to obtain ${\mathcal A}^+_{\slashed{\mathbf{C}}}$.
  \end{enumerate}
  Note that the operations are not commutative, and should be taken in this order, even if the order is not recorded in the notation.
\begin{cor}\label{the point of all this mess}
  An augmentation $\slashed{\varepsilon} \colon {\mathcal A}^+_{\slashed{\mathbf{C}}} \to \F$ induces a  dg-module structure over ${\mathcal A}_{\mathbf{S}}$ on $\kk_{\mathbf{C}}$.
\end{cor} 
\begin{proof}
 The  augmentation $\slashed{\varepsilon}$ induces an augmentation $\varepsilon \colon {\mathcal A}^+_{{\mathbf{C}}} \to \op{End}(\kk_{\mathbf{C}})$ by Lemma 2.6. The augmentation $\varepsilon$ induces an augmentation $\rho \colon {\mathcal A}^+_{\mathbf{S}} \to \op{End}(\kk_{\mathbf{C}})$ by Corollary \ref{from aug to rep 2}. Finally $\rho$ induces a dg-module structure over ${\mathcal A}_{\mathbf{S}}$ on $\kk_{\mathbf{C}}$ by Lemma \ref{the meaning of e}.
\end{proof}

  Now we introduce a total ordering of $i^{-1}(s)$ for every $s \in \mathbf{S}$.  This may look as an unnatrural choice from the algebraic point of view, but it is motivated by the geometry. The elements in $i^{-1}(s)$ will be denoted by $s_1, \ldots, s_{k_s}$ according to their ordering. We denote by ${\mathcal I}_{\mathbf{C}}$ the bilateral ideal of ${\mathcal A}^+_{{\mathbf{C}}}$ generated by elements $e_{s_i} \otimes_{\kk_{\mathbf{S}}} e_{s_j}$ with $j \ge i$, and by ${\mathcal I}_{\slashed{\mathbf{C}}}$
  the similarly defined ideal of ${\mathcal A}^+_{\slashed{\mathbf{C}}}$.
  \begin{lemma} \label{ho finito le labelle}
    ${\mathcal I}_{\mathbf{C}}$ and ${\mathcal I}_{\slashed{\mathbf{C}}}$ are differential ideals.
\end{lemma}
\begin{proof}
  Since the differential on ${\mathcal A}^+_{\slashed{\mathbf{C}}}$ is induced by the differential on ${\mathcal A}^+_{{\mathbf{C}}}$ it is enough to prove the lemma for ${\mathcal I}_{\mathbf{C}}$.
  For simplicity of notation we write $e_{s_i} \otimes_{\kk_{\mathbf{S}}} e_{s_j} = e_{s_i}^{s_j}$.
  We need to show that $\partial^+_{\mathbf{C}}(e_{s_i}^{s_j}) \in {\mathcal I}_{\mathbf{C}}$ whenever $e_{s_i}^{s_j}\in {\mathcal I}_{\mathbf{C}}$. By the definition of $\partial^+_{\mathbf{C}}$, we see that $\partial^+_{\mathbf{C}}(e_{s_i}^{s_j})$ is a sum of elements of the form
  $e_{s_{i_0}}^{s_{i_1}} \otimes_{\kk_{\mathbf{C}}}e_{s_{i_1}}^{i_2} \otimes_{\kk_{\mathbf{C}}} \cdots  \otimes_{\kk_{\mathbf{C}}} e_{s_{\ell-1}}^{s_\ell} \otimes_{\kk_{\mathbf{C}}} e_{s_{\ell}}^{s_{\ell+1}}$ with $i_0=i$ and $i_{l+1}=j$. Since $i \ge i$, the sequence $i_0, \ldots, i_{\ell+1}$ cannot be strictly increasing, and therefore there must be some $i$ for which $e_{s_i}^{s_{i+1}} \in {\mathcal I}_{\mathbf{C}}$.
\end{proof}
Note that ${\mathcal A}^+_{\slashed{\mathbf{C}}} / {\mathcal I}_{\slashed{\mathbf{C}}}$ can be obtained also by omitting the idempotents in ${\mathcal A}^+_{\mathbf{C}} / {\mathcal I}_{\mathbf{C}}$; i.e.\ quotient and omission of idempotents commute. An important feature of these quotients is that they are semi-free. Putting all these constructions together, we obtain the following corollary.
\begin{cor}\label{the point of all this mess 2}
  An augmentation $\widetilde{\slashed{\varepsilon}} \colon {\mathcal A}^+_{\slashed{\mathbf{C}}}/ {\mathcal I}_{\slashed{\mathbf{C}}} \to \F$ induces a  dg-module structure over ${\mathcal A}_{\mathbf{S}}$ on $\kk_{\mathbf{C}}$.
\end{cor} 
\begin{proof}
  The augmentation $\widetilde{\slashed{\varepsilon}}$ induces an augmentation $\slashed{\varepsilon} \colon {\mathcal A}^+_{\slashed{\mathbf{C}}} \to \F$ by composition with the quotient map, and then we apply Corollary \ref{the point of all this mess}.
\end{proof}

It is pehaps the time that we explained the reason why the reader had to endure all this abstract nonsense. In the application, ${\mathcal A}_{\kk_{\mathbf{S}}}$ will be the Chekanov-Eliashberg algebra of the attaching link of the critical handles of $W$ and $\mathbf{S}$ will be the set of connected components of the link. Every compact exact Lagrangian in $W$ will produce, by a geometric construction, a Legendrian link $\bs{\Lambda}$ such that every connected component of $\bs{\Lambda}$ is a Reeb pushoff of a connected component of the attaching link, an immersed filling $\Sigma$ and an immersed cap $\mathbf{C}$, both for $\bs{\Lambda}$. Then $\kk_{\mathbf{C}}$ will be the ring with idempotents corresponding to the connected components of $\bs{\Lambda}$ and ${\mathcal A}^+_{\kk_\mathbf{C}}/ {\mathcal I}_{\mathbf{C}}$ will be quasi-isomorphic to a Chekanov-Eliashberg-type dga associated to the cap $\mathbf{C}$ which we call, without much imagination, the cap algebra. The filling $\Sigma$ will produce an augmentation of the cap algebra which, via Corollary \ref{the point of all this mess}, will produce a dg-module over the Chekanov-Eliashberg algebra of the attaching link.

\section{The short resolution}
\label{sec:short-resolution}
Let ${\mathcal A}_{\mathbf{S}}$ be a differential graded algebra over an idempotent ring $\kk_{\mathbf{S}}$. A \emph{semi-pojective} dg-bimodule over ${\mathcal A}_{\mathbf{S}}$ is a dg-bimodule ${\mathcal P}$ over ${\mathcal A}_{\mathbf{S}}$ together with a filtration by dg-submodules so that the associated graded dg-module is made of copies of ${\mathcal A}_{\mathbf{S}} \otimes_{\kk_{\mathbf{S}}} M \otimes_{\kk_{\mathbf{S}}} {\mathcal A}_{\mathbf{S}}$ for a $k_{\mathbf{S}}$ bimodule $M$ with the trivial differential induced by the ${\mathcal A}_{\mathbf{S}}$-factors in the tensor product (in \cite[Appendix C8]{DrinfeldDG} such are called semi-free). If ${\mathcal B}$ is a bimodule over ${\mathcal A}_{\mathbf{S}}$, a semi-projective resolution of ${\mathcal B}$ is a semi-projective bimodule ${\mathcal P}$ together with a bimodule morphism ${\mathcal P} \to {\mathcal B}$ inducing an isomorphism in homology.

In this section we describe a particularly simple semi-projective resolution of the diagonal bimodule which exists when ${\mathcal A}_{\mathbf{S}}$ is a semi-projective algebra, and  that will allows us to compute $R\hom_{{\mathcal A}_{\mathbf{S}}}(V_0, V_1)$ given two dg-modules $V_0$ and $V_1$ over ${\mathcal A}_{\mathbf{S}}$. Such resolution was previously considered by Keller \cite[Proposition 3.7]{Keller:Deformed} and Legout \cite{Legout:Calabi}. We will call it the {\em short resolution} and denote it by ${\mathcal S}_{\mathcal A_{\mathbf{S}}}$.

From now on we assume that ${\mathcal A}_{\mathbf{S}}$ is semi-projective and is generated by a $\kk_{\mathbf{S}}$-bimodule ${\mathcal C}$.
We form the ${\mathcal A}_{\mathbf{S}}$-bimodule ${\mathcal A}_{\mathbf{S}} \otimes_{\kk_{\mathbf{S}}}
  {\mathcal C} \otimes_{\kk_{\mathbf{S}}} {\mathcal A}_{\mathbf{S}}$ and for every element
$c \in {\mathcal C}$ we denote $\hat{c}= \mathbb{I}_{\mathbf{S}} \otimes_{\kk_{\mathbf{S}}} c \otimes_{\kk_{\mathbf{S}}} \mathbb{I}_{\mathbf{S}}$ and, consequently, $\mathbf{a} \hat{c} \mathbf{b}= \mathbf{a} \otimes_{\kk_{\mathbf{S}}} c \otimes_{\kk_{\mathbf{S}}} \mathbf{b}$.
  Let $\Delta \colon {\mathcal A}_{\mathbf{S}} \to {\mathcal A}_{\mathbf{S}} \otimes_{\kk_{\mathbf{S}}} {\mathcal C} \otimes_{\kk_{\mathbf{S}}} {\mathcal A}_{\mathbf{S}}$ be the linear map which is uniquely determined by the following two properties:
\begin{itemize}
\item  $\Delta(c)= \hat{c}$ for every $c \in {\mathcal C}$, and
\item $\Delta(\mathbf{ab})= \mathbf{a} \Delta (\mathbf{b})+ \Delta(\mathbf{a}) \mathbf{b}$ (i.e. $\Delta$ is a derivation). 
\end{itemize}
As a shorthand notation, we will write $\hat{\mathbf{a}}= \Delta(\mathbf{a})$ for all $\mathbf{a} \in {\mathcal A}$. We observe that $\hat{s} = 0$ for all $s \in \mathbf{S}$.
We define a $\kk_{\mathbf{S}}$-linear map $D$ on ${\mathcal A}_{\mathbf{S}} \otimes_{\kk_{\mathbf{S}}} {\mathcal C} \otimes_{\kk_{\mathbf{S}}} {\mathcal A}_{\mathbf{S}}$ by
$$D(\mathbf{a} \hat{c} \mathbf{b}) = (\partial_{\mathbf{S}} \mathbf{a}) \hat{c} 
\mathbf{b} + \mathbf{a} (\widehat{\partial_{\mathbf{S}}c}) \mathbf{b} + \mathbf{a} \hat{c} (\partial_{\mathbf{S}} \mathbf{b}).$$
It is easy to see that $D$ is a differential which makes ${\mathcal A}_{\mathbf{S}} \otimes_{\kk_{\mathbf{S}}} {\mathcal C} \otimes_{\kk_{\mathbf{S}}} {\mathcal A}_{\mathbf{S}}$ a dg-bimodule over ${\mathcal A}_{\mathbf{S}}$.

Next we consider the dg-bimodule ${\mathcal A}_{\mathbf{S}} \otimes_{\kk_{\mathbf{S}}} {\mathcal A}_{\mathbf{S}}$ with the usual K\"unneth differential and define  dg--bimodule  morphisms
$$\mu \colon {\mathcal A}_{\mathbf{S}} \otimes_{\kk_{\mathbf{S}}} {\mathcal A}_{\mathbf{S}} \to {\mathcal A}_{\mathbf{S}}$$
by $\mu(\mathbf{a} \otimes_{\kk_{\mathbf{S}}} \mathbf{b}) = \mathbf{a} \mathbf{b}$ and
$$\iota \colon {\mathcal A}_{\mathbf{S}} \otimes_{\kk_{\mathbf{S}}} {\mathcal C} \otimes_{\kk_{\mathbf{S}}} {\mathcal A}_{\mathbf{S}} \to {\mathcal A}_{\mathbf{S}} \otimes_{\kk_{\mathbf{S}}} {\mathcal A}_{\mathbf{S}}$$
by $\iota(\mathbf{a} \otimes_{\kk_{\mathbf{S}}} x \otimes_{\kk_{\mathbf{S}}} \mathbf{b})=
\mathbf{a} x \otimes_{\kk_{\mathbf{S}}} \mathbf{b} + \mathbf{a} \otimes_{\kk_{\mathbf{S}}} x  \mathbf{b}$ (i.e.\ $\iota(\mathbf{a} \hat{x} \mathbf{b}) = \mathbf{a} x \otimes_{\kk_{\mathbf{S}}} \mathbf{b} + \mathbf{a} \otimes_{\kk_{\mathbf{S}}} x  \mathbf{b}$).
The only nontrivial verification is that $\iota$ is a chain map, which follows from the relation $\iota(\Delta(\mathbf{x}))= \mathbf{x} \otimes_{\kk_{\mathbf{S}}} \mathbb{I}_{\mathbf{S}} + \mathbb{I}_{\mathbf{S}}  \otimes_{\kk_{\mathbf{S}}} \mathbf{x}$ for every $\mathbf{x} \in 
{\mathcal A}_{\mathbf{S}}$, which implies $\iota(\mathbf{a} (\widehat{\partial_{\mathbf{S}} c}) \mathbf{b}) = \mathbf{a} (\partial_{\mathbf{S}} c) \otimes_{\kk_{\mathbf{S}}} \mathbf{b}+  \mathbf{a} \otimes_{\kk_{\mathbf{S}}} (\partial_{\mathbf{S}} c)  \mathbf{b}$.

\begin{lemma}\label{toward the short resolution}
The sequence 
$$0 \to {\mathcal A}_{\mathbf{S}} \otimes_{\kk_{\mathbf{S}}} {\mathcal C} \otimes_{\kk_{\mathbf{S}}} {\mathcal A}_{\mathbf{S}} \xrightarrow{\iota} {\mathcal A}_{\mathbf{S}} \otimes_{\kk_{\mathbf{S}}} {\mathcal A}_{\mathbf{S}} \xrightarrow{\mu} {\mathcal A}_{\mathbf{S}} \to 0$$ is exact.
\end{lemma}
\begin{proof}
  It is evident that $\mu$ is surjective and that $\mu \circ \iota =0$, which implies that $\op{Im}(\iota) \subset \ker(\mu)$, so it remains to prove that $\mu$ is injective and $\ker(\mu) \subset \op{Im}(\iota)$. To prove that $\ker(\mu) \subset \op{Im}(\iota)$ we observe that $\iota(\widehat{\mathbf{a}} \mathbf{b})= \mathbf{a} \otimes_{\kk_{\mathbf{S}}} \mathbf{b} + \mathbb{I}_{\mathbf{S}} \otimes_{\kk_{\mathbf{S}}} \mathbf{ab}$ for every $\mathbf{a}, \mathbf{b} \in {\mathcal A}_{\mathbf{S}}$, and therefore for every $A \in \ker (\mu)$ there exists $B \in \op{Im}(\iota)$ such that $A+B= \mathbb{I}_{\mathbf{S}} \otimes_{\kk_{\mathbf{S}}} \mathbf{b}$ for some $\mathbf{b} \in {\mathcal A}_{\mathbf{S}}$. Moreover $\mu(\mathbb{I}_{\mathbf{S}} \otimes_{\kk_{\mathbf{S}}} \mathbf{b})= 0$, which implies $\mathbf{b}=0$, and therefore $\ker (\mu) \subset \op{Im}(\iota)$.

  Now we consider an element $A \in {\mathcal A}_{\mathbf{S}} \otimes_{\kk_{\mathbf{S}}} \mathcal{C}  \otimes_{\kk_{\mathbf{S}}} {\mathcal A}_{\mathbf{S}}$  such that $\iota(A)=0$ and write it (uniquely) as a linear combination $A=\sum_i \mathbf{a}_i \hat{c}_i \mathbf{b}_i$ where $\mathbf{a}_i$ and $\mathbf{b}_j$ are words in a pure basis of ${\mathcal C}$ and $c_i$ is an element of the same basis. Recall that $\mathcal{A}_{\mathbf{S}}$ is endowed with the word-length filtration from the tensor product. In the sum $A=\sum_i \mathbf{a}_i \hat{c}_i \mathbf{b}_i$ we consider the terms whose factors $\mathbf{a}_i$ have maximal length and index those by $j$. The terms $\mathbf{a}_jc_j \otimes_{\kk_{\mathbf{S}}} \mathbf{b}_j$ are the ones of maximal length on the left in $\iota(A)$, and therefore $\sum_j \mathbf{a}_jc_j \otimes_{\kk_{\mathbf{S}}} \mathbf{b}_j=0$ because they cannot be cancelled by any other term. Now we regroup the sum putting together the terms with the same $\mathbf{b}_j$, and get
  $$\sum_{k,l} \mathbf{a}_{k,l} c_{k,l} \otimes_{\kk_{\mathbf{S}}} \mathbf{b}_k =0.$$ 
Note that this is just a relabelling of the $\mathbf{a}_j$, $\mathbf{b}_j$ and $c_j$, but the term of this sum are the same as the terms of the previous one. Thus the sum should vanish term by term because the $\mathbf{b}_j$ are elements of a basis of ${\mathcal A}_{\mathbf{S}}$, and therefore  $\sum_l  \mathbf{a}_{k,l} c_{k,l}=0$ for every $k$.  But the  $\mathbf{a}_{k,l} c_{k,l}$ are also elements of a basis of ${\mathcal A}_{\mathbf{S}}$, and this is a contradiction. This proves that $\iota$ is injective and therefore the lemma.
\end{proof}
We define ${\mathcal S} = \op{Cone}(\iota)$ and the dg-bimodue morphism $m \colon {\mathcal S} \to {\mathcal A}_{\mathbf{S}}$
as $\mu$ on the ${\mathcal A}_{\mathbf{S}} \otimes_{\kk_{\mathbf{S}}} {\mathcal A}_{\mathbf{S}}$ summand and $0$ on the ${\mathcal A}_{\mathbf{S}} \otimes_{\kk_{\mathbf{S}}} {\mathcal C} \otimes_{\kk_{\mathbf{S}}} {\mathcal A}_{\mathbf{S}}$ summand. We will denote the differential on ${\mathcal S}_{{\mathcal A}_{\mathbf{S}}}$ by $\mathfrak{d}$.

\begin{lemma} \label{the short resolution at last}
  ${\mathcal S} \xrightarrow{m} {\mathcal A}_{\mathbf{S}}$ is a
    semi-pojective resolution of ${\mathcal A}_{\mathbf{S}}$.
\end{lemma}
\begin{proof}
The fact that $\mu$ induces a quasi-isomorphism follows from Lemma \ref{toward the short resolution} and a simple diagram chasing argument.

To see that ${\mathcal S}$ is semi-projective, observe that the action filtration $F_0\subset \ldots \subset F_p\subset \ldots \subset \mathcal{C}$ induces a filtration ${\mathcal F}_0 \subset \ldots \subset {\mathcal F}_p \subset \ldots \subset {\mathcal S}$ where ${\mathcal F}_p= {\mathcal A}_{\mathbf{S}} \otimes_{\kk_{\mathbf{S}}} F_p  \otimes_{\kk_{\mathbf{S}}} {\mathcal A}_{\mathbf{S}}$ and the differential of ${\mathcal S}$ preserves this filtration. The quotients ${\mathcal F}_p / {\mathcal F}_{p-1}  = {\mathcal A}_{\mathbf{S}} \otimes_{\kk_{\mathbf{S}}} F_p /F_{p-1} \otimes_{\kk_{\mathbf{S}}} {\mathcal A}_{\mathbf{S}}$ are isomorphic, as dg-bimodles, to direct summands of the free dg-bimodules  ${\mathcal A}_{\mathbf{S}} \otimes F_p /F_{p-1} \otimes {\mathcal A}_{\mathbf{S}}$  by Lemma \ref{tutti proiettivi}.
  \end{proof}

It follows then from \cite[Appendix C]{DrinfeldDG} that given two right dg-modules $V_0$ and $V_1$ over ${\mathcal A}_{\mathbf{S}}$,  we can compute the derive morphisms as $R\hom_{{\mathcal A}_{\mathbf{S}}}(V_0, V_1)= \hom_{{\mathcal A}_{\mathbf{S}}}(V_0\otimes_{{\mathcal A}_{\mathbf{S}}} \mathcal{S}, V_1)$, and analogously for left modules.

If ${\mathcal A}_{\mathbf{C}_0}$ and ${\mathcal A}_{\mathbf{C}_1}$ are differential graded algebras which are obtained from ${\mathcal A}_{\mathbf{S}}$ by expansion of idempotents, $\slashed{\varepsilon}_0 \colon {\mathcal A}_{\slashed{\mathbf{C}}_0} \to \F$ and $\slashed{\varepsilon}_1 \colon {\mathcal A}_{\slashed{\mathbf{C}}_1} \to \F$ are augmentations, and $V_0$ and $V_1$ are the induced dg-modules over ${\mathcal A}_{\mathbf{S}}$, we want to compute $R\hom_{{\mathcal A}_{\mathbf{S}}}(V_0, V_1)$ in terms of $\slashed{\varepsilon}_0$ and $\slashed{\varepsilon}_1$. To this aim, we introduce several bimodules. First, we  recall that ${\mathcal A}_{\mathbf{S}}$ is a tensor algebra over a $\kk_{\mathbf{S}}$-bimodule ${\mathcal C}$ and
${\mathcal A}^+_{\mathbf{S}}$ is a tensor algebra over the bimodule ${\mathcal C}^+ = {\mathcal C} \oplus \kk_{\mathbf{S}}$, where we called $e$ the generator of the $\kk_{\mathbf{S}}$ summand.

We define a dg-bimodue
$${\mathcal S}^+ = {\mathcal A}^+_{\mathbf{S}} \otimes_{\kk_{\mathbf{S}}} {\mathcal C}^+ \otimes_{\kk_{\mathbf{S}}}{\mathcal A}^+_{\mathbf{S}}$$
over ${\mathcal A}^+_{\mathbf{S}}$. We denote an element $\mathbf{a} \otimes_{\kk_{\mathbf{S}}} x \otimes_{\kk_{\mathbf{S}}} \mathbf{b}$ by $\mathbf{a} \hat{x} \mathbf{b}$ for brevity. The differential on ${\mathcal S}^+$ is defined by $\mathfrak{d}^+ \hat{x} = \widehat{\partial_{\mathbf{S}}^+ x}.$ This formula can be rewritten slightly more explicitly as  $\mathfrak{d}^+ \hat{x} = \widehat{\partial_{\mathbf{S}} x}+ e \hat{x} + \hat{x} e  + \hat{e} x + x \hat{e}.$

To a dg-module $(V, d)$ over ${\mathcal A}_{\mathbf{S}}$ we associate a module $V^+$ (with trivial differential) over ${\mathcal A}^+_{\mathbf{S}}$ by defining $V^+ =V$ as a $\kk_{\mathbf{S}}$-module and defining the action of ${\mathcal A}^+_{\mathbf{S}}$ by extending the action of ${\mathcal A}_{\mathbf{S}}$ by $e \cdot v = d(v)$ for all $v \in V$.
\begin{lemma}\label{+ and d}
  If $V_0$ and $V_1$ are dg-modules over ${\mathcal A}_{\mathbf{S}}$, then $$\hom_{{\mathcal A}_{\mathbf{S}}}({\mathcal S} \otimes_{{\mathcal A}_{\mathbf{S}}} V_0, V_1) \cong \hom_{{\mathcal A}^+_{\mathbf{S}}}({\mathcal S}^+ \otimes_{{\mathcal A}^+_{\mathbf{S}}} V_0^+, V_1^+)$$
  as chain complexes.
\end{lemma}
\begin{proof}
  The two chain complexes are both isomorphic as vector spaces to $\hom_{\kk_{\mathbf{S}}}({\mathcal C}^+ \otimes_{\kk_{\mathbf{S}}} V_0, V_1)$.
  If we denote the two differentials by $D_{V_0, V_1}$ and $D_{V_0, V_1}^+$ respectively, then for every  $\varphi \in \hom_{\kk_{\mathbf{S}}}({\mathcal C}^+ \otimes_{\kk_{\mathbf{S}}} V_0, V_1)$, $x \in {\mathcal C}^+$ and $v \in V_0$ we have
  \begin{align*}
    (D_{V_0, V_1} \varphi)(\hat{x}\otimes_{\kk_{\mathbf{S}}} v)= & \varphi(\widehat{\partial_{\mathbf{S}}x} \otimes_{\kk_{\mathbf{S}}} v) + x \varphi(\hat{e} \otimes_{\kk_{\mathbf{S}}} v) + \varphi(\hat{e}  \otimes_{\kk_{\mathbf{S}}} xv) + \\ & \varphi(\hat{x} \otimes_{\kk_{\mathbf{S}}} d_0(v)) + d_1(\varphi( \hat{x} \otimes_{\kk_{\mathbf{S}}} v)), \\
    (D_{V_0, V_1}^+ \varphi)(\hat{x}\otimes_{\kk_{\mathbf{S}}} v)= & \varphi(\widehat{\partial_{\mathbf{S}}x} \otimes_{\kk_{\mathbf{S}}} v) + x \varphi(\hat{e} \otimes_{\kk_{\mathbf{S}}} v) + \varphi(\hat{e}  \otimes_{\kk_{\mathbf{S}}} xv) + \\ &       \varphi(\hat{x} \otimes_{\kk_{\mathbf{S}}} ev) + e \varphi( \hat{x} \otimes_{\kk_{\mathbf{S}}} v),
  \end{align*}
  where $d_0$ and $d_1$ are the differentials on $V_0$ and $V_1$ respectively. The two quantities are equal because $e$ acts as $d_0$ on $V_0^+$ and as $d_1$ on $V_1^+$. 
\end{proof}

Let ${\mathcal A}^+_{\mathbf{C}_0}$ and ${\mathcal A}_{\mathbf{C}_1}^+$ be obtained from ${\mathcal A}^+_{\mathbf{S}}$ by expansion of the idempotents. We recall that there are inclusions of algebras $T_0 \colon  {\mathcal A}^+_{\mathbf{S}} \to {\mathcal A}^+_{\mathbf{C}_0}$ and $T_1 \colon  {\mathcal A}^+_{\mathbf{S}} \to {\mathcal A}^+_{\mathbf{C}_1}$, so ${\mathcal A}^+_{\mathbf{C}_0}$ and ${\mathcal A}^+_{\mathbf{C}_1}$
also carry the structure of differential graded  ${\mathcal A}_{\mathbf{S}}^+$-bimodules. Thus we define the differential graded $({\mathcal A}^+_{\mathbf{C}_1} \mhyphen {\mathcal A}^+_{\mathbf{C}_0})$-bimodule
$${\mathcal S}^+_{\mathbf{C}_1, \mathbf{C}_0} = {\mathcal A}^+_{\mathbf{C}_1} \otimes_{{\mathcal A}^+_{\mathbf{S}}} {\mathcal S}^+ \otimes_{{\mathcal A}^+_{\mathbf{S}}} {\mathcal A}^+_{\mathbf{C}_0}$$
with the induced differential. Thus we can combine $T_0$ and $T_1$ to obtain an inclusion of bimodules  $T_{\mathbf{C}_1, \mathbf{C}_0} \colon {\mathcal S}^+  \to {\mathcal S}^+_{\mathbf{C}_1, \mathbf{C}_0}$.

Finally we consider omission of idempotents. We define the $(\kk_{\mathbf{C}_1} \mhyphen \kk_{\mathbf{C}_0})$-bimodule ${\mathcal C}^+_{\mathbf{C}_1, \mathbf{C}_0}= \kk_{\mathbf{C}_1} \otimes_{\kk_{\mathbf{S}}} {\mathcal C}^+ \otimes_{\kk_{\mathbf{S}}} \kk_{\mathbf{C}_0}$ and the  $({\mathcal A}^+_{\slashed{\mathbf{C}}_1} \mhyphen {\mathcal A}^+_{\slashed{\mathbf{C}}_0})$-bimodule
$${\mathcal S}^+_{\slashed{\mathbf{C}}_1, \slashed{\mathbf{C}}_0} = {\mathcal A}^+_{\slashed{\mathbf{C}}_1} \otimes {\mathcal C}^+_{\mathbf{C}_1, \mathbf{C}_0} \otimes {\mathcal A}^+_{\slashed{\mathbf{C}}_0}.$$
As a vector space, ${\mathcal S}^+_{\slashed{\mathbf{C}}_1, \slashed{\mathbf{C}}_0}$ is generated by  words $a_1 \cdots a_k \hat{x} b_1 \cdots b_l$ where $a_i \in {\mathcal C}^+_{\mathbf{C}_1}$, $\hat{x} \in  {\mathcal C}^+_{\mathbf{C}_1, \mathbf{C}_0}$, and $b_i \in {\mathcal C}^+_{\mathbf{C}_0}$, while ${\mathcal S}^+_{\mathbf{C}_1, \mathbf{C}_0}$ is generated by those words as above that moreover are composable, i.e.\ the right idempotent of a letter coincides with the left idempotent of the following one. Thus there is an injection $i \colon {\mathcal S}^+_{\mathbf{C}_1, \mathbf{C}_0} \to {\mathcal S}^+_{\slashed{\mathbf{C}}_1, \slashed{\mathbf{C}}_0}$ as vector spaces which allows us to define the differential $\mathfrak{d}_{\slashed{\mathbf{C}}_1, \slashed{\mathbf{C}}_0}$ on the elements of ${\mathcal C}^+_{\mathbf{C}_1, \mathbf{C}_0}$ by $\mathfrak{d}_{\slashed{\mathbf{C}}_1, \slashed{\mathbf{C}}_0}(\hat{x})= i(\mathfrak{d}_{\mathbf{C}_1, \mathbf{C}_0}(\hat{x}))$.

Now let us consider two augmentations $\slashed{\varepsilon}_i \colon {\mathcal A}^+_{\slashed{\mathbf{C}}_i} \to \F$. They induce left dg-modules $V_i$ over ${\mathcal A}_{\mathbf{S}}$ by Corollary \ref{the point of all this mess}, whose underlying $\kk_{\mathbf{S}}$-module is equal to $\kk_{\mathbf{C}_i}$, but they also induce an $({\mathcal A}^+_{\slashed{\mathbf{C}}_1} \mhyphen {\mathcal A}^+_{\slashed{\mathbf{C}}_0})$-bimodule structure on $\F$ which we denote by $\F_{\slashed{\varepsilon}_1, \slashed{\varepsilon}_0}$. The final result of the section is the following.
\begin{lemma} \label{after cast}
  There is an isomorphism of chain complexes
  $$\hom_{{\mathcal A}^+_{\slashed{\mathbf{C}}_1} \mhyphen {\mathcal A}^+_{\slashed{\mathbf{C}}_0}} ({\mathcal S}^+_{\slashed{\mathbf{C}}_1, \slashed{\mathbf{C}}_0}, \F_{\slashed{\varepsilon}_0,\slashed{\varepsilon}_1}) \cong \hom_{{\mathcal A}_{\mathbf{S}}}({\mathcal S} \otimes_{{\mathcal A}_{\mathbf{S}}} V_0, V_1).$$
\end{lemma}
\begin{proof}
  The isomorphism will be defined in several step. For the first step we recall that $\kk_{\mathbf{C}_1} \otimes \kk_{\mathbf{C}_0}$ is a $(\kk_{\mathbf{C}_1}^e \mhyphen \kk_{\mathbf{C}_0}^e)$-bimodule, and therefore
  $${\mathcal S}^{+, \mathit{ext}}_{{\slashed{\mathbf{C}}_1, \slashed{\mathbf{C}}_0}}= \kk_{\mathbf{C}_1} \otimes {\mathcal S}^+_{{\slashed{\mathbf{C}}_1, \slashed{\mathbf{C}}_0}} \otimes \kk_{\mathbf{C}_0}$$
  is a differential graded $({\mathcal A}_{\slashed{\mathbf{C}}_1}^{+, \mathit{ext}} \mhyphen {\mathcal A}_{\slashed{\mathbf{C}}_0}^{+, \mathit{ext}})$-bimodule, where ${\mathcal A}_{\slashed{\mathbf{C}}_i}^{+, \mathit{ext}} = {\mathcal A}_{\slashed{\mathbf{C}}_i}^+ \otimes \kk_{\mathbf{C}_i}^e$. We observe also that
  $$\slashed{\varepsilon}_i^{ext} = \slashed{\varepsilon} \otimes \op{Id}_{\kk_{\mathbf{C}_i}} \colon {\mathcal A}_{\slashed{\mathbf{C}}_i}^{+, \mathit{ext}} \to  \kk_{\mathbf{C}_i}^e$$
  induce a structure of $({\mathcal A}^+_{\slashed{\mathbf{C}}_1} \mhyphen {\mathcal A}^+_{\slashed{\mathbf{C}}_0})$-bimodule on $\kk_{\mathbf{C}_1} \otimes \kk_{\mathbf{C}_0}$.
  
  By properties of homomorphisms and tensor products we have
  \begin{align*}
    & \hom_{{\mathcal A}^{+, ext}_{\slashed{\mathbf{C}}_1} \mhyphen {\mathcal A}^{+, ext}_{\slashed{\mathbf{C}}_0}} ({\mathcal S}^{+, ext}_{\slashed{\mathbf{C}}_1, \slashed{\mathbf{C}}_0}, \kk_{\mathbf{C}_1} \otimes \kk_{\mathbf{C}_0}) \cong \\
    & \hom_{{\mathcal A}^+_{\slashed{\mathbf{C}}_1} \mhyphen {\mathcal A}^+_{\slashed{\mathbf{C}}_0}} ({\mathcal S}^+_{\slashed{\mathbf{C}}_0, \slashed{\mathbf{C}}_1}, \F_{\slashed{\varepsilon}_0,\slashed{\varepsilon}_1}) \otimes \hom_{\kk_{\mathbf{C}_1}^e \mhyphen \kk_{\mathbf{C}_0}^e}(\kk_{\mathbf{C}_1} \otimes \kk_{\mathbf{C}_0}, \kk_{\mathbf{C}_1} \otimes \kk_{\mathbf{C}_0}).
  \end{align*}
  By Schur's lemma $\hom_{\kk_{\mathbf{C}_1}^e \mhyphen \kk_{\mathbf{C}_0}^e}(\kk_{\mathbf{C}_1} \otimes \kk_{\mathbf{C}_0}, \kk_{\mathbf{C}_1} \otimes \kk_{\mathbf{C}_0})$ is one-dimensional and generated by the identity, which implies that the map
  \begin{align*} \hom_{{\mathcal A}^+_{\slashed{\mathbf{C}}_1} \mhyphen {\mathcal A}^+_{\slashed{\mathbf{C}}_0}} ({\mathcal S}^+_{\slashed{\mathbf{C}}_1, \slashed{\mathbf{C}}_0}, \F_{\slashed{\varepsilon}_0,\slashed{\varepsilon}_1}) & \to \hom_{{\mathcal A}^{+, ext}_{\slashed{\mathbf{C}}_1} \mhyphen {\mathcal A}^{+, ext}_{\slashed{\mathbf{C}}_0}} ({\mathcal S}^{+, ext}_{\slashed{\mathbf{C}}_1, \slashed{\mathbf{C}}_0}, \kk_{\mathbf{C}_1} \otimes \kk_{\mathbf{C}_0}), \\
    \varphi & \mapsto \varphi \otimes \op{Id}_{\kk_{\mathbf{C}_1} \otimes \kk_{\mathbf{C}_0}}
  \end{align*}
  is an isomorphism. For the second step we recall that there are inclusions ${\mathcal A}^+_{{\mathbf{C}_i}} \hookrightarrow {\mathcal A}^{+, \mathit{ext}}_{\slashed{\mathbf{C}}_i}$ which induce an inclusion ${\mathcal S}^+_{\mathbf{C}_1, \mathbf{C}_0} \hookrightarrow {\mathcal S}^{+, ext}_{\slashed{\mathbf{C}}_1, \slashed{\mathbf{C}}_0}$: in fact one can check that ${\mathcal S}^+_{\mathbf{C}_1, \mathbf{C}_0} \cong {\mathcal A}^+_{\mathbf{C}_1} \otimes_{\kk_{\mathbf{C}_1}} {\mathcal C}_{\mathbf{C}_1, \mathbf{C}_0}^+ \otimes_{\kk_{\mathbf{C}_0}} {\mathcal A}^+_{\mathbf{C}_0}$ and ${\mathcal S}^{+, ext}_{\slashed{\mathbf{C}}_1, \slashed{\mathbf{C}}_0} \cong {\mathcal A}^{+, \mathit{ext}}_{\slashed{\mathbf{C}}_1} \otimes_{\kk_{\mathbf{C}_1}} {\mathcal C}_{\mathbf{C}_1, \mathbf{C}_0}^+ \otimes_{\kk_{\mathbf{C}_0}} {\mathcal A}^{+, \mathit{ext}}_{\slashed{\mathbf{C}}_0}$. Thus restriction gives an isomorphism
  $$\hom_{{\mathcal A}^{+, ext}_{\slashed{\mathbf{C}}_1} \mhyphen {\mathcal A}^{+, ext}_{\slashed{\mathbf{C}}_0}} ({\mathcal S}^{+, ext}_{\slashed{\mathbf{C}}_1, \slashed{\mathbf{C}}_0}, \kk_{\mathbf{C}_1} \otimes \kk_{\mathbf{C}_0}) \cong \hom_{{\mathcal A}^+_{\mathbf{C}_1} \mhyphen {\mathcal A}^+_{\mathbf{C}_0}} ({\mathcal S}^+_{\mathbf{C}_1, \mathbf{C}_0}, \kk_{\mathbf{C}_1} \otimes \kk_{\mathbf{C}_0}).$$

For the third step, from ${\mathcal S}^+_{\mathbf{C}_1, \mathbf{C}_0} = {\mathcal A}^+_{\mathbf{C}_1} \otimes_{{\mathcal A}^+_{\mathbf{S}}} {\mathcal S}^+  \otimes_{{\mathcal A}^+_{\mathbf{S}}} {\mathcal A}^+_{\mathbf{C}_0}$ we deduce that
  $$\hom_{{\mathcal A}^+_{\mathbf{C}_1} \mhyphen {\mathcal A}^+_{\mathbf{C}_0}} ({\mathcal S}^+_{\mathbf{C}_1, \mathbf{C}_0}, \kk_{\mathbf{C}_1} \otimes \kk_{\mathbf{C}_0}) \cong \hom_{{\mathcal A}^+_{\mathbf{S}} \mhyphen {\mathcal A}^+_{\mathbf{S}}} ({\mathcal S}^+, \kk_{\mathbf{C}_1} \otimes \kk_{\mathbf{C}_0}).$$

  Now observe that for $i=0,1$, $\kk_{\mathbf{C}_i}$ with the ${\mathcal A}^+_{\mathbf{S}}$-module structure induced by $\varepsilon_i$ is $V_i^+$, and therefore, by the isomorphism $V_0^+ \cong (V_0^+)^*$ induced by the basis of idempotents and the adjunction between tensor product and homomorphisms, we have
  $$\hom_{{\mathcal A}^+_{\mathbf{S}} \mhyphen {\mathcal A}^+_{\mathbf{S}}} ({\mathcal S}^+, \kk_{\mathbf{C}_1} \otimes \kk_{\mathbf{C}_0}) \cong \hom_{{\mathcal A}^+_{\mathbf{S}}} ({\mathcal S}^+ \otimes V_0^+, V_1^+).$$
  Finally, from Lemma \ref{+ and d} we obtain
  $$\hom_{{\mathcal A}^+_{\mathbf{S}}} ({\mathcal S}^+ \otimes V_0^+, V_1^+) \cong \hom_{{\mathcal A}_{\mathbf{S}}} ({\mathcal S} \otimes V_0, V_1).$$
\end{proof}

\section{Chekanov-Eliashberg algebras}\label{ChekAlg}
  In this section we recollect some useful facts about Chekanov-Eliashberg algebras. They where first defined combinatorially by Chekanov in \cite{Chekanov} for Legendrian knots in the standard contact $\R^3$ and later defined analytically by Ekholm, Etnyre and Sullivan in \cite{Ekholm_Contact_Homology, Ekholm_&_Orientation_Homology, Ekholm_&_Invariants_of_embeddings} for Legendrian submanifolds in contactisations, i.e.\ contact manifolds of the form $(\R \times P, dt+ \theta)$ where $(P, \theta)$ is a Liouville manifolds. 

  Let $(Y, \alpha)$ be a contact manifold and $\Lambda \subset Y$ a closed Lagrangian submanifold. We will always assume that $\Lambda$ is {\em chord generic}, which means that the Reeb chords of $\Lambda$ are not part of closed Reeb orbits, distinct cords are disjoint, and if $p$ and $q$ are the starting and end point of a Reeb chord of length $T$ and $\varphi$ is the Reeb flow, then $d_p\varphi_T(T_p\Lambda)$ intersects $T_q\Lambda$ transversely in the contact hyperplane at $q$.  This is a generic property for $\Lambda$ provided that $\dim Y \ge 3$.

 Let $\mathcal{R}_\Lambda$ be the set of Reeb chords. For $a\in \mathcal{R}_\Lambda$ we denote by $s_-(a)$ and $s_+(a)$ the connected components of $\Lambda$ in which the start and endpoint of the chord reside. 
 Let $\kk_\Lambda$ be the idempotent ring over $\pi_0(\Lambda)$ --- note the difference in notation with Section \ref{sec: idempotents and dga} --- and let $C(\Lambda)$ be the vector space over $\F$ with basis $\mathcal{R}(\Lambda)$. We endow $C(\Lambda)$ with  the structure of a $\kk_\Lambda$-bimodule by 
    $$s_-\cdot a\cdot s_+=
    \begin{cases} a \text{ if } s_\pm = s_\pm(a),  \text{ and }\\
  0 \text{ otherwise}. 
\end{cases}$$

We denote by $\mathcal{A}(\Lambda)$ --- or $\mathcal{A}_{\kk_{\Lambda}}(\Lambda)$  if we need to be explicit about the ring of coefficients --- the tensor algebra of $C(\Lambda)$  over $\kk_\Lambda$.

If the Maslov class of $\Lambda$ vanishes, then ${\mathcal A}(\Lambda)$ is graded by the Conley-Zehnder index of the chords (see \cite{LCHgeneral} for the definition); otherwise the Conley-Zehnder index only gives a relative grading in a cyclic group. Moreover, if $\Lambda$ is disconnected, the grading depends on an arbitrary extra piece of data which is called a {\em Maslov potential}.

On $\mathcal{A}(\Lambda)$ we define a differential by counting $J$-holomorphic maps in the symplectisation of $(Y, \alpha)$. More recisely, we fix a cylindrical almost complex structure on $J$ on $\R \times Y$ which is adapted to $\alpha$ and denote by ${\mathcal M}_{\Lambda}(a; b_1, \ldots, b_n)$ the moduli space of $J$-holomorphic maps from a punctured disc to $\R \times Y$ which map the boundary of the disc to $\R \times \Lambda$, are positively asymptotic to the Reeb chord $a$ and negatively asymptotic to the Reeb chords $b_1, \ldots, b_n$ (ordered in the counterclockwise direction starting from $a$).

  \begin{rem}\label{rem: how we learnt not to worry, but not yet to love interior punctures}
  If $\alpha$ has closed Reeb orbits, one has to consider also bubbling of holomorphic planes in the compactification of the moduli spaces  ${\mathcal M}_\Lambda(a, b_1, \ldots, b_l)$ and of the more general ones which will be introduced later in the paper. One can therefore take three possible approaches:
  \begin{enumerate}
  \item allow degenerations to closed Reeb orbit and, consequently, define every algebraic invariant as a module over the contact homology algebra,
  \item require that closed Reeb orbits have large index, so that low-dimensional moduli spaces cannot bubble holomorphic planes, or
  \item require that $(Y, \alpha)$ is  filled by a Liouville domain and considere {\em anchored discs}, i.e.\ discs in the cobordisms with possibly negative ends at closed Reeb orbits which are capped by holomorphic planes in the filling. See \cite{EkholmCurve} for the precise definition.
  \end{enumerate}
  The first approach is suboptimal because there are still unresolved issues about invariance of the contact homology algebra. The second approach is the simplest one, but is not always possible. On the other hand, all contact manifolds in this paper will be the boundary of a Liouville domain, so we will take the third approach: our holomorphic curves will always be anchored, even if we will make no further mention of that. The drawback of using anchored discs is that they require abstract perturbations, even if of a simpler kind than those needed for contact homology. Readers who are uncomfortable with abstract perturbations can assume that the negative end of the Liouville cobordisms introduced later in the paper is the standard contact sphere, where the second approach works.  In Section \ref{sec:reduc-cont-stopp}, we explain how to use stops and partially wrapped Fukaya categories to reduce the computation of the Floer complex of compact Lagrangians to a computation of Legendrian contact homology in symplectisation where the second approach can be applied.
\end{rem}

We denote by ${\mathcal M}_{\Lambda}^i(a; b_1, \ldots, b_n)$ the subset of ${\mathcal M}_{\Lambda}(a; b_1, \ldots, b_n)$ consisting of maps of Fredholm index $i$. For a generic choice of the almost complex structure, ${\mathcal M}^i_{\Lambda}(a; b_1, \ldots, b_n)$ is a transversely cut out manifold of dimension $i$ carrying a free (if $i>0$) and proper action of $\R$ by translations in the symplectisation direction. We also denote
$$\widetilde{\mathcal M}_{\Lambda}^i(a; b_1, \ldots, b_n) = {\mathcal M}_{\Lambda}^i(a; b_1, \ldots, b_n)/ \R.$$

We define
  $$\partial a = \sum_{b_1\ldots b_n} \# \widetilde{\mathcal M}_{\Lambda}^1(a; b_1, \ldots, b_n) s_-(a) \cdot b_n \cdots b_1\cdot s_+(a),$$
  where the sum is taken over all words on Reeb chords of $\Lambda$, including the empty word, with the convention that it corresponds to the unit. The idempotents at the beginning and end of the word are absorbed into the chords if the word is not empty, so they have a nontrivial effect only for the empty word. Note that the order of the word is opposite to the order used in the original articles. We made this choice so that, in the case of a disconnected Legendrian, $a$ and the word $b_n \ldots b_1$ have starting and end points in the same connected component when reading the word from left to right.

  Invariance of ${\mathcal A}(\Lambda)$ up to quasi-isomorphism can be proved by a cobordism argument as follows. Suppose that, for $s \in \R$, we have a path of cylindrical almost complex structures $\{J_s \}$ and a Legendrian isotopy $\{ \Lambda_s \}$ which are constant for $s \le0$ and $s \ge 1$; then we can construct an almost complex structure $J$ on $\R \times Y$ and a Lagrangian cobordism $C$ such that $J=J_0$ and $C$ is a cylinder over $\Lambda_0$ on $(- \infty, 0] \times Y$, and $J=J_1$ and $C$ is a cylinder over $\Lambda_1$ on $[1, + \infty) \times Y$. We define the moduli spaces ${\mathcal M}^0_C(a; b_1, \ldots, b_n)$
of index zero $J$-holomorphic punctured discs in $\R \times Y$ with boundary on $C$, a positive end at the Reeb chord $a$ of $\Lambda_1$ and negative ends at chords $b_1, \ldots, b_n$ of $\Lambda_0$, and the continuation maps
  $$\Phi \colon {\mathcal A}(\Lambda_1, J_1) \to {\mathcal A}(\Lambda_0, J_0),$$
  $$\Phi(a) = \sum_{b_1 \ldots b_n} \#{\mathcal M}^0_C(a; b_1, \ldots, b_n) s_-(a) \cdot b_n \cdots b_1 \cdot s_+(a).$$
  The continuation maps satisfy the usual properties:
  \begin{itemize}
  \item the constant path of almost complex structure and isotopy yield the identity map,
  \item a compactly supported deformation of $J$ and $C$ yields a dg-homotopic map, and
  \item concatenation of paths of almost complex strutures and Legendrian isotopies corresponds, up to dg-homotopy,  to composition of continuation maps.
  \end{itemize}
  Therefore we can prove that continuation maps are quasi-isomorphisms by concatenating a path of data with the opposite path and then deforming it to the identity path.

  The following lemma\footnote{We thank Tobias Ekholm for suggesting the proof of this lemma} is useful for taking limits of Chekanov-Elieashberg algebras.
  \begin{lemma}\label{if nothing happens nothing happens}
    Suppose that there exists $Q>0$ such that:
    \begin{itemize}
    \item there is a canonical bijection between Reeb chords of action less than $Q$ in  $\Lambda_s$ and $\Lambda_{s'}$ for all $s, s' \in [0,1]$, and
    \item all the moduli spaces ${\mathcal M}_{\Lambda_s}(a; b_1, \ldots, b_n)$ where $a$ is a Reeb chord of $\Lambda_s$ of action less than $Q$ and $s \in [0,1]$ are regular.
    \end{itemize}
Then the continuation map $\Phi \colon {\mathcal A}(\Lambda_1, J_1) \to {\mathcal A}(\Lambda_0, J_0)$ is homotopic to a map which restricts to the canonical bijection on chords of action less than $Q$.
\end{lemma}
\begin{proof}
  First note that if the continuation map ${\mathcal A}(\Lambda_1, J_1) \to {\mathcal A}(\Lambda_0, J_0)$ does not satisfy the property we want to prove, then for every $s \in [0,1]$ at least one between the continuation maps ${\mathcal A}(\Lambda_1, J_1) \to {\mathcal A}(\Lambda_{s}, J_{s})$ and ${\mathcal A}(\Lambda_{s}, J_{s}) \to {\mathcal A}(\Lambda_0, J_0)$ also does not satisfy it. Therefore, proceeding by bisection, there are sequences $s_i, s_i' \to s_\infty$ with $s_i < s_i'$ such that the continuation maps ${\mathcal A}(\Lambda_{s_i'}, J_{s_i'}) \to {\mathcal A}(\Lambda_{s_i}, J_{s_i})$ do not satisfy the property we want to prove. These continuation maps are defined from families of almost complex structures $\widetilde{J}_i$ interpolating between $J_{s_i'}$ and $J_{s_i}$ and Lagrangian cobordisms $C_i$ with positive end at $\Lambda_{s_i'}$ and negative end at $\Lambda_{s_i}$ such that $\widetilde{J}_i \to J_{s_\infty}$ and $C_i \to \R \times \Lambda_{s_\infty}$ as $i \to \infty$. Since the continuation maps ${\mathcal A}(\Lambda_{s_i'}, J_{s_i'}) \to {\mathcal A}(\Lambda_{s_i}, J_{s_i})$ is not homotopic to the identity below action $Q$, for every $i$ there is a nontrivial index zero $\widetilde{J}_i$-holomorphic curve in $\R \times Y$ with boundary on $C_i$ and positively asymptotic to some chord of action less that $Q$. Here ``nontrivial'' means that, for $i$ large enough, it is not a perturbation of a trivial strip on $\R \times \Lambda_{s_\infty}$.

  Then we apply SFT compactness and obtain in the limit a nontrivial index zero $J_{s_\infty}$-holomorphic building with boundary on $\R \times \Lambda_{s_\infty}$  and the same asymptotics. This is a contradiction because $J_{s_\infty}$ is cylindrical and regular below action $Q$, and therefore the only index zero $J_{s_\infty}$-holomorphic buildings with boundary on $\R \times \Lambda_{s_\infty}$ and positive asymptrotic to a Reeb chord of action less than $Q$ are trivial strips. 
\end{proof}

Let ${\mathcal A}$ be a semi-projective differential graded algebra generated by a set $A$. Given a map $\mathfrak{a} \colon A \to [0, + \infty)$ and $Q \ge 0$ we denote by ${\mathcal A}^Q$ the sub-algebra of $A$ generated by those $a \in A$ such that $\mathfrak{a}(a) \ge Q$. We say that $\mathfrak{a}$ is an {\em action filtration} if ${\mathcal A}^Q$ is a sub-dga of ${\mathcal A}$ for every $Q \ge 0$.
This definition is an abstraction of the action filtration on the Chekanov-Eliashberg algebra.
An {\em action-filtered} differential graded algebra is a semi-projective differential graded algebra endowed with an action filtration. The following Lemma will be used to prove that very long chords between parallel copies of a Legendrian submanifold do not contribute to the quasi-isomorphism type of the Chekanov-Eliashberg algebra as the parallel copies become  closer and closer to each other. 
\begin{lemma}\label{taking limits}
  Let $\{ {\mathcal A}_n, f_n, g_n \}_{n \in \N}$ be a sequence of action-filtered differential graded algebras and dg-morphisms $f_n \colon {\mathcal A}_n \to {\mathcal A}_{n+1}$
  and $g_n \colon {\mathcal A}_{n+1} \to {\mathcal A}_n$ such that $f_n \circ g_n$ and $g_n \circ f_n$ induce the identity in homology, and assume that for each $n \in \N$ there exists $\delta_n>0$ such that $\mathfrak{a}(f_n(a)) < \mathfrak{a}(a)+ \delta_n$ for all $a \in {\mathcal A}_{n+1}$.
  If ${\mathcal B}$ is an action-filtered differential graded algebra, $Q_n$ a sequence such that
\begin{equation}\label{Qn grows fast}
  \lim_{n \to \infty} \left ( Q_n - \sum_{i=1}^{n-1} \delta_i \right ) \to + \infty,
\end{equation}
$i_n \colon {\mathcal B}^{Q_n} \to {\mathcal A}^{Q_n}$ is a dg-isomorphism for every $n$ and the diagram

 \begin{equation}\label{diagram for the limit}
  \xymatrix{
   {\mathcal A}_0  & {\mathcal A}_1 \ar[l]^-{g_0}  & \ldots \ar[l]^-{g_{1}}  & {\mathcal A}_n \ar[l]^-{g_{n-1}}  & \ldots \ar[l]^-{g_n} \\
  {\mathcal B}^{Q_0} \ar[u]^-{i_0} \ar[r]^-{\subset} &  {\mathcal B}_1^{Q_1} \ar[r]^-{\subset} \ar[u]^{i_1} & \ldots \ar[r]^-{\subset} & {\mathcal B}_n^{Q_n} \ar[r]^-{\subset} \ar[u]^{i_n} & \ldots
 }
 \end{equation}
commutes, then there is a dg-morphism $j \colon {\mathcal B} \to {\mathcal A}_0$ which induces an isomorphism in homology.
\end{lemma}

\begin{proof}
  We define the map $j \colon {\mathcal B} \to {\mathcal A}_0 $ as follows: for every $b \in {\mathcal B}$ we choose $n$ such that $b \in {\mathcal B}^{Q_n}$  and define $j(b)= g_0 \circ \ldots \circ g_{n-1} \circ i_n(b)$. The commutative diagram \eqref{diagram for the limit} shows that $j$ is a well defined dg algebra morphism.

  Next we show that $j$ is injective in homology. Suppose that for some cycle $b \in {\mathcal B}$ there is an element $a \in {\mathcal A}_0$ such that $\partial (a) = j(b)$. Then there is $\overline{Q}$ such that $a \in {\mathcal A}_0^{\overline{Q}}$. By Equation \eqref{Qn grows fast}, there is an $n$ such that $f_{n-1} \circ \ldots \circ f_1(a) \in {\mathcal A}_n^{Q_n}$, and therefore $\partial( i_n^{-1} \circ f_{n-1} \circ \ldots \circ f_1(a))= i_n^{-1} \circ f_{n-1} \circ \ldots \circ f_1 \circ j (b)$. Since $b$ and $i_n^{-1} \circ f_{n-1} \circ \ldots \circ f_1 \circ j (b)$ represent the same homology class, it follows that $b$ is a boundary, and therefore $j$ is injective in homology.

  Finally we prove that $j$ is surjective in homology. Take a cycle $a \in {\mathcal A}_0$ representing some homology class. Then there is some $n$ for which $f_{n-1} \circ \ldots \circ f_1(a) \in {\mathcal A}_n^{Q_n}$. We define $b$ = $i_n^{-1} \circ f_{n-1} \circ \ldots \circ f_1(a)$ and observe that $j(b)$ and $a$ represent the same homology class.
\end{proof}

\section{Immersed Lagrangian cobordisms and their differential graded algebras} \label{sec: immersed cobordisms}
In this section we define immersed exact Lagrangian cobordisms and associate to them differential graded algebras which generalise the Chekanov-Eliashberg algebras of Legendrian submanifolds. Representations of those algebras will act as bounding cochains for the Floer homology that will be introduced in Section \ref{sec:cthulhu-complex-with}.

  \begin{dfn}
    A {\em manifold with cylindrical ends} is a manifold $M$ with a compact, codimension zero submanifold $M^c$, a partition
    $$\partial M^c=\partial_- M^c \sqcup \partial_+ M^c$$
    and a fixed identification of $M \setminus M^c$ with
    $$(-\infty, 0) \times \partial_- M^c \sqcup (0, + \infty) \times \partial_+ M^c.$$
    The part of $M \setminus M^c$ which is identified to $(-\infty, 0) \times \partial_- M^c$ is called the {\em negative end} of $M$ and the part which is identified to $(0, + \infty) \times \partial_+ M^c$ is called the {\em positive end}.
  \end{dfn}
  We allow either $\partial_- M^c$ or $\partial_-+M^c$ (or both) to be empty. In our setting $M$ will have no boundary and $\partial_{\pm} M^c$ will be closed. Abusing the notation, we will write $\partial_\pm M$ for $\partial_{\pm} M^c$.

  \begin{dfn}
    A {\em Liouville cobordism} $(X, \theta)$ is a $2n$-dimensional manifold with cylindrical ends $X$ endowed with a one-form $\theta$ such that $d \theta$ is symplectic and $\theta$ pulls back to $e^s\alpha_+$ on  $(0, + \infty) \times \partial_+ X$ and to $e^s\alpha_-$ on $(- \infty, 0) \times \partial_- X$ for contact forms $\alpha_\pm$ on $\partial_\pm X$.
  \end{dfn}
  In this context the positive and negative end are rather called the convex and concave end, respectively.

  \begin{dfn}
    An immersed exact Lagrangian cobordism in the Liouville cobordism $(X, \theta)$ is a proper immersion $\iota \colon L \to X$ such that
    \begin{itemize}
    \item $L$ is a manifold with cylindrical ends and $\dim L =n$,
    \item the positive end of $L$ is mapped onto $(0, +\infty) \times \Lambda^+$ for a submanifold $\Lambda^+ \subset \partial_+X$ by a diffeomorphism respecting the product structure,
    \item  the negative end of $L$ is mapped onto $(-\infty, 0) \times \Lambda^-$ for a submanifold $\Lambda^- \subset \partial_-X$ by a diffeomorphism respecting the product structure, and
    \item $\iota^* \beta = dh$ for a function $ h \colon L \to \R$ which is constant on $(- \infty, 0) \times \partial_-L$ and $ (0, + \infty) \times \partial_+L$.
    \end{itemize}
  \end{dfn}
  The function $h$ is called the {\em potential} of $L$. These conditions imply that $\Lambda^{\pm}$ is a Legendrian submanifold of $(\partial_{\pm}X, \alpha_{\pm})$.  With an abuse of notation we will identify $L$ with its image and denote $\partial_{\pm}L=\Lambda^{\pm}$. We will always assume that $\Lambda^\pm$ are chord generic, $L$ has only transverse double points, and for every double point the function $h$ takes different values at the two preimages. For the definition of Floer homology it is necessary only that $h$ be constant on the negative end, but we have to require that $h$ be constant also on the positive end if we want to be sure that the concatenation of two exact Lagrangian cobordisms is still exact.

  We fix a generic almost complex structure $J$ on $X$ which is compatible with $d \theta$ and cylindrical in the ends. If $u \colon (0, +\infty) \times [0,1] \to X$ is a $J$-holomorphic map such that $u((0, +\infty) \times \{0,1 \}) \subset L$ and $\lim \limits_{s \to + \infty} u(t,s)=p$, we say that $u$ {\em approaches $p$ in the positive direction} if, for $s_0$ sufficiently large, $u$ maps $(s_0, + \infty) \times \{ 0 \}$ to the branch of $L$ near $p$ with higher potential and $(s_0, + \infty) \times \{ 1 \}$ to the branch with lower potential.
Similarly, if $u \colon (-\infty, 0) \times [0,1] \to X$ is a $J$-holomorphic map such that $u((- \infty, 0) \times \{0,1 \}) \subset L$ and $\lim \limits_{s \to - \infty} u(t,s)=p$, we say that $u$ {\em approaches $p$ in the negative direction} if, for $s_0$ sufficiently large, $u$ maps
$(- \infty, -s_0) \times \{0 \}$ to the branch of $L$ near $p$ with  higher potential and $(- \infty, -s_0) \times \{ 1 \}$ to the branch with lower potential. These conditions are borrowed from Legendrian contact homology; see e.g.\ \cite{LCHgeneral}.

If $a$ is a self-intersection point of $L$ or a Reeb chord of $\Lambda^+$ and $b_1, \ldots, b_l$ are self-intersection points of $L$ or chords of $\Lambda^-$ we denote by ${\mathcal M}_L(a, b_1, \ldots, b_l)$ the moduli space of $J$-holomorphic polygons in $X$ with boundary on $L$, a positive end at $a$, negative ends at $b_1, \ldots, b_n$, and which moreover satisfy the additional constraint that the positive end approaches $a$ in the positive direction if $a$ is a self-intersection point, and the negative ends converging to self-intersection points approach them in the negative direction. There could be also interior negative punctures which are dealt with according to Remark \ref{rem: how we learnt not to worry, but not yet to love interior punctures}.

These moduli spaces are regular for a generic $J$ because  in \cite{LCHgeneral} and \cite{LegendrianAmbient} the perturbation happens near the positive puncture and it is irrelevant if  the negative punctures are reeb chords or self-intersection points. Regularity for anchored discs was proved in \cite{EkholmCurve}. We denote by ${\mathcal M}_L^d(a, b_1, \ldots, b_l)$ the $d$-dimensional part of ${\mathcal M}_L(a, b_1, \ldots, b_l)$. If the Maslov class of $L$ vanishes (which implies in particular that $c_1(X)=0$)  we can grade Reeb chords and self-intersection points following \cite{LCHgeneral} and \cite{Seidel_Fukaya}. We will denote the degree of $a$ by $|a|$. The index formula for the Caucy-Riemann operator gives the following formulas for ${\mathcal M}_L(a, b_1, \ldots, b_l)$. The following Lemma follows from \cite[Theorem A.1]{CieliebakSwitching}
\begin{lemma}
  If $a$ is a self-intersection point, then
\begin{equation}\label{dimension formula for cobordism algebra}
  \dim {\mathcal M}_L(a, b_1, \ldots, b_l)= |a|-|b_1|- \ldots - |b_l| -1.
\end{equation}
If $a$ is a Reeb chord of $\Lambda^+$, then
\begin{equation}\label{dimension formula for cobordism maps}
  \dim {\mathcal M}_L(a, b_1, \ldots, b_l)= |a|-|b_1|- \ldots - |b_l|.
\end{equation}
\end{lemma}
We can also associate an action $\mathfrak{a}$ to self-intersections and Reeb chords as follows:
 \begin{itemize}
 \item if $a$ is a Reeb chord of $\Lambda^\pm$, then $\mathfrak{a}(a)= \int_{a} \alpha_\pm$, and
 \item if $a$ is a self-intersection point, then $\mathfrak{a}(a) =h(a^+)-h(a^-)$, where $a^+$ and $a^-$ are the preimages of $a$ such that $h(a^+)>h(a^-)$.
 \end{itemize}

 \begin{lemma}\label{d_L decreases action}
  If ${\mathcal M}_L(a, b_1, \ldots, b_l)$ is nonempty, then $\mathfrak{a}(a) > \mathfrak{a}(b_1) + \ldots + \mathfrak{a}(b_l)$.
 \end{lemma}
 \begin{proof}
   Let $\tilde \theta$ be the continuous and piecewise smooth $1$-form which coincides with $\theta$ on $X^c$, with $\alpha_-$ on $(- \infty, 0) \times \partial_- X$ and with $\alpha_+$ on $(0, + \infty) \times \partial_+ X$. Note that both $\theta$ and $\tilde \theta$ pull-back to zero on $L$ in the region where they differ, and therefore $dh= \iota^*\tilde \theta$. If $u \colon D_{l+1} \to X$ is a $J$-holomorphic map representing an element of ${\mathcal M}_L(a, b_1, \ldots, b_l)$, then $u^* d \tilde \theta \ge 0$ because $J$ is compatible with $\theta$ on $X^c$ and $J$ is cylindrical and compatible with the contact forms $\alpha_\pm$ in $X \setminus X^c$, and moreover the strict inequality holds where $\theta= \tilde \theta$. Stokes theorem can still be applied to $d \tilde \theta$ despite it being discontinuous because it can be applied separately to the pieces of $X$ on which $\tilde{\theta}$ is smooth, and therefore
   $$0 < \int_{D_{l+1}} u^*d \tilde{\theta}= \mathfrak{a}(a)- \mathfrak{a}(b_1) - \ldots - 
   \mathfrak{a}(b_l)$$
   because $h$ is constant on the ends of $L$.
 \end{proof}

 Let $\mathbb{F}$ be a field, which we will assume of characteristic two for simplicity. To an immersed exact Lagrangian cobordism $L$ we associate a differential graded algebra ${\mathcal D}_L$ over $\mathbb{F}$ called the {\em cobordism algebra}. This algebra is isomorphic to a dga defined by Asplund and Ekholm in \cite{AsplundEkholmSingular} for a Legendrian lift of $L$ and plays an analogous role to the obstruction algebra for immersed Lagrangians with no negative ends of \cite{generation} and the immersed dga for immersed cobordisms between Legendrian knots in three-dimensional jet spaces of \cite{MR4325415}. The whole discussion up to Definition \ref{def:imm-dga-map} shows that an immersed cobordism induces  what in \cite{MR4325415} is called  an {\rm immersed dga map} from the Chekanov-Eliashberg algebra of the positive end to the Chekanov-Eliashberg algebra of the negative end. It would be possible to repeat the constructions of this section over an idempotent ring, but we will leave this generalisation to the reades since we will have no use for it in the present article.

\begin{dfn}\label{cobordism algebra}
The cobordism algebra is the semi-projective differential graded algebra ${\mathcal D}_L$ generated over $\mathbb{F}$ by the chords of $\Lambda^-$ and by the self-intersection points of $L$. The differential $\partial_L$ is determined by the Leibniz rule together with the following properties:
    
       \begin{itemize}
       \item  if $a$ is a chord of $\Lambda^-$, then $\partial_L (a) = \partial a$, where $\partial a$ is the differential of $a$ in the Chekanov-Eliashberg algebra ${\mathcal A}_{\slashed{\Lambda}^-}(\Lambda^-)$, and 
       \item if $a$ is a self-intersection point, then 
         $$\partial_{L} (a) = \sum \# {\mathcal M}_L^0(a, b_1, \ldots, b_l) b_1 \ldots b_l,$$
         where $b_1, \ldots, b_l$ are self-intersection points of $L$ or chords of $\Lambda^-$. 
       \end{itemize}
     \end{dfn}
     The sum in the definition is finite, and therefore $\partial$ is well defined, by Lemma \ref{d_L decreases action} and SFT compactness because there are only finitely many generators of ${\mathcal D}_L$ below any given action. Lemma \ref{d_L decreases action} also implies that $\partial_L$ decreases the action filtration induced on ${\mathcal D}_L$ by $\mathfrak{a}$. Equation \eqref{dimension formula for cobordism algebra} implies that $\partial_L$ decreases the degree by $1$.
     
 So far we have called $\partial_L$ a differential, but we have not proved that $\partial_L^2=0$; it is time now to pay the debt.
 \begin{lemma} \label{partial_L^2=0}
   $\partial_L^2=0$.
 \end{lemma}
 \begin{proof}
   The lemma is, as usual, proved by analysing the degenerations of one-dimensional moduli spaces. The limit configurations which do not contribute to $\partial^2$ consist of $J$-holomorphic buildings containing a component with more than one positive puncture.
 Those buildings have a sub-building whose external ends (i.e.\ ends which do not connect one component of the building to another) are negative ends to Reeb chords of $\alpha_-$ or negative end to self-intersection points which are approached in the negative direction. Such a building would have negative energy, which is a contradiction.
 \end{proof}

 \begin{dfn}\label{def:imm-dga-map}
   Let ${\mathcal A}_\F(\Lambda^+)$ be the Chekanov-Eliashberg algebra of $\Lambda^+$ over $\F$ --- here we depart from the notation of Section \ref{sec: idempotents and dga}. We define a morphism of differential graded algebras $\Phi_L \colon  {\mathcal A}_\F(\Lambda^+) \to {\mathcal D}_L$ by
 \begin{equation} \label{cobordism map}
   \Phi_L(a)= \sum_{k=0}^\infty \sum_{b_1, \ldots, b_k} \# {\mathcal M}_L^0(a; b_1, \ldots, b_k) b_1 \ldots b_k,
 \end{equation}
 where $a$ is a Reeb chord of $\Lambda^+$ and each $b_i$ can be either a double point of $L$ or a Reeb chord of $\Lambda_-$. The sum is finite by Lemma \ref{d_L decreases action} and SFT compactness, and $\Phi_L$ has degree zero by Equation \eqref{dimension formula for cobordism maps}. The proof that $\Phi_L$  is a chain map is similar to the proof of Lemma \ref{partial_L^2=0} and will be omitted.\end{dfn}

Now we describe how the cobordism algebra behaves when an immersed exact Lagrangian cobordism is split as the concatenation of two cobordisms.  Let $(X, \theta)$ be a Liouville cobordism and let $Y \subset X$ be a separating hypersurface of contact type. In a neighbourhood of $Y$ of the form $(- \epsilon, \epsilon) \times Y$, where $Y$ is identified with $\{0\}\times Y$, we can write $\theta= e^s \alpha_Y$ where $\alpha_Y$ is a contact form on $Y$.  

Let $(X^+, \theta_+)$ and $(X^-, \theta_-)$ be the Liouville completions of the connected components of $X \setminus Y$ such that $\partial_-X= \partial_-X^-$, $\partial_+X= \partial_+ X^+$ and $\partial_+ X^- = \partial_- X^+=Y$. An almost complex structure $J$ on $X$ which is compatible with $\theta$ and cyindrical on $(- \epsilon, \epsilon) \times Y$ induces almost complex structures $J^+$ on $X^+$ and $J^-$ on $X^-$ which are compatible with $\theta_+$ and $\theta_-$ respectively.

\begin{dfn}\label{dfn: nicely split}
 An immersed exact Lagrangian cobordism $L \subset X$ is {\em nicely split} by $Y$ if $L \cap ((- \epsilon, \epsilon) \times Y) = (- \epsilon, \epsilon) \times \Lambda$ for some submanifold $\Lambda \subset Y$ which is Legendrian with respect to the contact form  $\alpha_Y$ and if the potential $h \colon L \to \R$ vanishes on $(- \epsilon, \epsilon) \times \Lambda$. 
\end{dfn}
Note that the second condition does not loses generality if $\Lambda$ is connected. We denote by $L^+$ and $L^-$ the Liouville completions of the connected components of $L \setminus Y$ such that $L^- \subset X^-$ and $L^+ \subset X^+$. In particular $\partial_- L^- = \partial^- L = \Lambda^-$,  $\partial_+L^+ = \partial^+ L = \Lambda^+$ and $\partial_+L^- = \partial^- L^+ = \Lambda$.

We can stretch the Liouville form near $Y$ as follows: we fix a smooth function $\phi \colon (- \epsilon, \epsilon) \to \R$ such that
\begin{itemize}
 \item $\phi(s) = 0$ for $s \le - \frac \epsilon 2$,
 \item $\phi(s)=1$ for $s \ge \frac \epsilon 2$,
 \item $\phi'(s) \ge 0$ for all $s \in (- \epsilon, \epsilon)$
 \end{itemize}
 and extend it to a function $\phi \colon X \to \R$ which is locally constant outside $(- \epsilon, \epsilon) \times Y$. Then for $R \ge 0$ we define $\theta_R= e^{R \phi} \theta$. Note that $\theta_0= \theta$. The next lemma is straightforward.
 \begin{lemma}\label{inflating}
 For every $R \ge 0$ the form $\theta_R$ is a Liouville form, and if $J$ is an almost complex structure which is compatible with $\theta$ and cylindrical in $(- \epsilon, \epsilon) \times Y$, then $J$ is compatible with $\theta_R$ for all $R \ge 0$. If the exact immersed Lagrangian cobordism $L$ with potential functrion $h$ is nicely split by $Y$, then $L$ is also exact for the Liouville forms $\theta_R$ with potential function $e^{R \phi}h$.
\end{lemma}  
Stretching of the Liouville form will be used to constrain holomorphic curves in $X$ by energy arguments. In the next two lemmas we explore the relationship between the cobordism algebras of $L$, $L^-$ and $L^+$. 

\begin{lemma}\label{subalgebra}
  Let $J$ be a generic almost complex structure on $X$ which is compatible with $\theta$ and cylindrical on $(- \epsilon, \epsilon) \times Y$, and let $J^-$ be the almost complex structure induced on $X^-$. If an immersed exact Lagrangian cobordism $L$ is nicely split by $Y$, the cobordism algebra ${\mathcal D}_L$ is defined using $J$ and the cobordism algebra ${\mathcal D}_{L_-}$ is defined using $J^-$, then 
  ${\mathcal D}_{L_-}$ is a dg sub-algebra of ${\mathcal D}_L$. 
\end{lemma}
\begin{proof}
  The generators of ${\mathcal D}_{L_-}$ are a subset of the generators of ${\mathcal D}_L$, and therefore ${\mathcal D}_{L_-}$ is a sub-algebra of ${\mathcal D}_L$. Next, if $a$ is a self-intersection point of $L^-$ and the moduli space ${\mathcal M}_L(a, b_1, \ldots b_l)$ is non-empty, then $b_1, \ldots, b_l$ are either Reeb chords of $\Lambda^-$ or self-intersection points of $L^-$. In fact, if we denote by $\mathfrak{a}_R$ the action of the double points computed using the Liouville form $\theta_R$, by Lemma \ref{inflating} the inequality of Lemma \ref{d_L decreases action} holds with $\mathfrak{a}_R$ instead of $\mathfrak{a}$. Since $a$ is a self-intersection point of $L^-$ we have $\mathfrak{a}_R(a)= \mathfrak{a}(a)$, but if some $b_i$ is a self-intersection point of $L^+$, then $\mathfrak{a}_R(b_i)=e^R\mathfrak{a}(b_i)$ and therefore Lemma \ref{d_L decreases action} is violated for $R$ large enough because $\mathfrak{a}_R(b_i)>0$.

  Finally, when $a$ is a self-intersection point of $L^-$ and $b_1, \ldots, b_l$ are either Reeb chords of $\Lambda^-$ or self-intersection points of $L^-$, there is an identification between  ${\mathcal M}_L(a, b_1, \ldots b_l)$ and ${\mathcal M}_{L^-}(a, b_1, \ldots b_l)$. Suppose on the contrary that the image of $u \in {\mathcal M}_L(a, b_1, \ldots b_l)$ goes  above $\{\epsilon \} \times Y$: then
  $$\mathfrak{a}(a)- \sum \mathfrak{a}(b_i)= \mathfrak{a}_R(a)- \sum \mathfrak{a}_R(b_i) = \int u^* \theta_R \ge Ce^R$$
  for some constant $C>0$. This is a contraddiction for $R$ large enough, and therefore every $J$-holomorphic curve in ${\mathcal M}_L(a, b_1, \ldots b_l)$ is also a $J_-$-holomorphic curve in ${\mathcal M}_{L^-}(a, b_1, \ldots b_l)$. A similar argument shows that every $J_-$-holomorphic curve in ${\mathcal M}_{L^-}(a, b_1, \ldots b_l)$ is also a $J$-holomorphic curve in ${\mathcal M}_L(a, b_1, \ldots b_l)$.
\end{proof}
\begin{dfn} \label{dga morphisms induced by cobordisms}
  We define the algebra morphism $\Phi_{L^-} \colon {\mathcal D}_{L^+} \to {\mathcal D}_L$ by
  \begin{equation}\label{marselha1}
\Phi_{L^-}(a)=a
\end{equation}
if $a$ is a self-intersection point of $L^+$, and
 \begin{equation}\label{marselha2}
\Phi_{L^-}(a)=  \sum_{b_1, \ldots, b_l} \# {\mathcal M}^0_{L^-}(a; b_1, \ldots, b_l) b_1 \ldots b_l
\end{equation}
if $a$ is a Reeb chord of $\Lambda$.
\end{dfn}

\begin{lemma} \label{lemma: mercoledi}
  The map $\Phi_{L^-}$ is a dg algebra morphism.
\end{lemma}

\begin{figure}
  \includegraphics[scale=0.5]{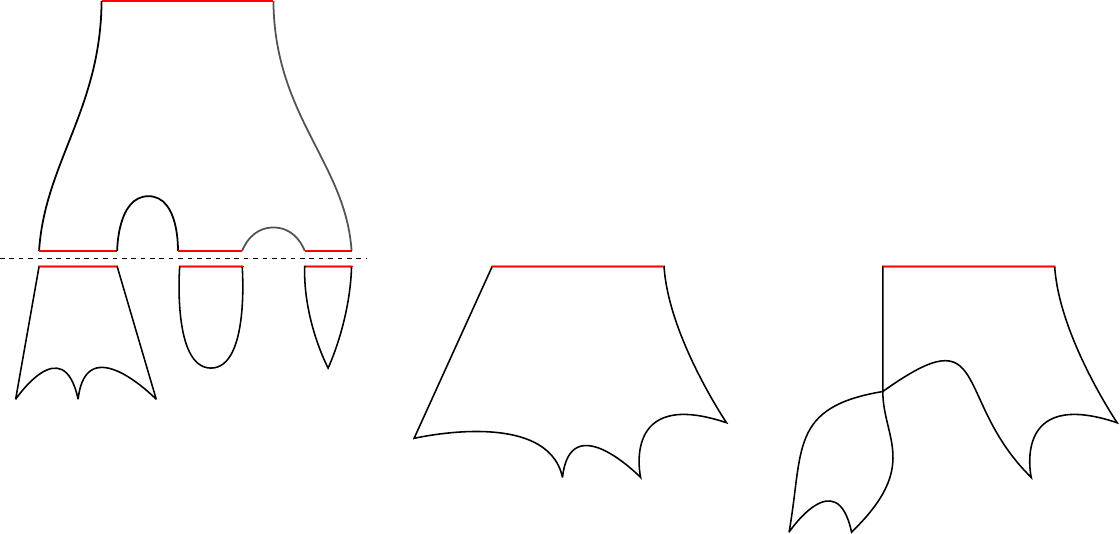}
  \caption{A one-dimensional family of $J$-holomorphic curves with one positive end at a chord of $\Lambda$ and several negative ends at self-intersection points of $L^-$ (centre)  can degenerate either as in the right or as in the left. We omit the negative asymptotics toward Reebd chords and closed orbits in all our Figures.}
  \label{fig:chainmap1}  
\end{figure}

\begin{figure}
    \label{fig:chainmap2}
    \includegraphics[scale=0.55]{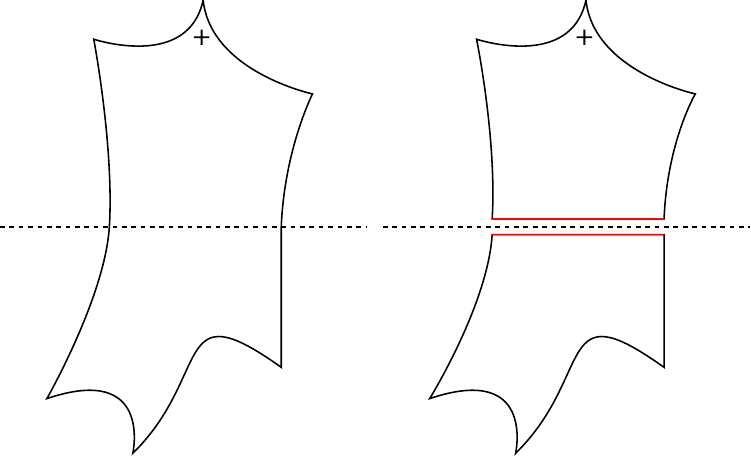}
    \caption{Rigid $J$-holomorphic curves with one positive end at $a$ and several negative ends at self-intersection points of $L$ (left) degenerate as in the left when we stretch the neck around $Y$.}
  \end{figure}

  \begin{proof}
    First we consider the case where $a$ is a Reeb chord of $\Lambda$.   
The boundary of the compactification of a one-dimensional moduli space ${\mathcal M}_{L_-}^1(a; b_1, \ldots, b_l)$ consists of two-level buildings as depicted in Figure \ref{fig:chainmap1};  namely:
\begin{itemize}
\item[(i)] a rigid holomorphic curve in $\mathcal M_{\Lambda}^1(a; b_1, \ldots, b_r)/\R$ followed by rigid holomorphic curves in ${\mathcal M}_{L^-}^0(b_i, b_{s_{i-1}+1}, \ldots, b_{s_i})$ for $1=s_0 < \ldots < s_r =l$, or
\item[(ii)] a rigid holomorphic curve in  ${\mathcal M}_{L^-}^0(a; b_1, \ldots, b_{r-1}, c, b_s, \ldots b_l)$ for $1 \le r \le s \le l+1$ (with the obvious meaning for the exceptional cases $r=1$ and $s=l+1$) followed by a rigid holomorphic curve in ${\mathcal M}_{L^-}^0(c; b_{r+1}, \ldots, b_{s-1})$ (with the obvious meaning of the exceptional case $r=s$). 
\end{itemize}
Buildings of type (i) contribute to $\Phi_{L^-} \circ \partial_{L^+}$, while buildings of type (ii) contribute to $\Phi_{L^-} \circ \partial_L$ because the double points of $L^-$ and the Reeb chords of $\Lambda^-$ generate a dg subalgebra of ${\mathcal D}_L$ by Lemma \ref{subalgebra}.

Now we consider the case where $a$ is a self-intersection point of $L^+$. If the almost complex structure on $X$ has a very long neck around $Y$, there is a bijection between
\begin{itemize}
\item rigid holomorphic polygons in $X$ with boundary on $L$, one positive end at a self-intersection point $a$ of $L^+$ and negative ends at double points of $L$ and Reeb chords of $\Lambda^-$, which contribute to $\partial_L a$, and 
\item  two-level holomorphic buildings consisting of a rigid holomorphic polygon in $X^+$ with boundary on $L^+$, a positive end at $a$ and negative ends at double points of $L^+$ and chords of $\Lambda$ followed by rigid polygons in $X^-$ with boundary on $L^-$, a positive end at a chord of $\Lambda$ and negative ends at double points of $L^-$ and Reeb chords of $\Lambda^-$, which contribute to $\Phi_{L^-}(\partial_{L^+} a)$.
\end{itemize}
\end{proof}

Though details are left for the reader, the  map $\Phi_-$ and the inclusion ${\mathcal D}_{L^-} \hookrightarrow {\mathcal D}_{L}$ are the maps used to define the composition of immersed DGA maps in \cite{MR4325415}, therefore showing that the cobordism algebra construction is functorial with respect to composition of cobordism. 
 \section{Floer homology for immersed Lagrangian cobordisms}
  \label{sec:cthulhu-complex-with}
  In this  section we introduce the main technical tool of the article; namely a Floer theory for immersed exact Lagrangian cobordisms in Liouville cobordisms which extends the theory defined in \cite{Floer_cob} in several directions. Our presentation will be rater sketchy, leaving the details to a future
work. 

 Let $L_0$ and $L_1$ be two immersed exact Legendrian cobordisms in $(X, \theta)$ such that
\begin{itemize}
\item intersection points between $L_0$ and $L_1$ are finite, transverse and distinct from the self-intersection points of $L_0$ and $L_1$, and
\item Reeb chords between $\Lambda^\pm_0$ and $\Lambda^\pm_1$ (in either direction) are nondegenerate,
\end{itemize} 
 and let ${\mathcal D}_{L_0}$ and ${\mathcal D}_{L_1}$ be their cobordism algebras. We define $\underline{\op{Cth}}(L_0, L_1)$ as the free ${\mathcal D}_{L_1} \mhyphen {\mathcal D}_{L_0}$ bimodule generated by Reeb chords from $\Lambda_1^+$ to $\Lambda_0^+$, intersection points between $L_0$ and $L_1$ and Reeb chords from $\Lambda_1^-$ to $\Lambda_0^-$. We can split it as a direct sum
  \begin{equation} \label{splitting of cth}
    \underline{\op{Cth}}(L_0, L_1) = \underline{\op{LCC}}(\Lambda^+_0, \Lambda^+_1) \oplus \underline{\op{CF}}(L_0, L_1) \oplus \underline{\op{LCC}}(\Lambda^-_0, \Lambda^-_1)
  \end{equation}  according to the nature of the generators. We define a differential
  $$\mathfrak{d} \colon  \underline{\op{Cth}}(L_0, L_1) \to  \underline{\op{Cth}}(L_0, L_1),
  \quad \mathfrak{d} = \left ( \begin{matrix} \mathfrak{d}_{++} & \mathfrak{d}_{+0} & \mathfrak{d}_{+-} \\ 0 & \mathfrak{d}_{00} & \mathfrak{d}_{0-} \\ 0 & \mathfrak{d}_{-0} & \mathfrak{d}_{--} \end{matrix} \right)$$
   satisfying the Leibniz rule by counting holomorphic curves as in \cite{Floer_cob}, or possibly anchored versions thereof. The matricial form of $\mathfrak{d}$ is written with respect to the direct sum decomposition \eqref{splitting of cth}. 

  Note that there are three main differences with the construction in \cite{Floer_cob}:
  \begin{itemize}
  \item[$(i)$] there the cobordisms are embedded, while here they are only immersed,
  \item[$(ii)$] there the differential goes from the negative end to the positive hand, while here it goes form the positive end to the negative end, and 
  \item[$(iii)$] there the pure chords at the negative ends are augmented, while here pure chords at the negative ends and self-intersection points are kept as coefficients.
  \end{itemize}

A consequence of $(ii)$ and of the dimension formulas of \cite[Section 3.2]{Floer_cob} is that $\mathfrak{d}$ has degree $-1$. Despite these differences, the proof that $\mathfrak{d}^2=0$ remains basically the same.
  We briefly describe the various components of the differential, referring to \cite{Floer_cob} for the precise definition of the moduli spaces involved. However, compared to \cite{Floer_cob}, here we number pure chords and self-intersection points in clockwise order. We also identify ordered sets of pure chords and self-intersection points with the corresponding word. 

  The map $\mathfrak{d}_{--}$ is defined as
  $$\mathfrak{d}_{--} a = \sum_{b, \bs{\eta}, \bs{\zeta}} \# \widetilde{\mathcal M}_{\Lambda_0^-, \Lambda_1^-}^1(a; \bs{\zeta}, b, \bs{\eta}) \bs{\zeta} b \bs{\eta}$$
  where $a,b$ are Reeb chords from $\Lambda_1^-$ to $\Lambda_0^-$, $\bs{\eta}= (\eta_1, \ldots, \eta_{\ell_0})$ and $\bs{\zeta} = (\zeta_1, \ldots, \zeta_{\ell_1})$ are (possibly empty) ordered sets of Reeb chords of $\Lambda^-_i$ for $i=0,1$ respectively, and the moduli spaces $\widetilde{\mathcal M}_{\Lambda_0^-, \Lambda_1^+}^1(a; \bs{\zeta}, b, \bs{\eta})$ are the zero-dimensional part of the moduli spaces defined in \cite[Section 3.1]{Floer_cob}.
Similarly the map $\mathfrak{d}_{++}$ is defined as
  $$\mathfrak{d}_{++}a = \sum_{b, \bs{\eta}, \bs{\zeta}} \# \widetilde{\mathcal M}_{\Lambda_0^+,\Lambda_1^+}^1(a; \bs{\zeta}, b, \bs{\eta})  \Phi_{L_0}(\bs{\zeta})  b \Phi_{L_1}(\bs{\eta}),$$
  where $a,b$ are Reeb chords from $\Lambda_1^+$ to $\Lambda_0^+$, $\bs{\eta}= (\eta_1, \ldots, \eta_{\ell_0})$ and $\bs{\zeta} = (\zeta_1, \ldots, \zeta_{\ell_1})$ are (possibly empty) ordered sets of Reeb chords of $\Lambda^+_i$ for $i=0,1$ respectively,
  and $\widetilde{\mathcal M}_{\Lambda_0^+, \Lambda_1^+}^1(a; \bs{\zeta}, b, \bs{\eta})$
  are the zero-dimensional part of the moduli spaces defined in \cite[Section 3.1]{Floer_cob}.

  The maps $\mathfrak{d}_{+,-}, \mathfrak{d}_{+0}, \mathfrak{d}_{00}$ and $\mathfrak{d}_{0-}$ are defined as
  $$\mathfrak{d}_{**} a =  \sum_{b, \bs{\eta}, \bs{\zeta}} \# {\mathcal M}_{L_0, L_1}^0(a; \bs{\zeta}, b, \bs{\eta}) \bs{\zeta} b \bs{\eta}$$
  where $a$ can be either a Reeb chord from $\Lambda^+_1$ to $\Lambda^+_0$ or an intersection point between $L_0$ and $L_1$, $b$ can be either an intersection point between $L_0$ and $L_1$ or a Reeb chord from $\Lambda_1^-$ to $\Lambda_0^-$, $\bs{\eta}= (\eta_1, \ldots, \eta_{\ell_0})$ and $\bs{\zeta} = (\zeta_1, \ldots, \zeta_{\ell_1})$ are (possibly empty) ordered sets of Reeb chords of $\Lambda^-_i$ and double points of $L_i$ for $i=0,1$ respectively, and the moduli spaces ${\mathcal M}_{\Lambda_0^-, \Lambda_1^+}^0(a; \bs{\zeta}, b, \bs{\eta})$ are the zero-dimensional part of the moduli spaces defined in \cite[Sections 3.2.2 -- 3.2.5]{Floer_cob}, where ``tentacles'' are allowed to converge to pure Reeb chords of $\Lambda_i^-$ and to self-intersection points of $L_i$, and when they converge to self-intersection points, they approach them in the negative direction.

  Finally the map $\mathfrak{d}_{-0}$ is defined as
  $$\mathfrak{d}_{-0}a =  \sum_{b, \bs{\eta}, \bs{\zeta}} \# {\mathcal M}_{L_0, L_1}^0(a; \bs{\zeta}, b, \bs{\eta}) \bs{\zeta} b \bs{\eta},$$
  where $a \in L_0 \cap L_2$, $b$ is a Reeb chord from $\Lambda_1^-$ to $\Lambda_0^-$, $\bs{\eta}= (\eta_1, \ldots, \eta_{\ell_0})$ and $\bs{\zeta} = (\zeta_1, \ldots, \zeta_{\ell_1})$ are (possibly empty) ordered sets of Reeb chords of $\Lambda^-_i$ and double points of $L_i$ for $i=0,1$ respectively, and the moduli spaces ${\mathcal M}_{L_0, L_1}^0(a; \bs{\zeta}, b, \bs{\eta})$ are the zero-dimensional part of the {\em Nessie} moduli spaces, which consist of two-level buildings consisting of a ``banana'' --- i.e.\ a holomorphic disc  in $(\R \times \partial_- X, d(e^s \alpha_-))$ with boundary in $\R \times (\Lambda_0^- \cup \Lambda_1^-)$ and asymptotic, as $s \to + \infty$, to $b$ and a Reeb chord $b'$ from $\Lambda^-_0$ to $\Lambda^-_1$ ---  followed by a ``neck'', i.e.\ a holomorphic curve in $X$ with bounary on $L_0 \cup L_1$ which is negatively asymptotic to $a$ and $b'$. See \cite[Sections 3.2.6 and 4.1.5]{Floer_cob}. Both the banana and the neck may have tentacles which, altogether, are asymptotic to pure cords or self-intersection points $(\eta_1, \ldots, \eta_{\ell_0})$ and $(\zeta_1, \ldots, \zeta_{\ell_1})$, and the self-intersection points are approached in the negative directrion. 
  
We define the action $\mathfrak{a}$ of the generators of the Cthulhu complex as follows:
\begin{itemize}
   \item if $a$ is a chord from $\Lambda_1^\pm$ to $\Lambda_0^\pm$, then  $\mathfrak{a}(a)=\int_a\alpha_\pm +c^\pm_1(\gamma)-c^\pm_0(\gamma),$
 where $c^\pm_i$ are the values taken by the potentials in the corresponding ends,
           \item if $a$ is an intersection point between $L_0$ and $L_1$, then $\mathfrak{a}(a)=h_1(a)-h_0(a)$.
           \end{itemize}
 The action of the vector space generators of $\underline{\op{Cth}}(L_0,L_1)$ is then defined as 
$$\mathfrak{a}(\zeta_1 \ldots \zeta_{\ell_1} a  \eta_1 \ldots \eta_{\ell_0}) = \mathfrak{a}(\zeta_1)+  \ldots + \mathfrak{a}(\zeta_{\ell_1}) + \mathfrak{a}(a) + \mathfrak{a}(\eta_1) + \ldots + \mathfrak{a}(\eta_{\ell_0}).$$
The proof of the following lemma is similar to that of lemma \ref{d_L decreases action} and therefore it will be omitted.
\begin{lemma}
  The differential $\mathfrak{d}$ preserves the filtration induced by $\mathfrak{a}$, i.e.\
  If $\bs{\zeta} b \bs{\eta}$ appear in $\mathfrak{d}(a)$, then $\mathfrak{a}(a)> \mathfrak{a}(\bs{\zeta} b \bs{\eta})$.
\end{lemma}
As a consequence, the sums defining the various components of $\mathfrak{d}$ are finite.

  Given dg modules $(V_0, \delta_0)$ and $(V_1, \delta_1)$ over the dg algebras ${\mathcal D}_{L_0}$ and ${\mathcal D}_{L_1}$ respectively, we define 
  \begin{equation}\label{eqn: Cth with representations}
    \op{Cth}_{V_0, V_1}^*(L_0, L_1)= \hom_{{\mathcal D}_{L_1}}(\underline{\op{Cth}}(L_0,L_1) \otimes_{{\mathcal D}_{L_0}} V_0, V_1).
  \end{equation}
  This group inherit a differential $D \colon \op{Cth}_{V_0, V_1}^*(L_0, L_1) \to \op{Cth}^*_{V_0, V_1}(L_0, L_1)$. We observe that an element $\varphi \in \op{Cth}^*_{V_0, V_1}(L_0, L_1)$ is determined by its value on the elements of the form $a \otimes_{{\mathcal D}_{L_0}} v$ where $a$ is a generator of $\underline{\op{Cth}}(L_0,L_1)$ and $v \in V_0$. Then $D$ is characterised by
  \begin{equation}\label{eqn: definition of D}
    (D \varphi)(a \otimes v)= \varphi(\mathfrak{d} a \otimes_{{\mathcal D}_{L_0}} v) + \varphi (a \otimes_{{\mathcal D}_{L_0}} \delta_0 v)  + \delta_1 \varphi(a_{{\mathcal D}_{L_0}} \otimes v).
  \end{equation}
If $\varepsilon_i \colon {\mathcal D}_{L_i} \to \mathbb{F}$, for $i=0,1$, are augmentations, we denote by $\mathbb{F}_{\varepsilon_i}$ the induced ${\mathcal D}_{L_i}$-modules with underlying vector space $\mathbb{F}$ and trivial differential. Then we will often write 
$\op{Cth}_{\varepsilon_0, \varepsilon_1}^*(L_0, L_1)$ instead of  $\op{Cth}_{\mathbb{F}_{\varepsilon_0}, \mathbb{F}_{\varepsilon_1}}^*(L_0, L_1)$. These are the groups which were  defined in \cite{Floer_cob} in the case of embedded exact Lagrangian cobordisms.

  The definition of $\op{Cth}_{V_0, V_1}^*(L_0, L_1)$ is, at least partially, cohomological. We will occasionally use also the groups
 $$\op{Cth}_{V_0, V_1}(L_0, L_1) = V_1 \otimes_{{\mathcal D}_{L_1}} \underline{\op{Cth}}(L_0,L_1) \otimes_{{\mathcal D}_{L_0}} V_0$$
 arising from a completely homological construction. The two construction are distinguished in the notation by the presence, or absence, of a star.
\begin{rem}
  As for the cobordism algebra, the Cthulhu complex could also be defined over a an idempotent ring, but only the version over $\F$ will be used in this article.
   \end{rem}

\section{A relative exact triangle for the concatenation of cobordisms}\label{sec: relative exact sequence}
Here we describe a relative exact triangle for the Floer homology of  a pair of immersed Lagrangian cobordisms when they are split along a  hypersurface of contact type. Legout has constructed an analogous Mayer-Vietoris sequence in \cite{legout2020ainfinity}. We recall some notation we introduced in Section \ref{sec: immersed cobordisms}. Let $(X, \theta)$ be a Liouville cobordism and let $Y \subset X$ be a separating hypersurface of contact type. In a neighbourhood of $Y$ of the form $(- \epsilon, \epsilon) \times Y$, where $Y$ is identified with $\{0\}\times Y$, we can write $\theta= e^s \alpha_Y$ where $\alpha_Y$ is a contact form on $Y$.  Let $(X^+, \theta_+)$ and $(X^-, \theta_-)$ be the completions of the connected components of $X \setminus Y$. If $J$ is an almost complex structure on $X$ which is compatible with $\theta$ and cyindrical on $(- \epsilon, \epsilon) \times Y$, we denote by  $J^+$ and $J^-$ the induced almost complex structures on $X^+$ and $X^-$ respectively. If an immersed exact Lagrangian cobordism $L \subset X$ is  nicely split by $Y$ (see Definition \ref{dfn: nicely split}) we denote by $L^+$ and $L^-$ the Liouville completions of the connected components of $L \setminus Y$.

\begin{rem}
 Like in the two previous sections, in this one too we will work over $\F$. An extension to idempotent rings would be possible but unnecessary.
\end{rem}
Lemmas \ref{subalgebra} and \ref{lemma: mercoledi} have the following corollary.
\begin{cor}\label{induced augmentations} Let $L$ be an immersed exact Lagrangian cobordism which is nicely split by a contact type hypersurface. An augmentation $\varepsilon_L \colon {\mathcal D}_L \to \F$ induces an augmentation $\varepsilon_{L^-} \colon {\mathcal D}_{L^-} \to \mathbb{F}$ by restriction, and an augmentation $\varepsilon_{L^+} \colon {\mathcal D}_{L^+} \to \mathbb{F}$ by $$\varepsilon_{L^+} = \varepsilon_L \circ \Phi_{L^-},$$
i.e.~the pull-back under the cobordism dg-morphism.
\end{cor}

The main result of this section is the following relative exact triangle.
\begin{thm}\label{relative exact triangle}
  Let $L_0$ and $L_1$ be immersed exact lagrangian cobordisms which are nicely split by $Y$, let $\varepsilon_{L_i} \colon {\mathcal D}_{L_i} \to \mathbb{F}$, $i=0,1$, be augmentations of their cobordism algebras, and   $\varepsilon_{L_i^\pm} \colon {\mathcal D}_{L^\pm_i} \to \mathbb{F}$ the augmentations of ${\mathcal D}_{L_i^\pm}$ from Corollary \ref{induced augmentations}. If all the intersection points between $L_0^+$ and $L_1^+$ have positive action,  then there is an exact triangle
  $$\xymatrix{
    H\op{Cth}^*_{\varepsilon_{L_0^+}, \varepsilon_{L_1^+}}(L_0^+, L_1^+) \ar[r]  & H\op{Cth}^*_{\varepsilon_{L_0}, \varepsilon_{L_1}}(L_0, L_1)  \ar[d] \\
    & H\op{Cth}^*_{\varepsilon_{L_0^-}, \varepsilon_{L_1^-}}(L_0^-, L_1^-).  \ar[ul] &
    }$$
  \end{thm}

When the cobordisms are embedded and their negative ends are Lagrangian fillable, this triangle is precisely the triangle of \cite[Section 8.3]{CieliebakOancea}  involving a pair of filled Lagrangian cobordisms.
  
  \begin{rem}
    The condition about the action of the intersection points can be obtained by a compactly supported hamiltonian isotopy of $L_0$, and therefore is unnecessary. However, removing it would require more invariance of Cthulhu homology than we have proved so far.
  \end{rem}
  The proof of Theorem \ref{relative exact triangle} will occupy the rest of this section.
 Since there is a canonical bijection between $L_0 \cap L_1$ and $(L_0^- \cap L_1^-) \cup (L_0^+ \cap L_1^+)$, we have an identification
$$CF(L_0, L_1) = CF(L_0^+, L_1^+) \oplus CF(L_0^-, L_1^-)$$
as vector spaces, and therefore we can write
$$\op{Cth}_{\varepsilon_{L_0}, \varepsilon_{L_1}}^*(L_0, L_1)= LCC(\Lambda^+_0, \Lambda^+_1) \oplus CF(L_0^+, L_1^+) \oplus CF(L_0^-, L_1^-) \oplus LCC(\Lambda^-_0, \Lambda^-_1).$$
We define
\begin{align*}
  \op{Cth}^+_{\varepsilon_{L_0}, \varepsilon_{L_1}}(L_0, L_1) & = LCC(\Lambda^+_0, \Lambda^+_1) \oplus CF(L_0^+, L_1^+), \\ \op{Cth}^-_{\varepsilon_{L_0}, \varepsilon_{L_1}}(L_0, L_1)  & = CF(L_0^-, L_1^-) \oplus LCC(\Lambda^-_0, \Lambda^-_1).
\end{align*}
We denote by $d_{**}$ and $d^\pm_{**}$, with $* \in \{ +.-,0 \}$, the components of the differentials of $\op{Cth}_{\varepsilon_{L_0}, \varepsilon_{L_1}}^*(L_0, L_1)$ and $\op{Cth}_{\varepsilon_{L^\pm_0}, \varepsilon_{L^\pm_1}}^*(L^\pm_0, L_1^\pm)$ respectively.

For all $R \ge 0$ we define actions $\mathfrak{a}_R$ of generators of $\op{Cth}_{\varepsilon_{L_0}, \varepsilon_{L_1}}^*(L_0, L_1)$ using a stretched Liouville forms $\theta_R$ as in Section \ref{sec: immersed cobordisms}. The maps $d_{**}$ increase the action $\mathfrak{a}^R$ for all $R$ because the differential of  $\op{Cth}_{\varepsilon_{L_0}, \varepsilon_{L_1}}(L_0, L_1)$ is cohomological as in \cite{Floer_cob}.

\begin{lemma}\label{no nessie}
  If $a \in L^+_0 \cap L_1^+$, then $d_{-0}(a)=0$ and $d^+_{-0}(a)=0$
\end{lemma}
\begin{proof}
  We will prove that there is no nessie between an intersection point of $L^+_0 \cap L_1^+$ and a Reeb chord from $\Lambda_1^-$ to $\Lambda_0^-$ or from $\Lambda_0$ and $\Lambda_1$. We first consider nessies in $X$. Let $v \colon D \to X$ be the neck of a Nessie connecting an intersection point $a \in L^+_0 \cap L_1^+$ and a Reeb chord $b'$ from $\Lambda_0^-$ to $\Lambda_1^+$; then
  $$0 < \int_D d \theta_R= c_1^- - c_0^- - \mathfrak{a}_R(a),$$
  where $c_0^-$ and $c_1^-$ are the value of the potentials of $L_0$ and $L_1$ at the negative end. Since $\mathfrak{a}_R(a) = e^R \mathfrak{a}$ and $ \mathfrak{a}>0$, we obtain a contradiction for $R$ large enough.

  The proof that there are no nessies with boundary on $L_0^+$ and $L_1^+$ is similar but simpler: since the potentials of $L_0^+$ and $L_1^+$ vanish at the negative end, it is enough to integrate $d\theta$. 
\end{proof}

\begin{lemma}\label{Cth^+ is a subcomplex}
 $\op{Cth}^+_{\varepsilon_{L_0}^+, \varepsilon_{L_1}^+}(L_0, L_1)$ is a subcomplex of $\op{Cth}_{\varepsilon_{L_0}, \varepsilon_{L_1}}^*(L_0, L_1)$.
\end{lemma}
\begin{proof}
 The component of the differential that could map an element of $\op{Cth}^+_{\varepsilon_{L_0}, \varepsilon_{L_1}}(L_0, L_1)$ out of that group are $d_{-0}$ and $d_{00}$. By Lemma \ref{no nessie} $d_{-0}(a)=0$ if $a \in L_0^+ \cap L_1^+$, so it remains to prove that $d_{00}(a) \in CF(L_0^+, L_1^+)$ whenever $a \in L_0^+ \cap L_1^+$. We recall that the input of $d_{00}$ is at the negative end and the output at the positive end as in \cite{Floer_cob}, and unlike the map $\mathfrak{d}_{00}$ because of the cohomological nature of $\op{Cth}_{\varepsilon_{L_0}, \varepsilon_{L_1}}^*(L_0, L_1)$. Therefore, if $d_{00}(a)$ does not belong to $CF(L_0^+, L_1^+)$, then there is a holomorphic map with a negative end at $a \in L_0^+ \cap L_1^+$, the positive end at some $b \in CF(L_0^-, L_1^-)$, and possibly other negative ends at pure chords of $\Lambda_0^-$ and $\Lambda_1^-$ and self-intersection points of $L_0$ and $L_1$. Then by integrating the pull-back of the form $d \theta_R$ by that holomorphic map  we obtain $\mathfrak{a}_R(a)<\mathfrak{a}_R(b)$, i.e.\ $e^R \mathfrak{a}(a)< \mathfrak{a}(b)$, for all $R \ge 0$. This is clearly a contradiction because  
 $\mathfrak{a}(a)>0$ and there are only finitely many intersection points between $L^-_0$ and $L_1^-$.
\end{proof}

\begin{lemma} \label{what is born below remains below}
  The holomorphic curves in $X$ contributing to $d_{-0}$, $d_{00}$ and $d_{0-}$ relating generators of $\op{Cth}_{\varepsilon_{L_0}, \varepsilon_{L_1}}^-(L_0, L_1)$ are in bijection with the holomorphic curves in $X^-$ contributing to $d_{-0}^-$, $d_{00}^-$ and $d_{0-}^-$.
\end{lemma}
\begin{proof}
  Holomorphic curves in $X$ or $X^-$ which contribute to the corresponding Cthulhu differential and have all ends below $Y$ are completely contained in the connected component of $X$ or $X^-$ below $Y$ by an action argument. In fact, if there is a portion of such a curve above $Y$, then its $d \theta_R$-energy grows exponentially with $R$, while the $\mathfrak{a}_R$ action of the ends remains constant. This leads to a contradiction for $R$ sufficiently large.
\end{proof}
$\op{Cth}_{\varepsilon_{L_0}, \varepsilon_{L_1}}^-(L_0, L_1)$ can be viewed as a chain complex with differential induced by the identification with the quotient complex
$$\op{Cth}_{\varepsilon_{L_0}, \varepsilon_{L_1}}^*(L_0, L_1) / \op{Cth}_{\varepsilon_{L_0}, \varepsilon_{L_1}}^-(L_0, L_1).$$
Lemma \ref{what is born below remains below} shows that it can be identified also to a quotient complex of  $\op{Cth}^*_{\varepsilon_{L_0^-}, \varepsilon_{L_1}^-}(L_0^-, L_1^-)$ because $\varepsilon_{L_0^-}$ and  $\varepsilon_{L_1^-}$ are restrictions of $\varepsilon_{L_0}$ and  $\varepsilon_{L_1}$ respectively by lemma \ref{subalgebra}.

In the next lemma we stretch the neck to relate the components of the differential of $\op{Cth}^*_{\varepsilon_{L_0}, \varepsilon_{L_1}}(L_0, L_1)$ to the components of the differential of  $\op{Cth}^*_{\varepsilon_{L_0^\pm}, \varepsilon_{L_1^\pm}}(L_0^\pm, L_1^\pm)$.
 \begin{lemma}\label{Cth for a split J}
   The chain complex $\op{Cth}^*_{\varepsilon_{L_0}, \varepsilon_{L_1}}(L_0, L_1)$ is homotopic to the chain complex with the same underlying vector space, i.e.\
   $$LCC(\Lambda^+_0, \Lambda^+_1) \oplus CF(L_0^+, L_1^+) \oplus CF(L_0^-, L_1^-) \oplus LCC(\Lambda^-_0, \Lambda^-_1)$$
   and differential
   \begin{equation} \label{big differential}
     \left ( \begin{matrix} d^+_{++} & d^+_{+0} & d^+_{+-} \circ d^-_{+0} & d^+_{+-} \circ d^-_{+-} \\
       0 & d^+_{00} & d^+_{0-} \circ d^-_{+0} & d^+_{0-} \circ d^-_{+-} \\
       0 & 0 & d^-_{00} & d^-_{0-} \\
       0 & 0 & d^-_{-0} & d^-_{--}
       \end{matrix} \right ).
   \end{equation}
 \end{lemma}
 \begin{proof}
   We stretch the neck along $Y$ and analyse how the holomorphic maps contributing to the differential of $\op{Cth}^*_{\varepsilon_{L_0^\pm}, \varepsilon_{L_1^\pm}}(L_0^\pm, L_1^\pm)$ degenerate. The count of isolated holomorphic maps in the one-parameter family of stretching almosty complex structures provides a chain homotopy between the original differential of  $\op{Cth}^*_{\varepsilon_{L_0}, \varepsilon_{L_1}}(L_0, L_1)$ and the differential defined by the completely stretched almost complex structure. We will prove that the latter is described by Equation \eqref{big differential}.
   
   The first column of zeros is a consequence of the form of the Cthulhu differential. The zeros of the second column are a consequence of Lemma \ref{Cth^+ is a subcomplex}. The submatrix $\left ( \begin{matrix} d^-_{00} & d^-_{0-} \\ d^-_{-0} & d^-_{--} \end{matrix} \right )$ is a consequence of Lemma \ref{what is born below remains below}. Observe that Lemmas  \ref{Cth^+ is a subcomplex} and \ref{what is born below remains below} hold for every almost complex structure which is cylindrical near $Y$, and in particular for every almost complex structure in the chosen stretrching one-parameter family.

   To obtain the submatrix $\left ( \begin{matrix} d^+_{++} & d^+_{+0} \\ 0 & d^+_{00} \end{matrix} \right )$ we observe that by the proof of Lemma \ref{Cth^+ is a subcomplex} there is no neck of a Nessie in $X^+$ with boundary on $L_0^+$ and $L_1^+$, and therefore curves contributing to the differential of  $\op{Cth}^*_{\varepsilon_{L_0}, \varepsilon_{L_1}}(L_0, L_1)$ can only degenerate into two-level buildings, one of whose levels is a curve contributing to $d^+_{**}$, and the other one consists of curves with  pure boundary components (i.e.\ only on $L^-_0$ or $L^-_1$). These curves contribute to the dg morphisms $\Phi_{L_0^-}$ and $\Phi_{L_1^-}$ from  Lemma \ref{dga morphisms induced by cobordisms}. The proof of the submatrix
   $\left ( \begin{matrix} d^+_{+-} \circ d^-_{+0} & d^+_{+-} \circ d^-_{+-} \\ d^+_{0-} \circ d^-_{+0} & d^+_{0-} \circ d^-_{+0} \end{matrix} \right )$ is similar, with the only difference that now the bottom level of the limit building has a component with a positive end at a negative chord from $\Lambda_1$ to $\Lambda_0$ and the top level has a component with a negative end at the same chord.
\end{proof}
We will show that $\op{Cth}^*_{\varepsilon_{L_0}, \varepsilon_{L_1}}(L_0, L_1)$ is quasi-isomorphic to a mapping cone between $\op{Cth}^*_{\varepsilon_{L_0^+}, \varepsilon_{L_1^+}}(L_0^+, L_1^+)$ and $\op{Cth}^*_{\varepsilon_{L_0^-}, \varepsilon_{L_1^-}}(L_0^-, L_1^-)$ with the aid of the following algebraic lemma.
  \begin{lemma}\label{boring algebraic lemma}
 Let $A$, $B$, $C$ be chain complexes, $\nu \colon A \to C$, $\mu \colon C \to B$ chain maps and $\eta \colon \op{Cone}(\mu) \to \op{Cone}(\nu)$ the chain map define as $\op{id} \colon C \to C$ and zero on all other components. Then $\op{Cone}(\eta)$ is quasi-isomorphic to $\op{Cone}(\mu \circ \nu)$.
  \end{lemma}
  \begin{proof}
    Let $F \colon \op{Cone}(\eta) \to \op{Cone}(\mu \circ \nu)$ be the map described by the following diagram
    $$\xymatrix{
      A \ar[r]^-\nu \ar[d]_{\op{id}} & C \ar[drr]^-\mu & C \ar[l]_{\op{id}} \ar[dr]^-\mu \ar[r]^-\mu & B \ar[d]^-{\op{id}} \\
      A \ar[rrr]^-{\mu \circ \nu} & & & B.
    }$$
    It is easy to verify that $F$ is a chain map. We introduce the following length three filtrations on $\op{Cone}(\eta)$ and $\op{Cone}(\mu \circ \nu)$:
    $$F^0\op{Cone}(\eta) = B, \quad F^1\op{Cone}(\eta) = C \oplus C \oplus B, \quad
    F^2\op{Cone}(\eta) = A \oplus C \oplus C \oplus B,$$
    $$F^0\op{Cone}(\mu \circ \nu) = F^1\op{Cone}(\mu \circ \nu) = B, \quad F^2\op{Cone}(\mu \circ \nu) = A \oplus B.$$
    The map $F$ induces an isomorphism between the homologies of the associated graded complexes, and therefore it is a quasi-isomorphism.
  \end{proof}
  \begin{proof}[Proof of Theorem \ref{relative exact triangle}]
    We apply Lemma \ref{boring algebraic lemma} to
    $$A= LCC(\Lambda_0^-, \Lambda_1^-) \oplus CF(L_0^-, L_1^-)$$
    with differential $\left ( \begin{matrix} d^-_{00} & d^-_{0-} \\ d^-_{-0} & d^-_{--} \end{matrix} \right )$,
    $$B= CF(L_0^+, L_1^+) \oplus LCC(\Lambda_0^+, \Lambda_1^+)$$
    with differential  $\left ( \begin{matrix} d^+_{++} & d^+_{+0} \\ 0 & d^+_{00} \end{matrix} \right )$,
    $$C = LCC(\Lambda_0, \Lambda_1),$$
    $\mu = d^+_{0-} \oplus d^+_{+-}$ and $\nu= d^-_{+-}+d^-_{+0}$. Since
    $\op{Cone}(\mu \circ \nu)$ is the chain complex of Lemma \ref{Cth for a split J}, $\op{Cone}(\mu)= \op{Cth}^*_{\varepsilon{L_0^+}, \varepsilon_{L_1^+}}(L_0^+, L_1^+)$ and $\op{Cone}(\nu) = \op{Cth}^*_{\varepsilon_{L_0^-}, \varepsilon_{L_1^-}}(L_0^-, L_1^-)$, then 
$\op{Cth}^*_{\varepsilon_{L_0}, \varepsilon_{L_1}}(L_0, L_1)$
    is quasi-isomorphic to the cone of a chain map
    $$\op{Cth}^*_{\varepsilon_{L_0^-}, \varepsilon_{L_1^-}}(L_0^-, L_1^-) \to  \op{Cth}^*_{\varepsilon_{L_0^+}, \varepsilon_{L_1^+}}(L_0^+, L_1^+).$$
    This implies Theorem \ref{relative exact triangle}.
  \end{proof}
  
\section{The geometric construction}
\label{sec:geom-constr}
Let  $(W,d\theta)$ be a Weinstein domain which is obtained by critical Weinstein handle-attachments on a subcritical Weinstein sub-domain $(W^{sc},d\theta)$ along a Legendrian link $\mathbf{S}$ in $\partial W^{sc}$
which consists of a finite number of Legendrian spheres $S_\sigma$, $\sigma \in \pi_0(\mathbf{S})$. We define $W^{crit}=W\setminus \mathrm{int} W^{sc}$. This is a compact Weinstein cobordism from $\partial_- W^{crit} = \partial W^{sc}$ to $\partial_+ W^{crit} = \partial W$.  Each Legendrian sphere $S_\sigma$ is the boundary of a Lagrangian disc $C_\sigma$ in $W^{crit}$ : the core of the Weinstein handle attached to $S_\sigma$.

In this section we will prove that every closed Exact Lagrangian submanifold $L$ of $(W, \theta)$ can be deformed in a controlled way into an immersed exact Lagrangian $\overline{L}$ which is nicely split by $\partial W^{sc}$ into an immersed exact Lagrangian  filling (i.e.\ a cobordism with empty negative end) in $W^{sc}$ and an immersed exact Lagrangian cap (i.e.\ a cobordism with empty positive end) in $W^{crit}$. The deformation will use some contact topological notions that we are going to recall briefly.

The {\em contactisation} of $(W, \theta)$ is the contact manifold $(W \times \R, dz+\theta)$, where $z$ is the coordinate on $\R$.
If $\iota \colon L \to W$ is an  exact Lagrangian immersion with potentiual function $h$, we  define a Legendrian immersion $\iota^\dagger \colon L \to W \times \R$ by $\iota^\dagger(x)=(\iota(x), - h(x))$. Since the potential function is well defined only up to constant, the Legendrian immersion $\iota^\dagger$ is well defined only up to translations in the $\R$ direction of the contactisation. Even if this indeterminacy can be a concern if $L$ is not connected, it will have a minimal impact on our constructions. Note that $\iota^\dagger$ is an embedding if and only if for every pair of points $x, y \in L$ such that $\iota(x)=\iota(y)$ we have $h(x) \ne h(y)$. This is a generic condition in the regular homotopy class of $\iota$. In order to keep the exposition simple, we will {\em always} use $L$ to denote $\iota(L)$ and $L^\dagger$ to denote $\iota^\dagger(L)$. When $L^\dagger$ is embedded, the double points of $L$ are in bijection with the Reeb chords $L^\dagger$, and the length of a Reeb chords is the absolute value of the difference between the values of the potential function at its endpoints.
Finally, any regular exact Larangian homotopy $\iota_t \colon L \to W$ can be lifted to a regular Legendrian homotopy $\iota_t^\dagger \colon L \to W \times \R$. On the other hand every Legendrian submanifold of $\R \times W$ projects to a Lagrangian immersion in $W$ which is called the {\em Lagrangian projection} and, similarly, every Legendrian isotopy projects to a regular exact Lagrangian homotopy.

The precise statement that we will show in this section is the following.

\begin{thm}
\label{thm:maingeometric}
  For every closed exact Lagrangian submanifold $L$ of $W$ there exists an immersed exact
  Lagrangian submanifold $\overline{L}$ in $W$ with only transverse double points  and a potential function $\overline{h} \colon \overline{L} \to \R$ such that:
\begin{enumerate}
\item the Legendrian lift $\overline{L}^\dagger$ of $\overline{L}$ is embedded and
  Legendrian isotopic to the Legendrian lift $L^\dagger$ of $L$,
\item $\overline{L} \cap \partial W^{sc}$ is an embedded Legendrian link $\bs{\Lambda}$ which is Legendrian isotopic to a Legendrian link consisting of parallel copies of attaching spheres pushed off along the Reeb flow,
\item $\overline{L} \cap W^{crit}$ consists of several $C^2$-close Hamiltonian isotopic
  copies of  Lagrangian cores such that copies of different cores are disjoint and copies of the same disc pairwise intersect in a single point,
\item $\overline{L}^\dagger$ and $\bs{\Lambda}$ are chord generic, 
\item all Reeb chords   of $\overline{L}^\dagger$ corresponding to the double points of $\overline{L} \cap W^{crit}$ are longer than all Reeb chords  of $\overline{L}^\dagger$ corresponding to the double points of $\overline{L} \cap W^{sc}$,  and
\item the potential function $\overline{h}$ vanishes in a neighbourhood of $\overline{L} \cap \partial W^{sc}$.
\end{enumerate}
\end{thm}

\begin{rem}
  The number of copies of the core disc $C_\sigma$ that we get in the construction is equal to the geometric intersection number of $L$ and a generic perturbation of $C_\sigma$. Unfortunately we know very little about  these numbers in general. They are bounded below by the rank of the Floer homology of $L$ with the corresponding cocore, and there must be at least one intersection point with one cocore disc, since otherwise  $L \subset W^{sc}$, but subcritical Weinstein manifolds cannot contain any closed exact Lagrangian submanifold. 
\end{rem}

Before proving Theorem \ref{thm:maingeometric}, we need to find convenient coordinates on which we will perform our construction. First we introduce an useful exact symplectomorphism between the symplectisation of a jet space and a cotangent bundle.
Let $N$ denote any closed manifold, and let $\theta_N=-\mathbf{p}d \mathbf{q}$ be the canonical Liouville form in $T^*N$. We consider the Liouville form $e^t(dz+\theta_N)$ on $\R \times T^*N \times \R$, where $t$ is the coordinate on the first $\R$-factor and $z$ is the coordinate on the second $\R$-factor.  We consider also the cotangent space $T^*(\R^+ \times N)$ with the canonical Liouville form $\theta_{\R^+ \times N}=-p_sds + \theta_N$, where $s$ is the coordinate of the $\R$-factor in $\R \times N$ and $p_s$ is its conjugate momentum. If we define the map
$$G \colon \R \times T^*N \times \R \to T^*\R^+ \times T^*N \cong T^*(\R^+ \times N)$$
by $G(t, \xi, z) = (e^t, z, e^t\xi)$, then 
  \begin{equation}\label{eqn: prurito}
    G^*(\theta_{\R^+ \times N} + d(s p_s))= G^*(s dp_s + \theta_N) = e^t(dz+\theta_N).
  \end{equation}

Since the modifications in Theorem \ref{thm:maingeometric} happen near the Lagrangian core discs independently of one another, we focus our attention to a single Lagrangian core disc $C \subset W^{crit}$. Let $S = \partial C \subset \partial_- W^{crit}$ be the corresponding attaching sphere. Since the Liouville flow $\psi_t$ preserves $C$, we can take also a larger disc $\overline{C} \supset C$ such that a collar $\overline{C}^\delta$ of $\partial \overline{C}$ is parametrised via
\begin{equation} \label{eqn: collar}
[-\delta,\delta] \times S \hookrightarrow C, \quad  (t,\widetilde{\mathbf{q}}) \mapsto \psi_t(\widetilde{\mathbf{q}}).
\end{equation}
We give also an alternative parametrisation of the same collar by
\begin{equation}\label{eqn: collar 2}
  [e^{-\delta}, e^\delta] \times S \ni (s, \widetilde{\mathbf{q}})= (e^t, \widetilde{\mathbf{q}}).
\end{equation} 

Let $D_\epsilon^*\overline{C}$ be the open disc cotangent bundle of $\overline{C}$ of radius $\epsilon$ with respect to some fixed Riemannian metric on $\overline{C}$, and let $\theta_{\overline{C}}= - \mathbf{p} d \mathbf{q}$ be the canonical Liouville form on $D_\epsilon^*\overline{C}$. The parametrisation \eqref{eqn: collar 2} induces an identification of  $D_\epsilon^*\overline{C}^\delta$ with $D_\epsilon^*([e^{- \delta},e^\delta] \times S)$ and of $\theta_{\overline{C}}|_{D_\epsilon^*\overline{C}^\delta}$ with $-p_s ds + \theta_S$, where $(s, p_s)$ are the canonical coordinates on $T^*[e^{- \delta}, e^\delta]$ and $\theta_{S}$ is the canonical Liouville form of $T^*S$.

\begin{lemma}\label{nice coordinates}
  There is a neighbourhood $V$ of $\overline{C}$ in $W$ such that
  \begin{enumerate}
  \item there is a diffeomorphism $F \colon D_\epsilon^*\overline{C} \to V$ which identifies the zero section of $D_\epsilon^*\overline{C}$ with $\overline{C}$, and 
  \item $F^*\theta = \theta_{\overline{C}}+dg$, where $g \colon D_\epsilon^* \overline{C} \to \R$ is a smooth function  such that $g = s p_s$ in $D^*_\epsilon \overline{C}^\delta$.
  \end{enumerate}
\end{lemma}

\begin{proof}
  We use the standard Legendrian neighbourhood theorem to identify a neighbourhood of $S$ in $\partial W^{sc}$ with $D^*_\epsilon S \times [- \epsilon, \epsilon]$ and the contact form on $\partial W^{sc}$ with $dz+ \theta_S$. Flowing this neighbourhood using the Liouville flow, we  identify a neighbourhood of $S$ in $W$ with $[-\delta, \delta] \times D^*_\epsilon S \times [- \epsilon, \epsilon]$ so that the collar $\overline{C}^\delta$ corresponds to $[- \delta, \delta] \times \bs{0}_S \times \{ 0 \}$ (where $\bs{0}_S$ is the zero section of $T^*S$) and  the Liouville form $\theta$ is written as $d(e^t(dz+\theta_S))$.

  Using Equation \eqref{eqn: prurito} we can identify $[-\delta, \delta] \times D^*_\epsilon S \times [- \epsilon, \epsilon]$ with a neighbourhood of the zero section of $T^*([e^{- \delta}, e^\delta] \times S)$. The inverse of this map (restricted to its image) can be extended to a symplectic embedding
  $$F \colon D_\epsilon^*\overline{C} \to W.$$
  Since $\overline{C}$ is simply connected there is a function $g \colon \overline{C} \to \R$ such that 
$F^* \theta  = \theta_C + dg$, and by Equation \eqref{eqn: prurito} we can assume that 
$g=sp_s$ in $\overline{C}^\delta$.
  \end{proof}

  The symplectomorphism $F$ identifies exact Lagrangian immersions in $(V, \theta|_V)$ with exact Lagrangian immersions in $(D^*_\epsilon \overline{C}, \theta_{\overline{C}})$ because $F^* \theta - \theta_C$ is exact. Moreover $F$ can be lifted to a strict contactomorphism
 \begin{align}\label{eqn: contactomorphism}
   F^\dagger \colon (\R \times D^*_\epsilon \overline{C}, dz+\theta_{\overline{C}}) & \to (\R \times V, dz+\theta) \\ \nonumber (\xi, z) & \mapsto (F(\xi), z-g(\xi))
 \end{align}
 which identifies Legendrian submanifolds of $\R \times V$ with Legendrian submanifolds of $\R \times D^*_\epsilon \overline{C}$. The latter have the advantage that they can be described by their {\em front}, i.e. the image of the projection to $\overline{C} \times \R$.
  If $f \colon \overline{C} \to \R$ is a smooth function, its graph in $\overline{C} \times \R$ is the simplest example of a front. The corresponding Legendrian submanifold is the graph of  the $1$-jet
  $$\Gamma(j^1f )= \{(x, d_xf, f(x)) : x \in \overline{C} \}.$$
  The next lemma gives a condition for the Lagrangian projection of the $1$-jet of a function to be cylindrical with respect to the Liouville form $F^*\theta$ in $D_\epsilon^*\overline{C}^\delta$.
  \begin{lemma}\label{cylindrical jets}
    The Lagrangian projection of the $1$-jet of $f \colon \overline{C}^\delta \to \R$ to $D_\epsilon^*\overline{C}^\delta$, i.e. the graph $df$ seen as a map $df \colon \overline{C}^\delta \to T^*\overline{C}^\delta$,  is cylindrical for the Liouville form $F^*\theta$ if and only if
    \begin{equation}\label{tgv}
      f(s, \tilde{\mathbf{q}})= s \tilde{f}(\tilde{\mathbf{q}}) +c
    \end{equation}
    holds there for some function $\tilde{f} \colon S \to \R$ and some constant $c \in \R$.
  \end{lemma}
  \begin{proof}
 The graph of $df$ is cylindrical if and only if $(df)^*\theta=0$. From Lemma \ref{nice coordinates} we have
    $$(df)^*\theta=-df+d\left (s\frac{\partial f}{\partial s} \right )= \left(-\frac{\partial f}{\partial \tilde{\mathbf{q}}}+s\frac{\partial ^2 f}{\partial s\partial \tilde{\mathbf{q}}}\right)d \tilde{\mathbf{q}}+s\frac{\partial ^2 f}{\partial s^2}ds.$$

    Therefore the graph of $df$ is cylincrical if and only if 
    $$\begin{cases} \dfrac{\partial ^2 f}{\partial s^2} =0,  \vspace{2mm} \\
      \dfrac{\partial f}{\partial \tilde{\mathbf{q}}} = s\dfrac{\partial ^2 f}{\partial s\partial \tilde{\mathbf{q}}}.
    \end{cases}$$
    The first equation implies that $f(s, \tilde{\mathbf{q}})= s \tilde{f}(\tilde{\mathbf{q}}) + \tilde{g}(\tilde{\mathbf{q}})$ and the second one implies that $\dfrac{\partial \tilde{g}}{\partial \tilde{\mathbf{q}}}=0$.
\end{proof}
We can identify $\{s=1 \} = D_\epsilon^*\overline{C}^\delta|_{\mathbf{S}}$ with a neighbourhood of the zero section of the one-jet space $T^* \mathbf{S} \times \R$ with local coordinates $\tilde{\mathbf{q}}, \tilde{\mathbf{p}}$ on $\mathbf{S}$ and global coordinate $p_s$ on $\R$ and the restriction of $F^*\theta$ to $D_\epsilon^*\overline{C}^\delta|_{\mathbf{S}}$ with the form $dp_s - \theta_{\mathbf{s}}$, i.e.\ the canonical form on $J^1 \mathbf{S}$. Therefore, if $f \colon \overline{C}^\delta \to \R$ is of the form $f(s, \tilde{\mathbf{q}})= s \tilde{f}(\mathbf{q}) +c$ and $\Gamma(j^1f)$ is contained in $D^*_\epsilon\overline{C}^\delta$, then the intersection of the Lagrangian projection of $\Gamma(j^1f)$ with $D_\epsilon^*\overline{C}^\delta|_{\mathbf{S}} \cong J^1 \mathbf{S}$ is the Legendrian submanifold $\Gamma(j^1 \tilde{f})$. 
  \begin{rem}\label{rem_potential}
 Define $V^\delta=F(D^*_\epsilon\overline{C}^\delta)$.
From Equation \eqref{eqn: contactomorphism} it follows that the Legendrian submanifold of  $(V \times \R, dz+\theta)$ that corresponds to $j^1f$ for $f \colon \overline{C} \to \R$ as in Equation \eqref{tgv} under the contactomorphism $F^\dagger$ has $z$--coordinate constantly equal to $c$ in $V^\delta \times \R$.
  \end{rem}
  \begin{proof}[Proof of Theorem \ref{thm:maingeometric}]
    For each core disc $C_i$ we fix enlargements $\overline{C}_i$, neighbourhoods $V_i \subset W$ of $\overline{C}_i$ and symplectomorphisms $F_i \colon D^*_\epsilon \overline{C}_i \to V_i$ as in Lemma \ref{nice coordinates}. After applying the negative Liouville flow to $L$ we can assume that $L \cap W^{crit}$ is contained in the union of the neighbourhoods $V_i$. Since this is an isotopy of exact Lagrangian submanifolds, it lifts to a Legendrian isotopy of their Legendrian lifts.

    Using the identifications $F_i \colon D_\epsilon^*\overline{C}_i \to V_i$ from Lemma \ref{nice coordinates}, we represent the Legendrian lift of $L \cap V_i$ by a front in $\overline{C}_i \times \R$. By a genericity argument we can assume that this front has no singularities over the centre of $\overline{C}_i$, and therefore, over a nieghbourhood of the centre of $\overline{C}_i$, it is the union of graphs of smooth functions. We homotope those functions in a smaller neighbourhood of the centre of $\overline{C}_i$ so that they become constant with pairwise distinct values inside an even smaller neighbourhood of the centre. In case there are two lagrangian submanifolds $L$ and $K$, we perform the construction independently for each, but at this point we homotope the functions  further so that the values of the functions corresponding to $L$ are larger than the values of the functions corresponding to $K$. Moreover we peform this homotopy so that no intersection between the graphs corresponding to the same Lagrangian  is created. After this step, we keep working on each Lagrangian separately taking care not to cross the graphs corresponding to different submanifolds. This modification extends to a homotopy of fronts which lifts to a Legendrian isotopy between $L^\dagger$ and a Legendrian submanifold $\widetilde{L}^\dagger$. The projection of this Legendrian isotopy to $W$ produces a reguar exact Lagrangian homotopy from $L$ to an immersed Lagrangian $\widetilde{L}$ such that, in a neighbourhood of the cocore $D_i$ coincides with $k_i$ copies of the core $C_i$. After flowing $\widetilde{L}$ backward with the Liouville flow, we can assume that $\widetilde{L} \cap V_i$ coincides with $k_i$ copies of $\overline{C}_i$ and its Legendrian lift $\widetilde{L}^\dagger \cap V_i \times \R$, identified to a Legendrian submanifold of  $D_\epsilon^*\overline{C}_i \times \R$ via $F_i^\dagger$, coincides with the union of the $1$-jets of $k_i$ constant functions with pairwise distinct values  $c_1^i<\ldots<c_{k_i}^i$. Since the Liouville flow rescales the primitive of the symplectic form, we may further assume that $|c^i_j|<\eta$ for some arbitrarily small $\eta>0$.

    We fix radially symmetric functions $f_i \colon \overline{C}_i \to \R$ such that
    \begin{enumerate}
      \item $f_i=1$ near  $\partial \overline{C}_i$,
      \item $f_i=s$ in a neighbourhood of $\partial C_i$, and 
      \item $f_i$ has a unique critical point in $C_i$, which is a nondegenerate maximum.
      \end{enumerate}
      We define $f^i_j=c_j^if_i$; see Figure \ref{fig:profile}.
      
       \begin{figure}[htp]
        \vspace{3mm}
         \labellist
         \pinlabel $f^i_j(s)$ at 165 53
       	\pinlabel $c^i_j$ at 160 41
       	 \pinlabel $z=c^i_js$ at 170 88
         \pinlabel $1$ at 64 1
       	 \pinlabel $e$ at 154 1
       	 \pinlabel $e^{-\delta}$ at 45 1
         \pinlabel $z$ at 4 96
         \pinlabel $s$ at 167 11
         \endlabellist
        \includegraphics[scale=1.4]{profile}
        \caption{Here $s$ is an inward radial coordinate on the disc such that $s=e$ corresponds to the centre. The graph depicts the profile for a function on the disc that (1) only depends on the radial coordinate, (2) is a perturbation of the constant function $c^i_j$, (3) coincides with the linear function $c^i_j s$ near $s=1$ (i.e.~the boundary of the core), and (4) which has a unique non-degenerate critical point of maximum type inside $s \ge 1$ situated at the centre of the disc.}
        \label{fig:profile}
      \end{figure}

If $\eta$ is small enough, the $1$-jets $j^1f_j^i$ are Legendrian submanifolds of $D^*_\epsilon \overline{C}_i \times \R$ and the linear interpolation between $c_j^i$ and $f_j^i$ is a Legendrian isotopy (relative to the boundary) in  $D^*_\epsilon \overline{C}_i \times \R$. We define a Legendrian submanifold $\overline{L}_0^\dagger$  by replacing $\widetilde{L}^\dagger \cap (V_i \times \R)$ in $\widetilde{L}^\dagger$ with the image of the Legendrian submanifolds $j^1f^i_j$ under the contactomorphism $F_i^\dagger \colon D^*_\epsilon \overline{C}_i \times \R \to V_i \times \R$.

We have thus obtained a Legendrian submanifold $\overline{L}^\dagger_0$ in $W \times \R$ and, by projection, a Lagrangian immersion $\overline{L}_0$ in $W$ which satisfy conditions (1), (2) and (3), but neither $\overline{L}^\dagger_0$ nor $\bs{\Lambda}= \overline{L}_0 \cap \partial W^{sc}$ are chord generic. This however can be achieved by a small generic perturbation of $\overline{L}^\dagger_0$; near $\partial W^{sc} \times \R$ we make the perturbation by replacing $f_j^i=c_i^j s$ with $f_j^i=s \widetilde{f}^i_j$, where $\widetilde{f}^i_j \colon S_i \to \R$ is a generic function close to $c_j^i$. We still call $\overline{L}_0^\dagger$ and $\overline{L}_0$ the resulting manifolds. This perturbation ensures that the $z$--coordinate of $\overline{L}^\dagger_0$ is still $0$ near $\overline{L}^\dagger_0 \cap (\partial W^{sc} \times \R)$ (see Remark \ref{rem_potential}).

The condition (5) can finally be achieved in the following manner.
Given $T\ge0$ we define
$$\overline{L}_T= \psi_{-T}(\overline{L}_0 \cap W^{sc}) \cup \bigcup_{\tau \in [-T, 0]} 
\psi_{-\tau}(\bs{\Lambda}) \cup (\overline{L}_0 \cap W^{crit}).$$
In words, $\overline{L}_T$ is obtained by connecting the boundary of $\psi_{-T}(\overline{L}_0 \cap W^{sc})$ to the boundary of $\overline{L}_0 \cap W^{crit}$ by a cylinder which is tangent to the Liouville vector field. $\overline{L}_T$ is an immersed Lagrangian for every $T \ge 0$ because $\overline{L}_0$ is cylindrical near $\overline{L}_0 \cap \partial W^{sc}$. Moreover, the potential of $\overline{L}_0$ vanishes near $\overline{L}_0 \cap \partial W^{sc}$, and therefore the potential of $\psi_{-T}(\overline{L}_0 \cap W^{sc})$ also vanishes near the boundary. Therefore it can be extended to zero on the cylinder between $\psi_{-T}(\overline{L}_0 \cap W^{sc})$ and 
$\overline{L}_0 \cap W^{crit}$. This gives a well defined potential on $\overline{L}_T$ which becomes arbitrarily small on $\overline{L}_T \cap W^{sc}$ as $T$ becomes large. Hence for every $T \ge 0$, the Lagrangian immersions $\overline{L}_T$ satisfy the conditions (1)--(4) and for $T$  sufficiently large satisfy also the condition (5). Moreover the regular exact lagrangian homotopy $T \mapsto \overline{L}_N$ induces a Legendrian isotopy of the Legendrian lifts $\overline{L}_T^\dagger$. To finish the proof, we take $\overline{L} = \overline{L}_T$ for $T$ sufficiently large and $\overline{h}$ equal to the negative of the $z$ coordinate of the Legendrian lift $\overline{L}_T^\dagger$. 
  \end{proof}

\section{The cobordism algebra of multiple copies of the cores}\label{sec: cap algebra}

The construction of Section \ref{sec:geom-constr} motivate the study of immersed exact Lagrangian cobnordisms consisting of parallel copies of cocore discs. In this section we compute their cobordism algebras. We keep the notation we introduced in the first paragraph of that section.
  We denote by $\widehat{W}^{sc}$ the Liouville completion of $W^{sc}$ and by $\widehat{W}^{crit}$ the Liouville completion of $W^{crit}$. Thus any core disc $C_\sigma$ of the critical part is completed to a Lagrangian plane $\widehat{C}_\sigma$ with a negative end which is asymptotic to  the attaching sphere $S_\sigma$.

  To any string of integers $k_\sigma \ge 0$ for $\sigma \in \pi_0(\mathbf{S})$ we associate an immersed exact Lagrangian cobordism
  $$\widehat{\mathbf{C}} = \bigcup_{\sigma \in \pi_0(\mathbf{S})} ( \widehat{C}_\sigma^{(1)} \cup \ldots \cup C_\sigma^{(k_\sigma)})$$
  by taking $k_\sigma$ parallel copies of $\widehat{C}_\sigma$ as follows:
\begin{enumerate}
     \item Each $\widehat{C}_\sigma^{(i)}$ is a $C^2$-close Hamiltonian isotopic copy of $\widehat{C}_\sigma$,
     \item $\widehat{C}_\sigma^{(i)}$ and $\widehat{C}_{\sigma'}^{(j)}$ intersect transveresely at a single point for every $i \ne j$ if $\sigma = \sigma'$ and are disjoint if $\sigma \ne \sigma'$, 
     \item each $\widehat{C}_\sigma^{(i)}$ is cylindrical over a Legendrian submanifold $\Lambda_\sigma^{(i)}$ of $\partial W^{sc}$ which is Legendrian isotopic  to $S_\sigma$, 
 \item each $\Lambda_\sigma^{(i)}$ is a small perturbation of a push-off of $\Lambda_\sigma^{(i-1)}$ by the positive Reeb flow, and 
 \item the  Legendrian link $\bs{\Lambda} = \bigcup_{\sigma \in \pi_0(\mathbf{S})} (\Lambda_\sigma^{(1)} \cup \ldots \cup \Lambda_\sigma^{k_\sigma})$ is chord generic.
     \end{enumerate}
     Moreover, we will assume that these conditions are achieved by describing each $C_\sigma^{(i)}$ as the Lagrangian projection of the one-jet of a function  $f_{\sigma,i} \colon \widehat{C}_\sigma \to \R$ satisfying the conditions of Lemma \ref{cylindrical jets}  (i.e. $f_{\sigma,i}(s, \tilde{\mathbf{q}})= s \tilde{f}_{\sigma,i}(\widetilde{q})$ in the negative end of $\widehat{C}_\sigma$) such that,  for all $\sigma \in \pi_0(\mathbf{S})$,
     \begin{itemize}
     \item $f_{\sigma, 1} < \ldots < f_{\sigma, k}$,
     \item for all $i<j$ the difference $f_{\sigma, j}-f_{\sigma, i}$ has a unique critical point which is a nondegenerate  maximum and 
     \item the difference $\tilde{f}_{\sigma, j}- \tilde{f}_{\sigma, i}$ has two critical points, a maximum and a minimum.
     \end{itemize}
     \begin{dfn}
       If all conditions above are satisfied, we say that $\widehat{\mathbf{C}}$ is a {\em standard cap}.
     \end{dfn}
     Let $h \colon \widehat{\mathbf{C}} \to \R$ denote the $z$--coordinate of the above graphical Legendrian $\widehat{\mathbf{C}} \subset J^1\widehat{C}_\sigma,dz-\theta)$, i.e.~$h=f_{\sigma, i}$ on the $i$:th connected component $\widehat{C}_\sigma^{(i)}$. In particular, this means that $h$ is choice of potential for the Lagrangian projection of $\widehat{\mathbf{C}}$, i.e.~$dh=\theta|_{T\widehat{C}}$. Let ${\mathcal D}_{\mathbf{C}}$ be the cobordism algebra of $\widehat{\mathbf{C}}$ defined using the potential function $h$, also called the {\em cap algebra}. Since a limit argument will be necessary  to compute ${\mathcal D}_{\mathbf{C}}$ when $\mathbf{S}$ has infinitely many Reeb chords, we fix a regular homotopy $\mathbf{C}_\epsilon$ for  $\epsilon \in (0, 1]$ such that $C^{(i)}_{\sigma,\epsilon}$ is generated by the function $\epsilon f_{\sigma,i}$.
Thus $\mathbf{C}= \mathbf{C}_1$ and each branch $C^{(i)}_{\sigma,\epsilon}$ of $\mathbf{C}_\epsilon$ converges to the associated core $C_\sigma$ as $\epsilon$ goes to zero.  

The generators of ${\mathcal D}_{\mathbf{C}_\epsilon}$ are the intersection points $c_{i,j}^\sigma$ between $\widetilde{C}_{\sigma}^{(i)} \cap \widehat{C}_\sigma^{(j)}$ corresponding to the maxima of $\epsilon(f_{\sigma, j} - f_{\sigma, i})$ for all $\sigma \in \pi_0(\mathbf{S})$ and  $1 \le i < j \le k_\sigma$, and Reeb chords of
$\bs{\Lambda}_\epsilon = \partial \mathbf{C}_\epsilon$. The following lemma is proved by a standard implicit function theorem argument.
\begin{lemma}\label{short, long, and very long}
For every $Q>0$ there exists $\epsilon(Q) \in (0, 1]$ such that, for all $\epsilon \in (0, \epsilon(Q)]$  the Reeb chords of $\bs{\Lambda}_\epsilon$ are of  three types:
 \begin{itemize}
 \item ``short'' chords $e^{\sigma}_{i,j}$ ($e$-chords) and $m^{\sigma}_{i,j}$ ($m$-chords), for $\sigma\in\pi_0(\mathbf{S})$ and $1 \le i < j \le k_\sigma$, which starts  on $\Lambda_{\sigma,\epsilon}^{(i)}$ and end at $\Lambda^{(j)}_{\sigma,\epsilon}$; they correspond to the minimum and the maximum, respectively, of the function $\epsilon (f_{\sigma,j} - f_{\sigma,i})$, 
   
\item ``long'' Reeb chords $a_{i,j}$, for every Reeb chord $a$ of $\Lambda$  of action less than $Q$ and indices $1 \le i \le k_{\sigma_s(a)}$ and $1 \le j \le k_{\sigma_e(a)}$, which are close to $a$, start on $\Lambda^{(i)}_{\sigma_s(a),\epsilon}$ and end at $\Lambda^{(j)}_{\sigma_e(a),\epsilon}$; $\sigma_{s(a)}$ and $\sigma_{e(a)} \in \pi_0(\mathbf{S})$ thus denote the components of the start and endpoint of $a$, respectively; and

 \item ``very long'' chords of action larger than $Q$, over which we have no control.
\end{itemize}
\end{lemma}

A grading of $L$ induces a Maslov potential on $\mathbf{C}_\epsilon$. We fix a point of choice on each component of $\bs{\Lambda}$  and then naturally get $k_\sigma$ induced points on the $k_\sigma$ parallel copies $\bs{\Lambda}_{\sigma,\epsilon}$ for every $\epsilon$ and $\sigma\in\pi_0(\mathbf{S})$. We define $\mathfrak{p}^\sigma=(\mathfrak{p}^\sigma_1, \ldots, \mathfrak{p}^\sigma_{k_\sigma}) \in \Z^{k_\sigma}$ to be the value of the Maslov potential at these $k_\sigma$ points. Note that $\mathfrak{p}^\sigma_i$ depends on the initial choice of point, while the difference $\mathfrak{p}^\sigma_i-\mathfrak{p}^\sigma_j$ does not.
\begin{lemma}\label{degrees of generators}
For $\epsilon \in (0, \epsilon(Q)]$ the degree of the generators of ${\mathcal D}_{{\mathbf C}_\epsilon}$ of action less than $Q$ are:
\begin{equation}\label{gradi}
  \begin{aligned}
  |c^\sigma_{i,j}| & =\mathfrak{p}^\sigma_i - \mathfrak{p}^\sigma_j+n-1, \\
  |m^\sigma_{i,j}|  & =\mathfrak{p}^\sigma_i - \mathfrak{p}^\sigma_j+n-2, \\
  |e^\sigma_{i,j}|  & = \mathfrak{p}^\sigma_i - \mathfrak{p}^\sigma_j -1, \\
  |a_{i,j}|  & =|a|+\mathfrak{p}^{\sigma_s(a)}_i-\mathfrak{p}^{\sigma_e(a)}_j.
  \end{aligned}
\end{equation}
The first three equations hold for $1 \le i<j \le k_\sigma$ and the fourth for $1 \le i \le k_{\sigma_s(a)}$ and $1 \le j \le k_{\sigma_e(a)}$.
\end{lemma}
\begin{proof}
The calculation of the degrees is a straightforward application of \cite[Lemma 3.4]{NonIsoLeg} translated to the language of Maslov potentials; also see e.g.~\cite[Section 3.1]{Duality_EkholmetAl} for similar calculations.

We proceed by giving some more details. Recall that the degree of a Reeb chord $c$ of a Legendrian submanifold in a jet space is given by the formula 
$$ |c|=\mathfrak{p}_{start}-\mathfrak{p}_{end}+\operatorname{index}_p(f_{end}-f_{start})-1 $$
where $\mathfrak{p}_{start},\mathfrak{p}_{end} \in \Z$ are the values of the Maslov potentials at the start and endpoint of $c$, respectively, and $\operatorname{index}_p(f_{end}-f_{start})$ is the Morse index of the difference of the functions that define the sheets of the front projection above a neighborhood of $p$, which is the critical point that corresponds to $c$; see \cite[Lemma 3.4]{NonIsoLeg}. Recall also that the difference between the Maslov potentials is the same for all Reeb chords that start and end on some fixed components $i$ and $j$, respectively. The claimed degree computations then follow from the following facts:
\begin{itemize}
\item $c_{i,j}^\sigma$ corresponds to a local maximum for a function difference on the $n$-dimensional manifold $C_\sigma$ and its degree is equal to the degree of the corresponding chord of the Legendrian lift of $\mathbf{C}_\epsilon$;
\item $m_{i,j}^\sigma$ corresponds to a local maximum for a function difference on the $n-1$-dimensional manifold $\Lambda_\sigma$;
\item $e_{i,j}^\sigma$ corresponds to a local minimum for a function difference on the $n-1$-dimensional manifold $\Lambda_\sigma$; and
\item the degrees of $a_{i,j}$ and $a$ differ precisely by the difference of Maslov potentials.
\end{itemize}
\end{proof}

We will make extensive use of Ekholm's theory of gradient flow trees \cite{MorseFlow, Ekhoka} to compute the differential of ${\mathcal D}_{\mathbf{C}_\epsilon}$; see Appendix \ref{appendix:morse} for more details. Since gradient flow trees describe only pseudoholomorphic curves which are localised in a suitable sense, it is important to localise small energy curves. This is the goal of the next lemma.

\begin{lemma}\label{confinement}
Let $(Y, \alpha)$ be a contact manifold and $\Lambda \subset Y$ a Legendrian submanifold. We fix a compatible cylindrical almost complex structure $J$ on $\R \times Y$ and a Weinstein neighbourhood $U$ of $\Lambda$. Then for every $\eta >0$ there is a Weinstein neighbourhood $V_\eta \subset U$ of $\Lambda$ such that every $J$-holomorphic curve in $\R \times Y$ with boundary on a Legendrian submanifold of $U_\eta$ and Hofer energy less than $\eta$ is contained in $U$.     
\end{lemma}
\begin{proof}
We use the monotonicity property of the symplectic area of pseudoholomorphic curves \cite[Proposition 4.3.1]{Sikorav:SomeProperties} to show that the Hofer energy of a curve that passes through $\R \times \partial U$ must satisfy an a priori bound from below for any fixed cylindrical almost complex structure. To that end we use the argument from \cite[Lemma 5.1]{LegendrianAmbient} which, in a similar setting, derives a monotonicity property for the Hofer energy from the monotonicity of the symplectic area.
\end{proof}

Since the action of short chords and intersection points goes to zero as $\epsilon$ goes to zero, from Lemma \ref{confinement} it follows that, at least for $\epsilon$ small enough, for any $\sigma\in\pi_0(\mathbf{S})$ the differential of the short chords $e_{i,j}^\sigma, m_{i,j}^\sigma$ and the intersection points $c^{\sigma}_{i,j}$ involves only short chords and intersection points labelled by the same $\sigma$. In the next two lemmas we investigate the differential of short chords and self-intersection points. 

\begin{lemma}\label{boundary of c_ij}
  If $\epsilon$ is small enough and $\dim L \ge 2$, for $\sigma \in \pi_0(\mathbf{S})$ and $1 \le i<j \le k_\sigma$  we have
  \begin{gather*}
    \partial_{\mathbf{C}_\epsilon} (c_{i,j}^\sigma)= m_{i,j}^\sigma +\sum \limits_{i<h<j}
    (c_{i, h}^\sigma e_{h,j}^\sigma + e_{i,h}^\sigma c_{h,j}^\sigma),
  \end{gather*}
\end{lemma}

\begin{proof}

  We start with the argument that confines the rigid pseudoholomorphic discs that contribute to $\partial_{\mathbf{C}_\epsilon} c_{i,j}^\sigma$ to a Weinstein neighbourhhood of $\widehat{C}_\sigma$. Since the intersection points $c_{i,j}^\sigma$ can be made to have arbitrarily small action compared to the length of the long chords on $\mathbf{\Lambda}_\epsilon$ by taking $\epsilon$ sufficiently small, an energy argument as in \cite[Equation (11)]{Floer_cob} implies that $\partial_{\mathbf{C}\epsilon} c_{i,j}^\sigma$ cannot contain any long chord. Lemma \ref{confinement} then confines the pseudoholomorphic discs to a Weinstein neighborhood of $\widehat{C}_\sigma$.  We can therefore drop the letter $\sigma$ temporarily in order to simplify the notation. 

  The Legendrian lift of $\widehat{\mathbf{C}}_\epsilon$ can be described as a union of $1$-jets of functions $\epsilon \widehat{f}_i$ where $  \widehat{f}_i \colon \widehat{C} \to \R$ in $J^1\widehat{C}$. See Section \ref{sec:geom-constr}. Moreover, the negative end of $\widehat{C}$ is identified with $(- \infty, 0) \times S$ by the Liouville flow, and $\widehat{f}_i(t, \widetilde{\mathbf{q}})=e^t \widetilde{f}_i(\widetilde{\mathbf{q}})$; see the proof of Theorem \ref{thm:maingeometric}. Note that we have made the change of coordinates $s=e^t$ with respect to the notation of Lemma \ref{cylindrical jets}.  For $i<j$, the double point $c_{i,j}$ corresponds to the unique non-degenerate local maximum $\widetilde{c}_{i,j}$ of the difference $\widehat{f}_j-\widehat{f}_i$. We also denote by $\widetilde{e}_{i,j}$ and $\widetilde{m}_{i,j}$ the critical point of the difference functions $\widetilde{f}_j- \widetilde{f}_j$ on $S$ corresponding to the short chords $e_{i,j}$ and $m_{i,j}$ respectively.

The pseudoholomorphic discs in $\widehat{W}^{crit}$ 
 contributing to $\partial_{\mathbf{C}\epsilon} c_{i,j}$, for a suitable almost complex structure, can be described using the theory of gradient flow trees by Ekholm \cite{MorseFlow}.
 
Theorem \ref{thm:morse} and Lemma \ref{lemma:trees} can be applied to show that 
the holomorphic discs contributing to $\partial_{\mathbf{C}_\epsilon}(c_{i,j})$ are in bijection with the so-called Long Conical (LC for short) rigid gradient flow trees for $\widehat{\mathbf{C}}$ that are contained on $\widehat{C}=\R^n$, are positively
asymptotic to $\tilde{c}_{i,j}$ and negatively asymptotic to points of type $c_{h, k}$, $m_{h, k}$ or $e_{h, k}$ and for $i \le h < k \le j$. 

  The following is satisfied for a generic choice of functions $\widehat{f}_i$ and metrics on $\widehat{C}$ when $n =\dim L \ge 2$:
  \begin{itemize}
  \item The critical points $\widetilde{c}_{j,i}$ of the differences $\widehat{f}_j-\widehat{f}_i \colon \widehat{C} \to \R$ are all pairwise different;

  \item exactly one gradient flow trajectory of $\widehat{f}_j-\widehat{f}_i$ is asymptotic to $\widetilde{m}_{i,j}$ and all other ones are asymptotic to $\widetilde{e}_{i,j}$,
  \item The unique gradient flow trajectory  of $\widehat{f}_j-\widehat{f}_i$ which is  asymptotic to $\widetilde{m}_{i,j}$ does not pass through $\widetilde{c}_{h,k}$ for $\{h,k \} \ne \{i,j \}$, and
  \item every gradient flow trajectory  of $\widehat{f}_j-\widehat{f}_i$ passes through at most one of the points $\widetilde{c}_{h,k}$ with $\{ h,k \} \ne \{ i,j \}$.
  \end{itemize}

  The rigid LC flow trees that are positively asymptotic to $\widetilde{c}_{i,j}$ can now be seen to be of two types.
  \begin{itemize}
  \item \emph{Type 1:} the unique negative gradient flow line of $f_j-f_i$ of infinite length that is asymptotic to $\widetilde{m}_{i,j}$
  \item \emph{Type 2:} a finite-length flow line of $\widehat{f}_j - \widehat{f}_i$ that flows to $\widetilde{c}_{i,h}$ (resp. $\widetilde{c}_{h,j}$) for $i<h<j$, has a negative puncture there, and then is followed by a infinite length flow line of  $\widehat{f}_j-\widehat{f}_h$ (resp. $\widehat{f}_h-\widehat{f}_i) $ that is asymptotic to $\widetilde{e}_{h,j}$ (resp. $\widetilde{e}_{i,h}$).
  \end{itemize}

  See Figure \ref{fig:flowtree} for the two LC flow-lines of Type 2 that has a positive puncture at $\widetilde{c}_{1,3}$. 
\end{proof}

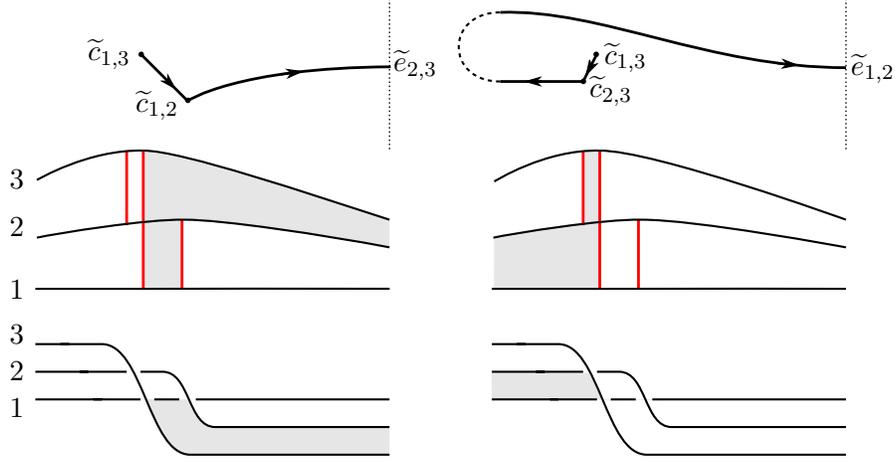
\begin{figure}[htp]
        \vspace{3mm}
         \labellist
         \pinlabel $\widetilde{c}_{1,3}$ at 160 108
         \pinlabel $\widetilde{c}_{2,3}$ at 156 99
        \pinlabel $\widetilde{e}_{1,2}$ at 227 105
        \pinlabel $\widetilde{e}_{2,3}$ at 103 106
        \pinlabel $\widetilde{c}_{1,3}$ at 20 109
        \pinlabel $\widetilde{c}_{1,2}$ at 33 95
         \pinlabel $1$ at -5 13
         \pinlabel $2$ at -5 23
         \pinlabel $3$ at -5 33
         \pinlabel $1$ at -5 45
         \pinlabel $2$ at -5 62
         \pinlabel $3$ at -5 75

         \endlabellist
        \includegraphics[scale=1.4]{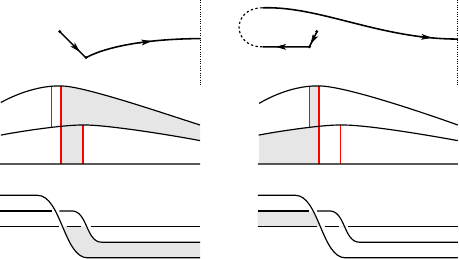}
        \caption{Two rigid gradient flow trees of ``Type 2'' that correspond to pseudoholomorphic discs in the handle with boundary on $\widehat{C}_\epsilon^{(1)} \cup \widehat{C}_\epsilon^{(2)} \cup \widehat{C}_\epsilon^{(3)}$ with a unique positive puncture at the double point $c_{1,3}$. Top: the gradient flow trees. Middle: the front projections. Bottom: the Lagrangian projections. There exists one additional pseudoholomorphic strip that is not shown in the figure; it connects $c_{1,3}$ directly to $m_{1,3}$. (Note that the Lagrangian projection depicts the case $\dim L=1$, which can be misleading.)}
        \label{fig:flowtree}
      \end{figure}

\begin{lemma} \label{boundary of m_ij}
  If $\epsilon$ is small enough, for every $\sigma \in \pi_0(\mathbf{S})$ and $1 \le i<j \le  k_\sigma$ every term appearing in $\partial_{\mathbf{C}_\epsilon}m^\sigma_{i,j}$ is a word of short chords which contains at least one $m$-chord.   
\end{lemma}
\begin{proof}
First, we recall that the boundary of $m^\sigma_{ij}$ in ${\mathcal D}_{\mathbf{C}_\epsilon}$ is by definition equal to the boundary of $m^\sigma_{i,j}$ in the Chekanov-Eliashberg algebra of $\bs{\Lambda}_\epsilon$.
  By energy considerations, no holomorphic curve in the symplectisation of $\partial W^{sc}$ with a positive end at $m_{i,j}^\sigma$ can have a negative end at a long or very long chord because they are longer than $m_{i,j}^\sigma$. To prove the lemma, it remains to show that there is no (or an algebraically zero number of) rigid holomorphic curve in the symplectisation of $\partial W_{sc}$ with a positive end at $m_{i,j}^\sigma$ and  negative ends only at $e$-chords.
  
Consider a fixed standard neighborhood $U$ of $\mathbf{S}$ inside $\partial W^{sc}$ which can be identified with a neighbourhood of the zero section of $J^1\mathbf{S}$. When 
$\epsilon$ goes to zero, the action of all $m$-chords also goes to zero and moreover $\bs{\Lambda}_\epsilon$ approaches $\mathbf{S}$. Then by Lemma \ref{confinement}  all pseudoholomorphic curves contributing to $\partial_{\mathbf{C}_\epsilon}m_{i,j}^\sigma$ are contained in $\R \times U$. Since every connected component of $U$ has is a standard neighbourhood of a connected component of $\mathbf{S}$, by connectedness  all pseudoholomorphic curves contributing to $\partial_{\mathbf{C}_\epsilon}m_{i,j}^\sigma$ are contained in $U_\sigma$, the connected component of $U$ containing $S_\sigma$.    
The Legendrian submanifold $\bs{\Lambda}_\epsilon$ is graded in $J^1\mathbf{S}$, and therefore  by Equation \eqref{gradi} the dimension of the moduli space ${\mathcal M}_{\bs{\Lambda}_\epsilon}(m_{i,j}^\sigma; e_{i, k_1}^\sigma, e_{k_1, k_2}^\sigma, \ldots , e^\sigma_{k_\ell, j})$ is
  $$\dim {\mathcal M}_{\bs{\Lambda}_\epsilon}(m_{i,j}^\sigma; e^\sigma_{i, k_1}, e^\sigma_{k_1, k_2}, \ldots , e^\sigma_{k_\ell, j})= |m^\sigma_{i,j}|- |e^\sigma_{i, k_1}| - \ldots - |e^\sigma_{k_\ell, j}|=n+\ell-1.$$
Then $\dim {\mathcal M}_{\bs{\Lambda}_\epsilon}(m^\sigma_{i,j}; e^\sigma_{i, k_1}, e^\sigma_{k_1, k_2}, \ldots , e^\sigma_{k_\ell, j})>1$ (and therefore the moduli space does not contribute to the boundary) unless $n=2$ and $\ell=0$. In this case, however, a simple gradient flow trees computation gives $$\# \widetilde{\mathcal M}^1_{\bs{\Lambda}_\epsilon}(m^\sigma_{i,j}, e^\sigma_{i,j})=0.$$
\end{proof}
\begin{lemma}\label{boundary of e_ij}
  If $\epsilon$ is small enough, for every $\sigma \in \pi_0(\mathbf{S})$ and $1 \le i<j \le  k_\sigma$ we have
  $$\partial_{\mathbf{C}_\epsilon}(e^\sigma_{i,j})= \sum_{i <h<j} e_{i,h}^\sigma e_{h,j}^\sigma + \mathfrak{m}$$
  where every term of $\mathfrak{m}$ is a word in short chords containing at least one $m$-chord.
\end{lemma}
\begin{proof}
  By an action argument we see that the differential only consists of words of short chords. In the case when all short chords are minimum-type chords $e_{kl}^\sigma$, a local degree computations shows that there must be precisely two such chords. That the number of these terms are as sought finally follows from Proposition \ref{louie2}. \end{proof}

From now on we assume that  Lemmas \ref{boundary of c_ij}, \ref{boundary of m_ij} and \ref{boundary of e_ij} hold for all $\epsilon \in (0,1]$. This can be obtained by rescaling the functions $\widehat{f}_i^\sigma$. Finally we describe the differential of long chords, but first we need to introduce some notation. Given a word $\mathbf{a}=a^1 \ldots a^n$ of composable chords of $\mathbf{S}$ and indices $i,j$ such that $1 \le i \le k_{\sigma_s(a^1)}$ and $1 \le j \le k_{\sigma_e(a^n)}$, we define
$$\mathbf{a}_{i,j} = \sum_{h_1, \ldots, h_{n-1}} a^1_{i,h_1} a^2_{h_1, h_2} \ldots a^{n-1}_{h_{n-2}, h_{n-1}} a^n_{h_{n-1}, j}$$
where the sum is over all $h_\ell$ such that $1 \le h_\ell \le k_{s_e(a^\ell)} = k_{s_s(a^{\ell+1})}$. This notation is extended by linearity in the obvious way to sums of words of chords. Also denote by $({\mathcal A}_{\slashed{\mathbf{S}}}, \partial_{\slashed{\mathbf{S}}})$ the dg algebra obtained from $({\mathcal A}_{\mathbf{S}}, \partial_{\mathbf{S}})$ by omitting the idempotent, and observe that, for any chord $a$, $\partial_{\mathbf{S}}(a)$ and $\partial_{\slashed{\mathbf{S}}}(a)$ differ only for the treatment of the constant term.

\begin{lemma}\label{boundary of a_ij}
  For every $Q>0$ there is $\epsilon'(Q) \in (0, \epsilon(Q)]$ such that, for all $\epsilon \in (0, \epsilon'(Q)]$, all chords $a$ of $\mathbf{S}$ with action less than $Q$ and all indices $i,j$ satisfying $1 \le i \le k_{\sigma_s(a)}$ and $1 \le j \le k_{\sigma_e(a)}$ we have
  $$\partial_{\mathbf{C}_\epsilon}(a_{i,j}) = (\partial_{\slashed{\mathbf{S}}})_{i,j} + \sum_{i <h < k_{\sigma_s(a)}} e_{i,h}^{\sigma_s(a)} a_{h,j} + \sum_{1<h<j} a_{i,h}e^{\sigma_e(a)}_{h,j} + \mathfrak{m}$$
  where $\mathfrak{m}$ is a sum of words containing an $m$-chord.
\end{lemma}
\begin{proof}
 From Proposition \ref{huey} we obtain that, if $\epsilon$ is small enough, for every chord $a_{i,j}$ of $\mathbf{\Lambda}_\epsilon$ of action less than $Q$ the term of $\partial_{\mathbf{C}_\epsilon} a_{i,j}$ involving only long chords is
  $(\partial_{\slashed{\mathbf{S}}} a)_{i,j}$.

By Lemma \ref{lemma: constraints on holomorphic curves}, the only words containing an $e$-chord and no $m$-chord which can appear in $\partial_{\mathbf{C}_\epsilon} a_{i,j}$ are $a_{i,h}e^{\sigma_e(a)}_{h,j}$ for $h<j$ or $e_{i,h}^{\sigma_s(a)}a_{h,j}$ for $h>i$. Furthermore, by Proposition \ref{louie} we have the counts
  $$\#\widetilde{\mathcal M}_{\bs{\Lambda}_\epsilon}^1(a_{i,j};  a_{i,h}, e^{\sigma_e(a)}_{h,j}) = \# \widetilde{\mathcal M}_{\bs{\Lambda}_\epsilon}^1(a_{i,j};  e_{i,h}^{\sigma_s(a)}, a_{h,j})=1,$$
  and therefore the term of $\partial_{\mathbf{C}_\epsilon} a_{i,j}$ involving at least one $e$-chord and no $m$-chord is
$$\sum_{i <h < k_{\sigma_s(a)}} e_{i,h}^{\sigma_s(a)} a_{h,j} + \sum_{1<h<j} a_{i,h}e^{\sigma_e(a)}_{h,j}.$$
\end{proof}
Given $Q>0$, let ${\mathcal D}_{\mathbf{C}_{\epsilon'(Q)}}^Q$ be the sub-dg algebra of ${\mathcal D}_{\mathbf{C}_{\epsilon'(Q)}}$ generated by short chords and long chords $a_{i,j}$ where $a$ has action less than $Q$. If $Q_1 > Q_0$ there is an inclusion ${\mathcal D}_{\mathbf{C}_{\epsilon'(Q_0)}}^{Q_0} \hookrightarrow {\mathcal D}_{\mathbf{C}_{\epsilon'(Q_1)}}^{Q_1}$, and therefore we can define the abstract dg algebra
$${\mathcal D}_{\mathbf{C}}^{\mathit{ab}} = \varinjlim {\mathcal D}_{\mathbf{C}_{\epsilon'(Q)}}^Q = \bigcup_Q {\mathcal D}_{\mathbf{C}_{\epsilon'(Q)}}^Q.$$

Wrapping up the definition. ${\mathcal D}_{\mathbf{C}}^{\mathit{ab}}$ is generated by elements $c_{i,j}^\sigma, m^\sigma_{i,j}, e_{i,j}^\sigma$ for all $\sigma \in \pi_0(\mathbf{S})$ and  $1 \le i < j \le k_\sigma$, and elements $a_{i,j}$ for every chord $a$ of $\mathbf{S}$ and $1 \le i \le k_{\sigma_s(a)}$, $1 \le j \le k_{\sigma_e(a)}$ with differential
$\partial_{\mathbf{C}}^{\mathit{ab}}$ defined by 
\begin{align*}
  \partial_{\mathbf{C}}^{\textit{ab}}(a_{i,j}) = &  (\partial_{\slashed{\mathbf{S}}}(a))_{i,j} + \sum_{i <h < k_{\sigma_s(a)}} e_{i,h}^{\sigma_s(a)} a_{h,j} + \sum_{1<h<j} a_{i,h}e^{\sigma_e(a)}_{h,j} + \mathfrak{m}(a_{ij})\\
   \partial_{\mathbf{C}}^{\textit{ab}}(c^\sigma_{i,j}) = &  m_{i,j}^\sigma +\sum \limits_{i<h<j} (c_{i, h}^\sigma e_{h,j}^\sigma + e_{i,h}^\sigma c_{h,j}^\sigma) \\
  \partial_{\mathbf{C}}^{\textit{ab}}(m^\sigma_{i,j}) = & \mathfrak{m}(m^\sigma_{i,j}) \\
  \partial_{\mathbf{C}}^{\textit{ab}}(e^\sigma_{i,j}) = & \sum_{i <h<j} e_{i,h}^\sigma e_{h,j}^\sigma + \mathfrak{m}(e^\sigma_{i,j})
\end{align*}
where $\mathfrak{m}(a_{ij})$ is a sum of words containing at least one $m$-chord each, and $\mathfrak{m}(m^\sigma_{i,j}), \mathfrak{m}(e^\sigma_{i,j})$ are sums of words of short chords containing at least one $m$-chord each.

\begin{lemma}\label{from D to Dab}
  There is a dg algebra morphism ${\mathcal D}^{\textit{ab}}_{\mathbf{C}} \to {\mathcal D}_{\mathbf{C}}$ which induces an isomorphism on homology.
\end{lemma}
\begin{proof}
  The continuation maps between the Chekanov-Eliashberg algebras ${\mathcal A}_{\bs{\Lambda}_\epsilon}$ and ${\mathcal A}_{\bs{\Lambda}_{\epsilon'}}$ induce continuation maps between ${\mathcal D}_{\mathbf{C}_\epsilon}$ and ${\mathcal D}_{\mathbf{C}_{\epsilon'}}$ which, by Lemma \ref{if nothing happens nothing happens}, coincide with the canonical identification between the generators on ${\mathcal D}^Q_{\mathbf{C}_\epsilon}$ and ${\mathcal C}^Q_{\mathbf{C}_{\epsilon'}}$ if $\epsilon, \epsilon' < \epsilon(Q)$. Moreover, one can find a sequence  $Q_n$ and $\epsilon_n < \epsilon(Q_n)$ such that Equation \eqref{Qn grows fast} is satisfied because the the concatenations of the Legendrian isotopies $\bs{\Lambda}_{\epsilon_{i+1}} \to \bs{\Lambda}_{\epsilon_i}$ for all $i$ has finite length. Then we can apply Lemma \ref{taking limits} to ${\mathcal A}_n= {\mathcal D}_{\mathbf{C}_{\epsilon_n}}$ and ${\mathcal B} = {\mathcal D}^{\textit{ab}}_{\mathbf{C}}$ to obtain a dg algebra morphism  ${\mathcal D}^{\textit{ab}}_{\mathbf{C}} \to {\mathcal D}_{\mathbf{C}_{\epsilon_0}}$, which we compose to the continuation map ${\mathcal D}_{\mathbf{C}_{\epsilon_0}} \to {\mathcal D}_{\mathbf{C}}$.
\end{proof}

Now we want to relate ${\mathcal D}^{\textit{ab}}_{\mathbf{C}}$ the the algebraic operations defined in Section \ref{sec: idempotents and dga}. Let ${\mathcal A}_{\mathbf{S}}$ be the Checkanov-Eliashberg algebra of $\mathbf{S}$ over the  ring $\kk_{\mathbf{S}}$ with idempotents associated to the connected components of $\mathbf{S}$, and let $\kk_{\mathbf{C}}$ be the ring with idempotents  associated to the connected components of $\bs{\Lambda}$. We define the differential graded algebra ${\mathcal A}^+_{\slashed{\mathbf{C}}}$ by applying first partial Morsification, then expansion of idempotents from $\kk_{\mathbf{S}}$ to $\kk_{\mathbf{C}}$, and finally omission of idempotents to ${\mathcal A}_{\mathbf{S}}$. Finally, since there is a natural ordering on the connected components of $\bs{\Lambda}$ corresponding to the same connected component of $\mathbf{C}$, we can form thew algebra $\vec{\mathcal A}^+_{\slashed{\mathbf{C}}}$ by killing all generators $e^\sigma_{i,j}$ with $i \ge j$ in  ${\mathcal A}^+_{\slashed{\mathbf{C}}}$, see Lemma \ref{ho finito le labelle}. If we unwrap the definition of $\vec{\mathcal A}^+_{\slashed{\mathbf{C}}}$, we see that it is the differential graded algebra generated by elements $a_{i,j}$ and $e^\sigma_{i,j}$ as for ${\mathcal D}^{\mathit{ab}}_{\mathcal C}$ with differential
\begin{align*}
  \vec{\partial}^+_{\slashed{\mathbf{C}}}(a_{i,j}) & = (\partial_{\slashed{\mathbf{S}}})_{i,j} + \sum_{i <h < k_{\sigma_s(a)}} e_{i,h}^{\sigma_s(a)} a_{h,j} + \sum_{1<h<j} a_{i,h}e^{\sigma_e(a)}_{h,j} \\
  \vec{\partial}^+_{\slashed{\mathbf{C}}}(e_{i,j}) & = \sum_{i <h<j} e_{i,h}^\sigma e_{h,j}^\sigma.
\end{align*}
Thus it is evident that there is a dg-morphism ${\mathcal D}^{\textit{ab}}_{\mathbf{C}} \to \vec{\mathcal A}^+_{\slashed{\mathbf{C}}}$ which maps all $c_{i,j}^\sigma$ and $m_{i,j}^\sigma$ to zero. However, we want to pull back augmentations from ${\mathcal D}^{\textit{ab}}_{\mathbf{C}}$ to $\vec{\mathcal A}^+_{\slashed{\mathbf{C}}}$, and therefore we need a dg-morphism in the opposite direction. The following lemma is the algebraic tool to produce such a morphism.
\begin{lemma}\label{from Chekanov}
  Let ${\mathcal A}$ be a differential graded algebra freeley generated as an algebra by elements $a_n, \ldots, a_1, a, b$ such that
\begin{itemize}
\item $\partial  b =0$,
\item $\partial  a = b$, and
\item $\partial a_i$ belong to the sub-algebra generated by $a_{i-1}, \ldots, a_1, a, b$ for every $i = 1, \ldots, n$.
\end{itemize}
  If ${\mathcal M}$ is the bilateral differential ideal generated by $a$ and $b$, then there is a dg-morphism ${\mathcal A} / {\mathcal M} \to {\mathcal A}$ inverting the projection ${\mathcal A} \to {\mathcal A} /{\mathcal M}$ to the right and inducing an isomorphism in homology.
\end{lemma}
\begin{proof}
  The lemma is proved by combining an argument in \cite[Section 8.4]{Chekanov} with \cite[Lemma 2.1]{Chekanov}.
\end{proof}
We say that $a$ and $b$ are in elimination position and ${\mathcal A}/{\mathcal M}$ is obtained by eliminating $a$ and $b$ from ${\mathcal A}$.
\begin{lemma}\label{elimination of c and m}
  There is a dg morphism $\vec{\mathcal A}^+_{\slashed{\mathbf{C}}} \to {\mathcal D}^{\textit{ab}}_{\mathbf{C}}$ which is a right inverse of the projection and induces an isomorphism in homology.
\end{lemma}
\begin{proof}
For every $\sigma \in \pi_0(\mathbf{\Lambda})$ and $i = 1, \ldots, k_{\sigma}-1$ we have $\partial_{\mathbf{C}}(c_{i,i+1}^\sigma)=m_{i,i+1}^\sigma$ and 
  $\partial_{\mathbf{C}}(m_{i,i+1}^\sigma)=0$, then we can apply Lemma \ref{from Chekanov}
  and eliminate all $c_{i,j}^\sigma$ and $m_{i,j}^\sigma$. After we have eliminated them, the generators $c_{i,j}^\sigma$ and $m_{i,j}^\sigma$ with $j-i=2$ are in elimination position, and so on , so we can proceed inductively on $j-i$ untill we have eliminated all $c^\sigma_{i,j}$ and $m^\sigma_{i,j}$. The resulting algebra is isomorphic to $\vec{\mathcal A}^+_{\slashed{\mathbf{C}}}$ and therefore Lemma \ref{from Chekanov} provides a dg-morphism $\vec{\mathcal A}^+_{\slashed{\mathbf{C}}} \to {\mathcal D}^{\textit{ab}}_{\mathbf{C}}$ which induces an isomorphism in homology.
\end{proof}  
\section{The Cthulhu complex of multiple cores}\label{Cth(C,C)}
Let $\widehat{\mathbf{C}}^0$ and $\widehat{\mathbf{C}}^1$ be standard caps such that $\widehat{\mathbf{C}}^0 \cup \widehat{\mathbf{C}}^1$ is also a standard cap, and moreover the connected components of $\bs{\Lambda}^0 = \partial \mathbf{C}^0$ are further in the positive Reeb direction from the corresponding connected component of $\mathbf{S}$ than the connected components of $\bs{\Lambda}^1 = \mathbf{C}^1$.
For $d=0,1$ we write $\bs{\Lambda}_\sigma^{(d,1)}, \ldots \bs{\Lambda}_\sigma^{(d,k_\sigma^d)}$ for the connected components of $\Lambda^d$ corresponding to $S_\sigma$, and $\widehat{C}_\sigma^{(d,i)}$ for the Lagrangian plane in $\widehat{\mathbf{C}}^d$ such that $\partial \widehat{C}_\sigma^{(d,i)} = \bs{\Lambda}_\sigma^{(d,i)}$. When these properties are satisfied we say that $\mathbf{C}_0$ and $\mathbf{C}_1$ are in {\em standard position}.

The goal of this section is to compute the Cthulhu complex $\underline{\op{Cth}}(\widehat{\mathbf{C}}_0,  \widehat{\mathbf{C}}_1)$ from the Chekanov-Eliashberg algebra ${\mathcal A}_{\slashed{\mathbf{S}}}$. We recall that this complex is a free $({\mathcal D}_{\mathbf{C}^1} \mhyphen {\mathcal D}_{\mathbf{C}_0})$-bimodule generated by intersection points between $\widehat{\mathbf{C}}^0$ and $\widehat{\mathbf{C}}^1$ and Reeb chords from $\bs{\Lambda}^1$ to $\bs{\Lambda}^0$. For $\sigma \in \pi_0(\mathbf{S})$, $1 \le i \le  k_\sigma^1$ and $1 \le j \le k_\sigma^0$ we denote by $\hat{c}^\sigma_{i,j}$ the intersection point between $C_\sigma^{(1,i)}$ and $C_\sigma^{(0,j)}$.

We choose families of standard caps $\mathbf{C}_{0, \epsilon}$ and $\mathbf{C}_{1,\epsilon}$ for $\epsilon \in(0,1]$ which are in standard position for every $\epsilon$, and such that $\mathbf{C}_{d,1}= \mathbf{C}_d$ and $C^{d, i}_{\sigma, \epsilon} \to C_\sigma$ as $\epsilon \to 0$ for $d=0,1$. As in the previous section, for every $Q>0$ there exists $\epsilon(Q)$ such that the Reeb chords of $\bs{\Lambda}_{0, \epsilon} \cup \bs{\Lambda}_{1, \epsilon}$ divide into short, long and very long, with action larger than $Q$.
We denote the short and long chords of $\bs{\Lambda}_{d, \epsilon}$ by $e^\sigma_{d;i,j}$,  $m^\sigma_{d; i,j}$ and $a_{d; i,j}$, the short chords from $\bs{\Lambda}_{1, \epsilon}$ to $\bs{\Lambda}_{0, \epsilon}$ by $\hat{e}^\sigma_{j,k}$ and $\hat{m}^\sigma_{i,j}$, and the long chords from  $\bs{\Lambda}_{1, \epsilon}$ to $\bs{\Lambda}_{0, \epsilon}$ corresponding to the chord $a$ of $\mathbf{S}$ by $\hat{a}_{i,j}$.

For $\epsilon < \epsilon(Q)$ we consider the $({\mathcal D}_{{\mathcal C}_{1, \epsilon}} \mhyphen {\mathcal D}_{{\mathcal C}_{0, \epsilon}})$-bimodule $\underline{\op{Cth}}^Q(\mathbf{C}_{0, \epsilon}, \mathbf{C}_{1, \epsilon})$ generated by intersection points and chords of action less than $Q$ (i.e.\ short and long chord). 
In order to describe the differential of $\underline{\op{Cth}}^Q(\mathbf{C}_{0, \epsilon}, \mathbf{C}_{1, \epsilon})$ we introduce the following notation: given a composable word $\mathbf{a}= a^1 \ldots a^n$ of chords of $\mathbf{S}$ and indices $i,j$ such that $1 \le i \le k^1_{\sigma_s(a^1)}$ and $1 \le j \le k^1_{\sigma_s(a^1)}$, we define $\hat{\mathbf{a}}_{i,j}$ as the sum over all possible lifts of $\mathbf{a}$ to a composable word of chords of $\bs{\Lambda}_{0, \epsilon} \cup \bs{\Lambda}_{1, \epsilon}$ starting at $C^{1,i}_{\sigma_s(a^1)}$, ending at $C^{0,j}_{\sigma_e(a^n)}$ and containing only one chord from $\bs{\Lambda}_1$ to $\bs{\Lambda}_0$ (and, therefore, no chord from $\bs{\Lambda}_0$ to $\bs{\Lambda}_1$). In more explicit terms
$$\hat{\mathbf{a}}_{i,j} = \sum_{\ell = 1}^{n-1} \sum_{h_1, \ldots, h_{n-1}} a^1_{1; i,h_1} \ldots a^{\ell-1}_{1; h_{\ell-2}, h_{\ell-1}} \hat{a}^\ell_{h_{\ell-1}, h_\ell} a^{\ell+1}_{0; h_{\ell}, h_{\ell+1}} \ldots a^n_{0; h_{n-1}, j}.$$
\begin{lemma}
  For all $Q>0$ there exists $\epsilon'(Q) < \epsilon(Q)$ such that, for every $\epsilon \in (0, \epsilon'(Q)]$,  the differential of  $\underline{\op{Cth}}^Q(\mathbf{C}_{0, \epsilon}, \mathbf{C}_{1, \epsilon})$ is
  
  \begin{align*}
    \mathfrak{d}(\hat{c}^\sigma_{i,j}) = &  \hat{m}_{i,j}^\sigma +\sum \limits_{i<h<k^1_\sigma} (c_{1; i, h}^\sigma \hat{e}_{h,j}^\sigma + e_{1;i,h}^\sigma \hat{c}_{h,j}^\sigma ) + \sum \limits_{1<h<j}  (\hat{c}_{i, h}^\sigma e_{0;h,j}^\sigma   +  \hat{e}_{i,h}^\sigma c_{0;h,j}^\sigma) \\
    \mathfrak{d}(\hat{m}^\sigma_{i,j}) = & \hat{\mathfrak{m}}(\hat{m}^\sigma_{i,j}) \\
    \mathfrak{d}(\hat{e}^\sigma_{i,j})= &  \sum_{i <h<k^1_\sigma} e_{1; i,h}^\sigma \hat{e}_{h,j}^\sigma + \sum_{1 \le h \le j} \hat{e}^\sigma_{i,h} e^\sigma_{0; h,j} + \hat{\mathfrak{m}}(\hat{e}^\sigma_{i,j}) \\
    \mathfrak{d}(\hat{a}_{i,j}) = &\widehat{(\partial_{\slashed{\mathbf{S}}})}_{i,j} + \sum_{1 <h < k^1_{\sigma_s(a)}} \hat{e}_{i,h}^{\sigma_s(a)} a_{0;h,j}  + \sum_{i<h < k^1_{\sigma_s(a)}} e_{1; i,h}^{\sigma_s(a)} \hat{a}_{h,j}+ \\
    & \sum_{1<h<k^0_{\sigma_e(a)}} \hat{a}_{i,h}e^{\sigma_e(a)}_{0; h,j} + \sum_{1 \le h \le j}  a_{1; i,h}\hat{e}^{\sigma_e(a)}_{h,j} +  \hat{\mathfrak{m}}(\hat{a}_{i,j}),
  \end{align*}
  where $\hat{\mathfrak{m}}$ denotes a sum of words containing either an $m$-chord or an $\hat{m}$-chord.
\end{lemma}
\begin{proof}
  The action of the intersection points $\widehat{c}_{i,j}$ goes to zero as $\epsilon$ goes to zero, and therefore we can assume that for $\epsilon$ small enough  $\mathfrak{d}_{-0}=0$ (i.e. there are no nessies); see \cite[Lemma 7.2]{Floer_cob} and \cite[Lemma 7.8]{Floer_cob}. Then all holomorphic curves that appear in the definition of $\mathfrak{d}$ appear also in the definition of the differential of ${\mathcal D}_{\mathbf{C}_{0,\epsilon} \cup \mathbf{C}_{1,\epsilon}}$, and therefore the lemma follows from Lemmas \ref{boundary of c_ij}, \ref{boundary of m_ij}, \ref{boundary of e_ij} and \ref{boundary of a_ij}.
  \end{proof}
  We define the $({\mathcal D}_{\mathbf{C}_1}^{\textit{ab}} \mhyphen {\mathcal D}_{\mathbf{C}_0}^{\textit{ab}})$-bimodule $\underline{\op{Cth}}^{ab}(\mathbf{C}_0, \mathbf{C}_1)$ by
  $$\underline{\op{Cth}}^{ab}(\mathbf{C}_0, \mathbf{C}_1) = \varinjlim \underline{\op{Cth}}^Q(\mathbf{C}_{0, \epsilon'(Q)}, \mathbf{C}_{1, \epsilon'(Q)}).$$

  Note that the dg-morphisms ${\mathcal D}_{\mathbf{C}_i}^{\mathit{ab}} \to {\mathcal D}_{\mathbf{C}_i}$ mage $\underline{\op{Cth}}(\mathbf{C}_0, \mathbf{C}_1)$ a $({\mathcal D}_{\mathbf{C}_1}^{\mathit{ab}} \mhyphen {\mathcal D}_{\mathbf{C}_0}^{\mathit{ab}})$-bimodule.
  
  \begin{lemma}
    $\underline{Cth}^{\textit{ab}}(\mathbf{C}_0, \mathbf{C}_1)$ is quasi-isomorphic to $\underline{Cth}(\mathbf{C}_0, \mathbf{C}_1)$
  \end{lemma}
  \begin{proof}
    First we need to construct continuation maps between $\underline{Cth}(\mathbf{C}_{0, \epsilon}, \mathbf{C}_{1, \epsilon})$ and $\underline{Cth}(\mathbf{C}_{0, \epsilon'}, \mathbf{C}_{1, \epsilon'})$ for $\epsilon, \epsilon' \in (0,1]$ which are compatible with the continuation maps between ${\mathcal D}_{\mathbf{C}_{i,\epsilon}}$ and ${\mathcal D}_{\mathbf{C}_{i, \epsilon'}}$, induce isomorphisms in homology, and coincide with the canonical identification on short and long chords when $\epsilon$ and $\epsilon'$ are small enough. Since the differential on the intersection points does not depend on $\epsilon$ it is enough to define the continuation maps geometrically between the submodules  $\underline{LCC}(\bs{\Lambda}_{0, \epsilon}, \bs{\Lambda}_{1, \epsilon})$ and $\underline{LCC}(\bs{\Lambda}_{0, \epsilon'}, \bs{\Lambda}_{1, \epsilon'})$ generated by the  chords and extend them to the identity on the interserction points. The submodule $\underline{LCC}(\bs{\Lambda}_{0, \epsilon}, \bs{\Lambda}_{1, \epsilon})$ (and of course the same for $\epsilon'$) can be derived from the Chekanov-Eliashberg algebra ${\mathcal A}_\F(\bs{\Lambda}_{0, \epsilon} \cup \bs{\Lambda}_{1, \epsilon})$ as follows. We define the bilateral ideal ${\mathcal M}_\epsilon  \subset {\mathcal A}_\F(\bs{\Lambda}_{0, \epsilon} \cup \bs{\Lambda}_{1, \epsilon})$ generated by
    \begin{itemize}
    \item chords from $\Lambda_{0, \epsilon}$ to $\Lambda_{1, \epsilon}$ and
    \item words $ab$ where the end  point of $a$ is $\Lambda_{0, \epsilon}$ and the starting point of $b$ is in $\Lambda_{1, \epsilon}$ or vice versa. 
    \end{itemize}
    Since the differential of every chord is a sum of composable words, ${\mathcal M}$ is a differential ideal. It is easy to see that $\underline{LCH}(\mathbf{\Lambda}_{0, \epsilon}, \mathbf{\Lambda}_{1, \epsilon})$ is isomorphic to the vector subspace of   ${\mathcal A}_\F(\bs{\Lambda}_{0, \epsilon} \cup \bs{\Lambda}_{1, \epsilon})/{\mathcal M}_\epsilon$ generated by words with starting points on $\bs{\Lambda}_{1, \epsilon}$ and end point on $\bs{\Lambda}_{0,\epsilon}$. Since the continuation maps between ${\mathcal A}_\F(\bs{\Lambda}_{0, \epsilon} \cup \bs{\Lambda}_{1, \epsilon})$ and ${\mathcal A}_\F(\bs{\Lambda}_{0, \epsilon'} \cup \bs{\Lambda}_{1, \epsilon'})$ preserve the starting and end point of every chord, they induce continuation maps between $\underline{LCC}(\bs{\Lambda}_{0, \epsilon}, \bs{\Lambda}_{1, \epsilon})$ and $\underline{LCC}(\bs{\Lambda}_{0, \epsilon'}, \bs{\Lambda}_{1, \epsilon'})$, and therefore between $\underline{\op{Cth}}(\mathbf{C}_{0, \epsilon}, \mathbf{C}_{1, \epsilon})$ and $\underline{\op{Cth}}(\mathbf{C}_{0, \epsilon'}, \mathbf{C}_{1, \epsilon'})$. The lemma is now proved by a bimodule version of Lemma \ref{taking limits}.
  \end{proof}

  Let the $\vec{\mathcal S}^+_{\slashed{\mathbf{C}}_1, \slashed{\mathbf{C}}_0}$ be the $(\vec{\mathcal A}^+_{\slashed{\mathbf{C}}_1} \mhyphen \vec{\mathcal A}^+_{\slashed{\mathbf{C}}_1})$-bimodule
  $$\vec{\mathcal S}^+_{\slashed{\mathbf{C}}_1, \slashed{\mathbf{C}}_0} =  \vec{\mathcal A}^+_{\slashed{\mathbf{C}}_1}\otimes_{{\mathcal A}^+_{\slashed{\mathbf{C}}_1}} {\mathcal S}^+_{\slashed{\mathbf{C}}_1, \slashed{\mathbf{C}}_0} \otimes_{{\mathcal A}^+_{\slashed{\mathbf{C}}_0}} {\mathcal A}^+_{\slashed{\mathbf{C}}_0}.$$
  If we unwrap the definition, we see that it is generated by elements
  \begin{itemize}
  \item   $\hat{e}^\sigma_{i,j}$ for every $\sigma \in \pi_0(\mathbf{S})$ and indices $i,j$ such that $1 \le i \le k^1_\sigma$ and $1 \le j \le k^0_\sigma$, and
  \item $\hat{a}_{i,j}$ for every chord $a$ of $\mathbf{S}$ and indices $i,j$ such that $1 \le i \le k^1_{\sigma_s(a)}$ and $1 \le j \le k^0_{\sigma_e(a)}$
\end{itemize}
and has differential
\begin{align*}
  \mathfrak{d}(\hat{e}^\sigma_{i,j})= &  \sum_{i <h<k^1_\sigma} e_{1; i,h}^\sigma \hat{e}_{h,j}^\sigma + \sum_{1 \le h \le j} \hat{e}^\sigma_{i,h} e^\sigma_{0; h,j}, \\
    \mathfrak{d}(\hat{a}_{i,j}) = &\widehat{(\partial_{\slashed{\mathbf{S}}})}_{i,j} + \sum_{1 <h < k^1_{\sigma_s(a)}} \hat{e}_{i,h}^{\sigma_s(a)} a_{0;h,j}  + \sum_{i<h < k^1_{\sigma_s(a)}} e_{1; i,h}^{\sigma_s(a)} \hat{a}_{h,j}+ \\
    & \sum_{1<h<k^0_{\sigma_e(a)}} \hat{a}_{i,h}e^{\sigma_e(a)}_{0; h,j} + \sum_{1 \le h \le j}  a_{1; i,h}\hat{e}^{\sigma_e(a)}_{h,j}.
\end{align*}
\begin{lemma}\label{elimination of chat and mhat}
  There is a chain map $\hat{p} \colon \underline{\op{Cth}}^{\textit{ab}}(\mathbf{C}_0, \mathbf{C}_1) \to \vec{\mathcal S}^+_{\slashed{\mathbf{C}}_1, \slashed{\mathbf{C}}_0}$ which induces an isomorphism in homology and makes the following diagram commutes
  $$\xymatrix{
    {\mathcal D}_{\mathbf{C}_1}^{\textit{ab}} \otimes   \underline{\op{Cth}}^{ab}(\mathbf{C}_0, \mathbf{C}_1) \otimes {\mathcal D}_{\mathbf{C}_0}^{\textit{ab}} \ar[r] \ar[d]_{p_1 \otimes \hat{p} \otimes p_0}  & \underline{\op{Cth}}^{ab}(\mathbf{C}_0, \mathbf{C}_1) \ar[d]_{\hat{p}} \\
    \vec{\mathcal A}^+_{\slashed{\mathbf{C}}_1} \otimes \vec{\mathcal S}^+_{\slashed{\mathbf{C}}_1, \slashed{\mathbf{C}}_0}\otimes \vec{\mathcal A}^+_{\slashed{\mathbf{C}}_1} \ar[r] & \vec{\mathcal S}^+_{\slashed{\mathbf{C}}_1, \slashed{\mathbf{C}}_0}
  }$$
  where the horizontal arrows are the bimodule multiplicatiosn and $p_i$ are the quotient maps.
\end{lemma}
\begin{proof}
  We define $\hat{p}$ by mapping all intersection points $c^\sigma_{d; i,j}, \hat{c}^{\sigma}_{i,j}$ and all chords $m^\sigma_{d; i,j}$ and $\hat{m}^\sigma_{i,j}$ to zero; then it is clear that $\hat{p}$ is a chain map and the diagram commutes. To prove that $\hat{p}$ induces an isomorphism in homology we proceeds in two steps. First we map $c^\sigma_{d; i,j}$ and $m^\sigma_{d; i,j}$ to zero, producing a map
  $$\hat{p}' \colon \underline{\op{Cth}}^{\textit{ab}}(\mathbf{C}_0, \mathbf{C}_1) \to \vec{\mathcal A}^+_{\slashed{\mathbf{C}}_1} \otimes_{{\mathcal D}_{\mathbf{C}_1}^{\textit{ab}}} \underline{\op{Cth}}^{\textit{ab}}(\mathbf{C}_0, \mathbf{C}_1)  \otimes_{{\mathcal D}_{\mathbf{C}_0}^{\textit{ab}}} \vec{\mathcal A}^+_{\slashed{\mathbf{C}}_1}.$$
  If we filter $\underline{\op{Cth}}^{\textit{ab}}(\mathbf{C}_0, \mathbf{C}_1)$ by action, the graded complexes on the left and on the right are directed sums of copies of ${\mathcal D}^{\textit{ab}}_{\mathbf{C}_1} \otimes {\mathcal D}^{\textit{ab}}_{\mathbf{C}_0}$ and  $\vec{\mathcal A}^+_{\slashed{\mathbf{C}}_1} \otimes  \vec{\mathcal A}^+_{\slashed{\mathbf{C}}_0}$ respectively, and the map induced by $\hat{p}$ between the graded complexes is $p_1 \otimes p_1$ on each summand. Since this map induces an isomorphism in homology by Lemma \ref{elimination of c and m}, it follows that $\hat{p}'$ also induces an isomorphism in homology.

  Then we define $\hat{p}'' \colon \vec{\mathcal A}^+_{\slashed{\mathbf{C}}_1} \otimes_{{\mathcal D}_{\mathbf{C}_1}^{\textit{ab}}} \underline{\op{Cth}}^{\textit{ab}}(\mathbf{C}_0, \mathbf{C}_1)  \otimes_{{\mathcal D}_{\mathbf{C}_0}^{\textit{ab}}} \vec{\mathcal A}^+_{\slashed{\mathbf{C}}_1} \to  \vec{\mathcal S}^+_{\slashed{\mathbf{C}}_1, \slashed{\mathbf{C}}_0}$ by sending all $\hat{c}^\sigma_{i,j}$ and $\hat{m}^\sigma_{i,j}$ to zero. Then $\hat{p}''$ induces an isomorphism in homology by an inductive elimination argument for the $\hat{c}^\sigma_{i,j}$ and $\hat{m}^\sigma_{i,j}$ as in Lemma \ref{elimination of c and m}.
\end{proof}

\section{Morphisms of representations from Floer homology} 
\label{sec: morphisms of rep}
In this section we put everything together to prove Theorem \ref{thm: main}. Let $L$ be closed exact Lagrangians in $W$. From Theorem \ref{thm:maingeometric} we obtain an immersed Lagrangian submanifolds $\overline{L}$ such that the Legendrian lifts $L^\dagger$ and $\overline{L}^\dagger$ are Legendrian isotopic, and a decomposition  $\overline{L} = \Sigma \cup \mathbf{C}$ where $\Sigma_i \subset W^{sc}$ is an immersed filling of a Legendrian link $\bs{\Lambda}$ and $\mathbf{C}$ is a standard cap.

\begin{lemma}\label{Sigma induces an augmentation} $\Sigma$ induces an augmentation $\varepsilon_{\Sigma} \colon {\mathcal D}_{\mathbf{C}} \to \mathbb{F}$. 
\end{lemma}
See \cite{Panrutherford} for a similar result.
\begin{proof}
  Since $L^\dagger$ has no Reeb chords its Chekanov-Eliashberg algebra admits the trivial augmentation.  Let ${\mathcal D}_{\overline{L}}$ be the cobordism algebra of $\overline{L}$, which in this case coincides with the Chekanov-Eliashverg algebra over $\F$ of $\overline{L}^\dagger$.  The existence of augmentations is a Legendrian isotopy invariant, and  therefore there exists an augmentation $\varepsilon \colon {\mathcal D}_{\overline{L}} \to \mathbb{F}$. Let
  $$\Phi_{\Sigma} \colon {\mathcal D}_{\mathbf C} \to {\mathcal D}_{\overline{L}}$$
  be the dga map from Definition \ref{dga morphisms induced by cobordisms}. Then  $\varepsilon_{\Sigma}= \varepsilon \circ \Phi_\Sigma$ is an augmentation of ${\mathcal D}_{\mathbf C}$.
\end{proof}
We define an ${\mathcal A}_{\mathbf{S}}$-module $V_L$ as follows. The augmentation $\epsilon_\Sigma \colon {\mathcal D}_{\mathbf{C}} \to \F$ induces an augmentation $\varepsilon_{\Sigma} \colon \vec{\mathcal A}^+_{\slashed{\mathbf{C}}} \to \F$ by Lemma \ref{from D to Dab} and Lemma \ref{elimination of c and m}. This augmentation defined a  ${\mathcal A}_{\mathbf{S}}$-module $V_L$ by Corollary \ref{the point of all this mess 2}.
\begin{lemma}
  Let $V_L$ be the ${\mathcal A}_{\mathbf{S}}$-module associated to $L$ by the above construction. If $D_\sigma$ is the Lagrangian cocore of the Weinstein handle attached to $S_\sigma$ for $\sigma \in \pi_0(\mathbf{S})$ and $L$ intersects $D_\sigma$ transversely, then
  \begin{itemize}
 \item $\dim_{\F} (\sigma \cdot V_{L})=|L \cap D_\sigma|$, and 
\item $\chi(\sigma \cdot V_{L})=L \bullet D_\sigma$ when $L$ is endowed with an orientation.
\end{itemize}
\end{lemma}
\begin{proof}
  We recall that $V_L$ has as underlying vector space the ring $\kk_{\mathbf{C}}=H^0(\mathbf{C},\kk)$ with idempotents associated to the connected components of $\mathbf{C}$, and thus of $\bs{\Lambda}$. Since the underlying $\kk_{\mathbf{C}}$-module to $V_L$ is $\kk_{\mathbf{C}}$,  $\dim_\F \sigma \cdot V_L=k_\sigma$, i.e.\ the number of components of $\bs{\Lambda}$ which are parallel to $S_\sigma$. By the construction of $\overline{L}$ in Section \ref{sec:geom-constr}, this is equal to the number of intersections between $L$ and $D_\sigma$. This proves the first part of the lemma. For the second part of the lemma, we need to show that the parity of the degree of the canonical basis element of $V_L=H^0(\mathbf{C},\kk)$ that corresponds to the component $C_{\sigma}^{(i)}$ is the same as the parity of the corresponding intersection point between $C_\sigma$ and $L$ which, in turn, is determined by the orientation of $L$. 
\end{proof}
This ends the proof of the first half of Theorem \ref{thm: main}. Now let $L_0$ and $L_1$ be two closed exact Lagrangian submanifolds in $W$ and, for $i=0,1$, let $\varepsilon_{\Sigma_i} \colon {\mathcal D}_{\mathbf{C}_i} \to \F$ be the corresponding  augmentations, and $V_{L_i}$ the induced ${\mathcal A}_{\mathbf{S}}$-modules.
\begin{lemma}\label{lemma: standard position}
  The construction of \ref{thm:maingeometric} can be performed so that the completions of $\widehat{\mathbf{C}}_0$ and $\widehat{\mathbf{C}}_1$ are in standard position   (see Section \ref{Cth(C,C)}).
\end{lemma}
\begin{proof}
  One should take two suitable Hamiltonian isotopic copies $C_0$ and $C_1$ of the core, push $L_0 \cap H$ to a small neighbourhood of $C_0$ and $L_1 \cap H$ a small neighbourhood of $C_1$, and perform the construction of \ref{thm:maingeometric} independently in those neighbourhoods.
\end{proof}
\begin{rem}
  Both the definition of standard position and Lemma \ref{thm:maingeometric} can be extended to any finite number of closed Lagrangian submanifolds.
\end{rem}

  \begin{lemma}\label{lemma: from floer to cthulhu}
    The Floer cohomology $HF(L_0, L_1)$ is isomorphic to the homology $H\op{Cth}^*_{\varepsilon_{\Sigma_0}, \varepsilon_{\Sigma_1}}(\mathbf{C}_0, \mathbf{C}_1)$.
  \end{lemma}

  \begin{proof}
We observe that the intersection points $\hat{c}^\sigma_{i,j}$ have positive action.  Then by Theorem \ref{relative exact triangle} there is an exact triangle
    $$\xymatrix{
      H\op{Cth}^*_{\varepsilon_{\Sigma_0}, \varepsilon_{\Sigma_1}}(\widehat{\mathbf{C}}_0, \widehat{\mathbf{C}}_1) \ar[rr] & & H\op{Cth}_{\varepsilon_0, \varepsilon_1}^*(\overline{L}_0, \overline{L}_1) \ar[dl] \\
      & H\op{Cth}_{\varepsilon_0, \varepsilon_1}^*(\Sigma_0, \Sigma_1). \ar[ul] &
    }$$
    Since $\Sigma_0$ and $\Sigma_1$ have no negative ends, we have $$H\op{Cth}_{\varepsilon_0, \varepsilon_1}^*(\Sigma_0, \Sigma_1) \cong HW^*((\Sigma_0, \varepsilon_0), (\Sigma_1, \varepsilon_1))$$
 by \cite[Appendix B.1.1]{ekholm-lekili}, and therefore $HW^*((\Sigma_0, \varepsilon_0), (\Sigma_1, \varepsilon_1))=0$ because wrapped Floer homology vanishes in subcritical Weinstein manifolds. See \cite[Section 6]{generation} for a definition of the wrapped Floer homology of immersed exact Lagrangian submanifolds and \cite[Section 7]{generation} for its vanishing in subcritical Weinstein manifolds. Thus $H\op{Cth}_{\varepsilon_0, \varepsilon_1}^*(\overline{L}_0, \overline{L}_1)$ is isomorphic to $H\op{Cth}^*_{\varepsilon_{\Sigma_0}, \varepsilon_{\Sigma_1}}(\widehat{\mathbf{C}}_0, \widehat{\mathbf{C}}_1)$.

    Next, $H\op{Cth}_{\varepsilon_0, \varepsilon_1}^*(\overline{L}_0, \overline{L}_1) \cong HF(L_0, L_1)$ because Cthulhu homology for closed Lagrangian manifolds is just Floer homology, which is invariant under regular homotopies that lift to Legendrian isotopies by \cite[Section 4.4]{generation}. 
  \end{proof}

  \begin{lemma}
    There is an isomorphism between $H\op{Cth}^*_{\varepsilon_{\Sigma_0}, \varepsilon_{\Sigma_1}}(\mathbf{C}_0, \mathbf{C}_1)$ and $H^*\op{Rhom}_{\mathcal A}(V_0, V_1)$.
  \end{lemma}
  \begin{proof}
    Let $\F_{\varepsilon_i}$ be the ${\mathcal D}^{\textit{ab}}_{\mathbf{C}}$-modules with underlying vector space $\F$ defined by the pull-back of $\epsilon_i$ to ${\mathcal D}_{\mathbf{C}}^{\textit{ab}}$. The first step of the proof is to show that there is an isomorphism $$H\op{Cth}^*_{\varepsilon_{\Sigma_0}, \varepsilon_{\Sigma_1}}(\mathbf{C}_0, \mathbf{C}_1) \cong H \hom_{{\mathcal D}^{\mathbf{ab}}_{\mathbf{C}_1}} (\underline{\op{Cth}}^{\textit{ab}}(\mathbf{C}_0, \mathbf{C}_1) \otimes_{{\mathcal D}^{\textit{ab}}_{\mathbf{C}_0}} \F_{\varepsilon_0}, \F_{\varepsilon_1}).$$
    We will prove it in two steps. We consider the intermediate bimodule
    $$\underline{\op{Cth}}^{\textit{abc}}(\mathbf{C}_0, \mathbf{C}_1) = {\mathcal D}^{\mathit{ab}}_{\mathbf{C}_1} \otimes_{{\mathcal D}_{\mathbf{C}_1}}\underline{\op{Cth}}(\mathbf{C}_0, \mathbf{C}_1) \otimes_{{\mathcal D}_{\mathbf{C}_0}} {\mathcal D}^{\mathit{ab}}_{\mathbf{C}_0}$$
    (the modifier $\mathit{abc}$ stands for ``abstract coefficients''). Then there is a tautological isomorphism 
    $$\op{Cth}_{\varepsilon_0, \varepsilon_1}(\mathbf{C}_0, \mathbf{C}_1) \cong
    \hom_{{\mathcal D}^{\mathit{ab}}_{\mathbf{C}_1}}(\underline{\op{Cth}}^{\textit{abc}}(\mathbf{C}_0, \mathbf{C}_1) \otimes_{{\mathcal D}^{\mathit{ab}}_{\mathbf{C}_1}} \F_{\epsilon_0}, \F_{\epsilon_1})$$
    coming from the properties of tensort product. Moreover $\underline{\op{Cth}}^{\mathit{ab}}(\mathbf{C}_0, \mathbf{C}_1)$ and $\underline{\op{Cth}}^{\mathit{abc}}(\mathbf{C}_0, \mathbf{C}_1)$ are both free bimodules and are quasi-isomorphic (see the proof of Lemma \ref{elimination of chat and mhat}), and therefore $\hom_{{\mathcal D}^{\mathit{ab}}_{\mathbf{C}_1}}(\underline{\op{Cth}}^{\textit{abc}}(\mathbf{C}_0, \mathbf{C}_1) \otimes_{{\mathcal D}^{\mathit{ab}}_{\mathbf{C}_1}} \F_{\epsilon_0}, \F_{\epsilon_1})$ is quasi-isomorphic to $\hom_{{\mathcal D}^{\mathit{ab}}_{\mathbf{C}_1}}(\underline{\op{Cth}}^{\textit{ab}}(\mathbf{C}_0, \mathbf{C}_1) \otimes_{{\mathcal D}^{\mathit{ab}}_{\mathbf{C}_1}} \F_{\epsilon_0}, \F_{\epsilon_1})$. 

    Next, using the $\vec{\mathcal A}^+_{\slashed{\mathbf{C}}_i}$-module structure on $\F_{\varepsilon_i}$ induced by the morphism $\vec{\mathcal A}^+_{\slashed{\mathbf{C}_i}} \to {\mathcal D}^{\mathit{ab}}_{\mathbf{C}_i}$ defined in Lemma \ref{elimination of c and m} we obtain an isomorphism
    $$H \hom_{{\mathcal D}^{\mathbf{ab}}_{\mathbf{C}_1}} (\underline{\op{Cth}}^{\textit{ab}}(\mathbf{C}_0, \mathbf{C}_1) \otimes_{{\mathcal D}^{\textit{ab}}_{\mathbf{C}_0}} \F_{\varepsilon_0}, \F_{\varepsilon_1}) \cong H\hom_{\vec{\mathcal A}^+_{\slashed{\mathbf{C}}_1}}(\vec{\mathcal S}^+_{\slashed{\mathbf{C}}_0, \slashed{\mathbf{C}}_1} \otimes_{\vec{\mathcal A}^+_{\slashed{\mathbf{C}}_0}} \F_{\varepsilon_0}, \F_{\varepsilon_1})$$
    by Lemma \ref{elimination of chat and mhat} (or, rather, by its proof, since the elimination of the generators $\hat{c}^\sigma_{i,j}$ and $\hat{m}^\sigma_{i,j}$ can also be performed in  $$\hom_{{\mathcal D}^{\mathbf{ab}}_{\mathbf{C}_1}} (\underline{\op{Cth}}^{\textit{ab}}(\mathbf{C}_0, \mathbf{C}_1) \otimes_{{\mathcal D}^{\textit{ab}}_{\mathbf{C}_0}} \F_{\varepsilon_0}, \F_{\varepsilon_1})).$$

  We recall that
  $$\vec{\mathcal S}^+_{\slashed{\mathbf{C}}_1, \slashed{\mathbf{C}}_0} =  \vec{\mathcal A}^+_{\slashed{\mathbf{C}}_1}\otimes_{{\mathcal A}^+_{\slashed{\mathbf{C}}_1}} {\mathcal S}^+_{\slashed{\mathbf{C}}_1, \slashed{\mathbf{C}}_0} \otimes_{{\mathcal A}^+_{\slashed{\mathbf{C}}_0}} {\mathcal A}^+_{\slashed{\mathbf{C}}_0},$$
  and therefore
  $$\hom_{\vec{\mathcal A}^+_{\slashed{\mathbf{C}}_1}}(\vec{\mathcal S}^+_{\slashed{\mathbf{C}}_0, \slashed{\mathbf{C}}_1} \otimes_{\vec{\mathcal A}^+_{\slashed{\mathbf{C}}_0}} \F_{\varepsilon_0}, \F_{\varepsilon_1}) \cong \hom_{{\mathcal A}^+_{\slashed{\mathbf{C}}_1}}({\mathcal S}^+_{\slashed{\mathbf{C}}_0, \slashed{\mathbf{C}}_1} \otimes_{{\mathcal A}^+_{\slashed{\mathbf{C}}_0}} \F_{\varepsilon_0}, \F_{\varepsilon_1})$$
  by the naturality properties of tensor product. 

  Finally, Lemma \ref{after cast} gives an isomorphism
  $$hom_{{\mathcal A}^+_{\slashed{\mathbf{C}}_1}}({\mathcal S}^+_{\slashed{\mathbf{C}}_0, \slashed{\mathbf{C}}_1} \otimes_{{\mathcal A}^+_{\slashed{\mathbf{C}}_0}} \F_{\varepsilon_0}, \F_{\varepsilon_1}) \cong \hom_{\mathcal A}({\mathcal S} \otimes_{\mathcal A} V_0, V_1).$$
 Since ${\mathcal S}$ is a semi-projective resolution of the diagonal bimodule by Lemma \ref{the short resolution at last}, we have
  $$H^*\hom_{\mathcal A}({\mathcal S} \otimes_{\mathcal A} V_0, V_1)= H^*\op{Rhom}_{\mathcal A}(V_0, V_1).$$
\end{proof}
This ends the proof of the second half of Theorem \ref{thm: main}.

\section{Proof of Corollary \ref{cor:main}}

We begin with some standard results from homological algebra. A dga $\mathcal{A}$ with a choice of unital dg-morhism  $\kk_{\mathbf{S}} \to \mathcal{A}$, where the differential of the domain is trivial, is called a $\kk_{\mathbf{S}}$-dga. A morphism $\mathcal{B} \to \mathcal{A}$ of $\kk_{\mathbf{S}}$-dgas, or $\kk_{\mathbf{S}}$-dg morphism, is a dg-morphism in the usual sense that commutes with the canonical choices of inclusions of $\kk_{\mathbf{S}}$. In the following section all dgas will be assumed to be $\kk_{\mathbf{S}}$-dgas, and all dg-morphisms will be assumed to be $\kk_{\mathbf{S}}$-dg morphisms and $\Z$-graded, unless stated otherwise.
\begin{prop}
 
  Assume that $\mathcal{A}$ is a $\Z$-graded $\kk_{\mathbf{S}}$-dga which satisfies $H_i(\mathcal{A})=0$ for all $i<0$ and $H_0(\mathcal{A})=\kk_{\mathbf{S}}$. Then there exists a semi-projective $\Z$-graded $\kk_{\mathbf{S}}$-dga $\mathcal{B}$ and a quasi-isomorphism $\Phi \colon \mathcal{B}\xrightarrow{q.is.}\mathcal{A}$ of $\kk_{\mathbf{S}}$-dgas, where all generators of $\mathcal{B}$ have strictly positive degrees, i.e.~$\mathcal{B}_i=0$ for all $i<0$, $\mathcal{B}_0=\kk_{\mathbf{S}}$, and $\partial|_{\mathcal{B}_1}=0$.
\end{prop}
\begin{proof}
  The dga $\mathcal{B}$ is constructed inductively by using the degree-filtration. We start by defining $\mathcal{B}^0=\kk_{\mathbf{S}}$ and note that there is a dga-morphism $\mathcal{B}^0 \to \mathcal{A}$ which is an isomorphism in homology for all degrees $i \le 0$.

  Now assume that we have managed to construct a dga $\mathcal{B}^j$ and a morphism $\Phi^j \colon \mathcal{B}^j \to \mathcal{A}$ which satisfies the conclusions of the lemma, except that $\Phi^j$ is only an isomorphism in the $i$-degree homology groups for $i \le j$. After the addition of a suitable number of free generators in degree $j+1$ to $\mathcal{B}^j$, yielding a new semi-projective dga $\tilde{\mathcal{B}}^j \supset \mathcal{B}^j$, a suitable lift of the dg-morphism $\Phi^j$ to $\tilde{\Phi}^j \colon \tilde{\mathcal{B}}^j \to \mathcal{A}$ can be constructed that, in addition to the above properties of $\Phi^j$, also is surjective in homology of degree $j+1$. The kernel of the map $[\tilde{\Phi}] \colon H_{j+1}(\tilde{\mathcal{B}}^j) \to H_{j+1}(\mathcal{A})$ can be represented by cycles that we can kill by adding generators in degree $j+2$ to yield an extension $\mathcal{B}^{j+1} \supset \tilde{\mathcal{B}}^j$. The morphism $\Phi^{j+1} \colon \mathcal{B}^{j+1} \to \mathcal{A}$ is constructed by extending $\tilde{\Phi}^j$ by zero on the latter generators in degree $j+2$.

  Since we can assume that $(\mathcal{B}^{j+1})_i=(\mathcal{B}^j)_i$ is satisfied for $i \le j$ in the construction, while $\Phi^{j+1}|_{(\mathcal{B}^{j+1})_i}=\Phi^j$ for $i\le j$, the sought dga and morphism $\mathcal{B} \to \mathcal{A}$ can be constructed as the limit.
\end{proof}
\begin{lemma}
  \label{lemma:cycle splitting}
    Let $\mathcal{B}$ be a $\Z$-graded $\kk_{\mathbf{S}}$-dga which satisfies $\mathcal{B}_i=0$ for all $i<0$, $\mathcal{B}_0=\kk_{\mathbf{S}}$, and $\partial|_{\mathcal{B}_1}=0$. For any finite dimensional $\Z$-graded dg $\mathcal{B}$-module $\tilde{V}$ with non-trivial homology $H(\tilde{V}) \neq 0$, there is a quasi-isomorphic module $V$ for which there is some $i_0 \in \Z$ such that
    \begin{itemize}
    \item $V_i =0$ for all $i < i_0$;
    \item $s\cdot V_{i_0} \neq 0$ for some $s \in \mathbf{S}$; and 
    \item $V_{i_0+1}$ are all cycles, i.e.~the differential satisfies $d(V_{i_0+1})=0$.
    \end{itemize}
     In particular, we have $H(V_{i_0})=V_{i_0} \neq 0$.
    \end{lemma}
    \begin{proof}

      Since $H(\tilde{V}) \neq 0$ and $\tilde{V}$ is finite dimensional we can find some $\tilde{i}_0 \in \Z$ so that $\tilde{V}_i=0$ for all $i<j_0$ and $s\cdot \tilde{V}_{j_0} \neq 0$ for some $s \in \mathbf{S}$. If $d(\tilde{V}_{j_0+1}) \neq 0$ then we chose some (possibly zero dimensional) $\F$-subspace $H \subset \tilde{V}_{j_0}$ complementary to $d(\tilde{V}_{j_0+1})$, and consider the $\F$-subspace
      $$M \coloneqq H \oplus \left(\tilde{V}_{j_0+1} \cap \ker d  \right) \oplus \bigoplus_{i=j_0+2}^\infty \tilde{V}_i \subset \tilde{V}.$$
      Because of the assumptions on $\mathcal{B}$, the $\F$-subspace $M \subset \tilde{V}$ is actually a dg-submodule. Indeed, for any $x \in H$ and $b\in \mathcal{B}$, $b\cdot x \in \tilde{V}_{j_0}$ implies $b \cdot x \in \F x$ sinc $\mathcal{B}_0=\kk_{\mathbf{S}}$, while $b\cdot x \in \tilde{V}_{j_0+1}$ implies that
      $$d(b\cdot x)=\partial(b)\cdot x+b\cdot \partial x=0$$
      since $\partial|_{\mathcal{B}_1}=0$.

      One can readily check that the inclusion $M \subset \tilde{V}$ is a quasi-isomorphism of $\mathcal{B}$-modules. If $H \neq 0$ then we are done and can take $V=M$. In the case when $H=0$, we can repeat the argument with $\tilde{V}$ replaced by $M$.  Since $H(\tilde{V}) \neq 0$ and $\dim \tilde{V} < \infty$, this process must terminate.
  \end{proof}

The following result is the main algebraic mechanism behind the conclusion of Corollary \ref{cor:main}.

  \begin{prop}
    \label{prop:trivialext}
    Let $\mathcal{B}$ be a $\Z$-graded $\kk_{\mathbf{S}}$-dga which satisfies $\mathcal{B}_i=0$ for all $i<0$, $\mathcal{B}_0=\kk_{\mathbf{S}}$, and $\partial|_{\mathcal{B}_1}=0$. If $V$ is a finite dimensional $\Z$-graded $\mathcal{B}$-module for which Ext-group $H^0(\op{Rhom}_{\mathcal{B}}(V,V))=\F$, then $V$ is quasi-isomorphic to a complex supported in a single degree and $H(s\cdot V)$ are at most one-dimensional for each $s \in \mathbf{S}$.
  \end{prop}
  \begin{proof}
    First, using Lemma \ref{lemma:cycle splitting} we can, by replacing $V$ by a quasi-isomorphic module, restrict ourselves to the case when $V_i=0$ for $i \le i_0$ while $V_{i_0} \neq 0$ and consists of cyles that inject into $H(V)$.

    It follows that there is a chain map in $\op{Hom}_{\mathcal{B}}(V,V)$ that vanishes on $V_{i}$ for all $i \ge i_0+1$ and which is the identity on $V_{i_0}$. Since $H^0(\op{Rhom}_{\mathcal{B}}(V,V))=\F$ is one-dimensional, and since the previously constructed map is non-trivial as a map in homology $H(V) \to H(V)$, the latter map is non-zero multiple of the identity $\id_{H(V)}$ in homology. This implies that the quotient $V \to V_{i_0}$ is a quasi-isomorphism.

What remains is to show that each $H(s\cdot V)$ is one-dimensional. By the previous paragraph, we can replace $V$ by a quasi-isomorphic version that is supported in a single degree. If $\dim(s \cdot V) > 1$ for some $s \in \mathbf{S}$, then one easily constructs a non-trivial endomorphism of $\op{Hom}_{\mathcal{B}}(V,V)$, which is automatically a chain map, and which is not equal to a multiple of the identity in homology. This contradicts that fact the zero:th Ext group $H^0(\op{Rhom}_{\mathcal{B}}(V,V))=\F$ is one-dimensional.
\end{proof}

We are now ready to prove Corollary \ref{cor:main}. Since $L$ is connected we have $H^0(\op{Rhom}_{\mathcal{A}}(V_L,V_L))=HF^0(L,L)=H^0(L)=\F$ by Theorem \ref{thm: main}. Using the above results, we get a dga $\mathcal{B}$ which is quasi-isomorphic to $\mathcal{A}$ via $\Phi \colon \mathcal{B} \xrightarrow{q.is} \mathcal{A}$, and which together with $V=\Phi^*V_L$ satisfies the assumptions of Proposition \ref{prop:trivialext}. Since $\op{Rhom}_{\mathcal{A}}(V_L,V_L)$ is quasi-isomorphic to $\op{Rhom}_{\mathcal{B}}(\Phi^*V_L,\Phi^*V_L)$, this concludes the proof of Corollary \ref{cor:main}.

\color{black}

\appendix
\section{An alternative approach: reducing to a contactisation}\label{sec:reduc-cont-stopp}
Several technical complications that arise in our setting --- and in particular the need for direct limits --- stem from the fact that the attaching link $\mathbf{S}$  may have (and in fact it is aspected to always have) infinite many Reeb chords. Recall that the completion $\widehat{W}^{sc}$ is exact symplectomorphic to a product Weinstein manifold $\widehat{P} \times \C$ for some (not uniquely determined) completion $\widehat{P}$ of a Weinstein domain $(P, \eta)$; see \cite{SteinWeinstein}.  Karlsson in \cite{Karlsson} proved that if $P$ itself is subcritical, i.e.~$\widehat{W}^{sc}=\widehat{Q}\times \C^2$ where $\widehat{P}=\widehat{Q}\times \C$, or equivalently, $W^{sc}$ has a handle decomposition with handles of index at most $n-2$, any Legendrian submanifold inside $\partial {W}^{sc}$ has a Chekanov--Eliashberg algebra of finite type in the following sense: for a particular choice of contact form, there entire Chekanov--Eliasherg algebra is quasi-isomorphic to a sub-dga generated by Reeb chords of small length. (There might exist arbitrarily long Reeb chords, but they can be ignored.)

If, on the other hand, $P$ is not subcritical, we can introduce a stop in $\partial W^{sc}$ disjoint from $\mathbf{S}$ so that the constructions and computations of the previous sections can be carried out in the symplectisation of a contactisation, where $\mathbf{S}$ has only finitely many chords. Thus the analysis needed can be reduced to that from \cite{LCHgeneral} and \cite{Duality_EkholmetAl} and, additionally, the differential graded algebras obtained will all be finitely generated. The price to pay will be that the algebra which will replace ${\mathcal A}_{\mathbf{S}}$ will be invariant only up to Legendrian isotopies of $\mathbf{S}$ in $\partial W^{sc}$ which do not intersect the stop.

The product decomposition $\widehat{W}^{sc}= \widehat{P} \times \C$ induces an open book decomposition of $\partial W^{sc}$ with page $(P, \eta)$ and trivial monodromy. That is, we can write
$$\partial W^{sc}= P \times S^1 \cup \partial P \times D^2$$
where the contact form on $\partial W^{sc}$ coincides with $\eta+d z$ on $P \times S^1$ and with $\eta|_{\partial P} + r^2d z$ on $\partial P \times D^2$ (here we take $z \in [- \pi, \pi]$ as coordinate in $S^1$). We denote by $B= \partial P \times \{0 \} \subset \partial P \times D^2 \subset \partial W^{sc}$ the binding.

Any Legendrian submanifold of a contact manifold endowed with a compatible open book decomposition can be made disjoint from a given page (including the binding) after a Legendrian isotopy; see Akbulut--Arikan \cite{AkbulutArikan}.In other words, we can assume that the Legendrian link $\mathbf{S}$ is contained in a subset
$$ P^o \times (-\pi+ \epsilon, \pi - \epsilon) \subset P^o \times S^1 \subset \partial W^{sc}$$
where $P^o = P \setminus \partial P$ and $\epsilon >0$ is small. Let $\mathfrak{p} =  P \times \{ \pi \}$ be a page of the open book decomposition. By \cite[Example 2.19]{GanatraPardonShende} there is a Weinstein sector $W^{sc}_{\mathfrak{p}}$ associated to the Weinstein pair $(W^{sc}, \mathfrak{p})$ whose end is modelled on the symplectisation of $(P \times [- \pi+\epsilon, \pi-\epsilon], dz+ \eta)$. Attaching critical handles along $\mathbf{S}$ produces a Weinstein sector $W_{\mathfrak{p}}$.

Let ${\mathcal A}_{\mathbf{S}}^0$ be the Chekanov-Eliashberg algebra of $\mathbf{S}$ as a Legendrian submanifold of the contactisation $P \times  [- \pi+\epsilon, \pi-\epsilon]$. With some more care we could prove that ${\mathcal A}_{\mathbf{S}}^0$ is isomorphic to the sub-dga of ${\mathcal A}_{\mathbf{S}}$ generated by Reeb chords which are disjoint from $\mathfrak{p}$. For every closed Lagrangian submanifolds $L_0$ and $L_1$, their Floer homology in $W$ and in $W_{\mathfrak{p}}$ are tautological isomorphic  because Floer homology between compact Lagrangian submanifolds is not affected by what happens near the boundary. Moreover, the constructions of the previous sections can be performed in $W_{\mathfrak{p}}$ instead of in $W$, and therefore we obtain the following result with a similar, but easier proof.
\begin{thm}\label{thm: main appendix}
  To any closed exact Lagrangian submanifold $L \subset W$  which intersects all cocores of the criitical Weinstein handles transversely we associate a differential graded ${\mathcal A}_{\mathbf{S}}^0$-module $V_L$ such that
\begin{itemize}
   \item $\dim_{\F} (\sigma \cdot V_{L})=|L \cap D_\sigma|$, and
  \item $\chi(\sigma \cdot V_{L})=L \bullet D_\sigma$ when $L$ is oriented,
  \end{itemize}
 for all $\sigma \in \pi_0(\mathbf{S})$.
 Moreover, given two closed exact Lagrangian submanifolds $L_0$ and $L_1$ as above, the isomorphism
  $$HF^*(L_0, L_1) \cong H^*\op{Rhom}_{{\mathcal A}_{\mathbf{S}}^0}(V_{L_0}, V_{L_1})$$
  holds. 
\end{thm}

\section{Gradient flow-trees on multiple copies of the core}
\label{appendix:morse}
Ekholm's theory of gradient flow-trees \cite{MorseFlow} is an efficient tool for finding the rigid pseudoholomorphic discs with boundary on a closed exact Lagrangian immersion that is contained inside a cotangent bundle. In the setting of immersed exact Lagrangian cobordisms of dimension two inside the symplectisation of a jet space, Ekholm--Honda--Kalman \cite{Ekhoka} adapted this technique to the count of pseudoholomorphic discs with strip-like ends as considered in the setting of SFT. Note that the symplectisation of $J^1M$ is symplectomorphic to $T^*(\R \times M)$.  The same technique also works in our setting, giving us means for computing small pseudoholomorphic discs in the cobordism $\widehat{W}^{crit}$ with boundary on multiple Hamiltonian isotopic copies of a completed core plane. Since the focus in \cite{Ekhoka} was on the case of Lagrangian surfaces in symplectisations of jet-spaces, we here give an account of their results adapted to the setting considered here.

\subsection{Gradient flows trees}
We start by recaling some general background and definitions of gradient flow trees.  The technique of gradient flow trees can be applied to (the Lagrangian projection of) a Legendrian submanifold  inside a jet-space, e.g.~$J^1\R^n$, if its front is generic and has singularities that are only cusp-edges. Here, however, we will be interested in the simpler case where Legendrian submanifold in $J^1\R^n$ will be a union of one-jets $j^1 f_i$ of globally defined smooth functions, and in particular the front will have no singularities at all  (except possibly intersections of different sheets). 

A gradient flow-tree for a Legendrian in $J^1\R^n$ is an immersed tree in the base $\R^n$ of the jet-space that satisfies particular lifting conditions with respect to the Legendrian.   More formally:
\begin{itemize}
\item \textit{Edges:} Each edge in the tree is associated to an unordered pair of sheets $\{j^1f_i,j^1f_j\}$ of the Legendrian, such that the edge is a non-constant gradient flow-line of the function differences $\pm(f_i-f_j)$, with no preferred orientation. For a fixed choice of ordering $(j^1f_i,j^1f_j)$ of the pair of sheets, the edge becomes endowed with a natural orientation as the negative gradient flow-line of $-\nabla(f_i-f_j)$, and a natural injective parametrisation by a connected, possibly infinite, open interval in $\R$ (well-defined up to translation). For any ordered pair we have a preferred lift of the edge to an oriented curve on the sheet $j^1f_i$.
\item \textit{Vertices:} There are matching conditions at the vertices of the tree that can be described as follows. Each edge corresponding to a pair $\{j^1f_i,j^1f_j\}$ gives an oriented curve in each of the two sheets $j^1f_i$ and $j^1f_j$ by the above. We require that these oriented curves close up at the vertices to form a piecewise smooth \emph{closed} oriented curve when projected to $T^*\R^n$. We say that there is a {\em puncture} at the vertex if the corresponding curve in $J^1\R^n$ has a discontinuity.
\end{itemize}
Since we consider the case when the front projection is an immersion,  all 1-valent vertices must correspond to critical points of $f_i-f_j$, i.e. punctures  of the flow-tree.   The relevance of these trees is the main result from \cite{MorseFlow} which states that there is a bijective correspondence of rigid gradient flow-trees and punctured discs in the cotangent bundle $T^*\R^n$ that have boundary on the Lagrangian projection, and which are pseudoholomorphic for a suitable choice of almost complex structure.

\subsection{Gradient flow trees for cobordisms via Morsifications}
\label{sec:morsification}
   The goal is to apply the technique of Morse flow-trees for finding the pseudoholomorphic discs with boundary on the   Lagrangian immersion $\widehat{\mathbf{C}}$ inside a small Weinstein neighbourhood of a single\footnote{In this appendix we can assume without loss of generality that all components of $\mathbf{C}$ are close to the same cores because everything happens in a small neighbourhood.} completed Lagrangian core $\widehat{C}$.   We identify the neighbourhood of $\widehat{C}$   with $T^*_{<\epsilon}\widehat{C}=T^*_{<\epsilon}\R^n$  for some suitable choice of metric.    In addition, we choose the metric so that $T^*_{<\epsilon}\R^n$ becomes endowed with a concave cylindrical end of the form $(-\infty,0] \times U$ for some neighbourhood $U \subset J^1S^{n-1}$ of the zero-section $j^10$. More precisely, we assume that the latter symplectisation is identified with subset of $T^*_{<\epsilon}\R^n$ via the canonial inclusion  
  $$(-\infty,0] \times J^1S^{n-1} \cong T^*((-\infty,0]\times S^{n-1}) \hookrightarrow T^*\R^n$$
  induced by the standard proper embedding $(-\infty,0]\times S^{n-1} \subset \R^n$. Furthermore, we assume that $\widehat{\mathbf{C}}$ is identified with the cylindrical Lagrangian $(-\infty,0] \times \boldsymbol{\Lambda} \subset (-\infty,0]\times J^1S^{n-1}$, while $\widehat{C}$ is equal to the zero-section $(-\infty,0] \times j^10$. Observe that the pseudoholomorphic discs in $T^*_{<\epsilon}\widehat{C}$ with boundary on $\widehat{\mathbf{C}}$ can have punctures at double-points of the immersion, and Reeb chord asymptotics to small Reeb chords on the Legendrian end $\boldsymbol{\Lambda}$. 

 The original version of the technique of Morse flow-trees was developed for compact Lagrangians. In order to apply this theory in this non-compact setting with cylindrical ends, we follow the same set-up as \cite{Ekhoka}. There, the counts of non-compact discs with strip-like ends were related to gradient flow-trees for an auxiliary \emph{immersed deformation} of the Lagrangian that has self-intersections that correspond to Reeb chords at the Legendrian ends.   In Section \ref{sec:morsification} below we recall this deformation, which is the \emph{Morsification procedure} from \cite[Section 2.3]{Ekhoka}. When this procedure is applied to $\widehat{\mathbf{C}}$, an immersed exact Lagrangian cobordism $\widetilde{\mathbf{C}} \subset T^*_{<\epsilon} \widehat{\mathbf{C}} \subset \widehat{W}^{crit}$ is produced, which is obtained by deforming the former in the negative end. 
 The produced cobordism was called a \emph{Morse cobordism} in the latter article; here we call $\widetilde{\mathbf{C}}$ the \emph{Morsification of $\widehat{\mathbf{C}}$}.  An important feature of the Morsification is that the set of double points that are created in the negative end are in a graded canonical bijection with the set of   \emph{short}   Reeb chords on $\boldsymbol{\Lambda} \subset J^1S^{n-1}$ that are contained in the negative end of $T^*_{<\epsilon} \widehat{\mathbf{C}}=T^*_{<\epsilon} \R^n.$ The double point corresponding to a  short   Reeb chord $\gamma$ will be denoted by $\tilde\gamma$. The construction in the present setting is carried out below, and we refer to Figure \ref{fig:morse} for an example.  In particular, we obtain a natural correspondence between the asymptotic constraints for discs with boundary on $\widehat{\mathbf{C}}$ and punctures asymptotic to only \emph{small} Reeb chords, and discs with boundary on the Morsification $\tilde{\mathbf{C}}$.

 The reason for passing to the Morsification cobordism is that we can apply the original theory of gradient flow-trees to find the rigid compact pseudoholomorphic discs with double-point asymptotics. The idea of \cite{Ekhoka} can thus be summarised as using gradient flow-trees to find the compact pseudoholomorphic discs with boundary on the Morsification $\widetilde{\mathbf{C}}$ that are rigid and, then, showing that these discs correspond to discs with boundary on the original cobordism $\widehat{\mathbf{C}}$ that are allowed to have Reeb chord asymptotics and that are rigid. Furthermore, under this identification, each asymptotic to the double point $\tilde\gamma$ of a disc with boundary on $\widetilde{\mathbf{C}}$ corresponds to a non-compact end asymptotic to the Reeb chord $\gamma$ on the corresponding disc with boundary on $\widehat{\mathbf{C}}$.

 We now proceed with the construction of the Morsification $\widetilde{\mathbf{C}}$ of $\widehat{\mathbf{C}}$ in the present setting. The first step is to choose a Legendrian lift of $\widehat{\mathbf{C}}$ to $J^1\R^n$.  Recall this lift consists of a number of sheets $j^1\hat f_i$ of one-jets of globally defined functions $\hat f_i \colon \R^n \to \R$. We make the further assumption that these lifts are chosen so that $\hat f_1< \hat f_2 < \ldots<\hat f_k$ is satisfied at the negative end, and such that these functions moreover all tend to zero as $x \in \R^n$ tends to $\infty$. In particular, the front projection of the Legendrian lift is immersed, which means that the gradient flow-trees can be applied, and that their possible vertices only are 1 and 2-valent punctures and 3-valent so-called $Y_0$-vertices; see \cite{MorseFlow} for a description.  Recall that we here only are concerned with the discs confined to a small neighbourhood of $\widehat{\mathbf{C}}$, and which are asymptotic to either double points or short Reeb chords on the Legendrian end.

 The Morsification $\widetilde{\mathbf{C}}$ can now be constructed by wrapping each sheet sufficiently by the negative Reeb flow as $t \to -\infty$. More precisely, we wrap the $i$:th sheet (which corresponds to $j^1 \hat{f}_i$) inside $(-\infty,k-i-1] \times J^1S^{n-1}$ sufficiently far, so that it wraps past precisely the sheets indexed by $j<i$. This is done by applying a symplectomorphism to the $i$:th sheet that is generated by a Hamiltonian of the form $\rho(t)$ where $\rho(t) \le 0$ is constant and negative inside $(-\infty,k-i-2/3]$, while it vanishes for $t \ge k-i-1/2$. See Figure \ref{fig:morse} for this wrapping applied to $\widehat{\mathbf{C}}=\widehat{\mathbf{C}}_0 \cup \widehat{\mathbf{C}}_1$. In the presence of pure chords of the components of the Legendrian $\boldsymbol{\Lambda} \subset J^1S^{n-1}$ the Morsification procedure is slightly more complicated; we refer to \cite{Ekhoka} for the general construction. Even if there are long pure chords on $\boldsymbol{\Lambda} \subset \partial W^{sc}$, they do not matter for the gradient flow-tree analysis here, since it takes place in a small neighbourhood of $\widehat{C}$ in which there are only short mixed chords.
 
Note that the wrapping produces new double points contained in $(-\infty,0] \times \partial W^{sc}$ that are in bijective correspondence with the Reeb chords on $\boldsymbol{\Lambda} \subset J^1S^{n-1}$. This correspondence can moreover be seen to be grading preserving; see \cite{Ekhoka} or \cite{LiftingPseudoholomorphic}. For a Reeb chord $\gamma$ on $\boldsymbol{\Lambda}$ we will denote by $\tilde{\gamma}$ the corresponding double point on $\widetilde{\mathbf{C}}$. Moreover, for every topological type ${\mathcal D}(c^+, \bs{\gamma})$ of punctured discs with boundary on $\widehat{\mathbf{C}}$, a unique boundary puncture positively asymptotic to a double point $c_+$ and boundary punctures negatively asymptotics to a word $\bs{\gamma}$ of wither double points of $\widehat{C}$ or Reeb chords of $\bs{\Lambda}$, there is a corresponding topological type ${\mathcal D}(c^+, \widetilde{\bs{\gamma}})$ of topological discs in $\widetilde{\mathbf{C}}$ with asymptotics at $c_+$ and $\widetilde{\bs{\gamma}}$.

 \begin{figure}[htp]
  \vspace{3mm}
         \labellist
         \pinlabel $1$ at -5 32
         \pinlabel $2$ at -5 25
         \pinlabel $\color{blue}1$ at -5 18
         \pinlabel $\color{blue}2$ at -5 10
         \pinlabel $1$ at -5 77
         \pinlabel $2$ at -5 86
         \pinlabel $\color{blue}2$ at -5 115
         \pinlabel $\color{blue}1$ at -5 98
         \pinlabel $\color{blue}\tilde{\mathbf{C}}_1$ at 163 102
         \pinlabel $\tilde{\mathbf{C}}_0$ at 163 80
         \pinlabel $\tilde{\mathbf{C}}_0$ at 163 35
         \pinlabel $\color{blue}\tilde{\mathbf{C}}_1$ at 163 50
         \pinlabel $t$ at 62 2
         \pinlabel $0$ at 49 -5
         \pinlabel $0$ at 110 56
         \pinlabel $t$ at 93 62
         \pinlabel $\color{red}\tilde{e}_{12}$ at 120 12
         \pinlabel $\color{red}e_{12}$ at 99 13
         \endlabellist
        \includegraphics[scale=1.75]{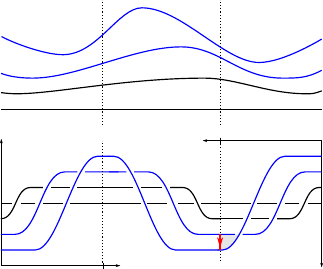}
        \vspace{5mm}
        \caption{A depiction of the Morsification $\widetilde{\mathbf{C}}_0 \cup \widetilde{\mathbf{C}}_1$ of $\widehat{\mathbf{C}}_0 \cup \widehat{\mathbf{C}}_1$, where the latter cobordism is the cylindrical extension of the cobordism shown in the middle of the picture. Top: The front projection (c.f.~Figure \ref{fig:profile}). Bottom: the Lagrangian projection in dimension one. Here we see the correspondence between the Reeb chord $e_{12}$ which is contained inside a contact slice, and the double point $\tilde{e}_{12}$. 
   In both the top and bottom picture, the middle depicts the part of the Legendrian that is contained inside the cotangent bundle of the core, while the left and right side are contained inside the cylindrical part; the latter subset is where the Morsification procedure gives rise to the double points.}
        \label{fig:morse}
      \end{figure}
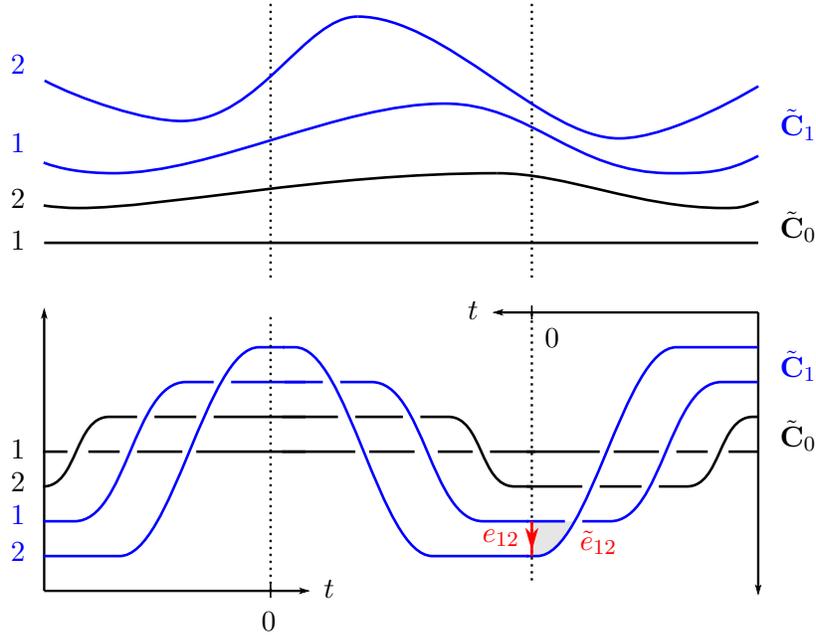

 \subsection{The correspondence between flow trees and holomorphic discs}
We can now formulate the main result of this appendix, which is the following bijection between counts of holomorphic discs in the setting described above, and counts of the corresponding gradient flow trees. The result is derived from \cite[Theorem 5]{Ekhoka}, which focused on the case of two-dimensional embedded Lagrangian cobordisms.
\begin{thm}
\label{thm:morse}
Consider a topological type $\mathcal{D}=\mathcal{D}_{\widehat{\mathbf{C}}}(c_+;\boldsymbol{\gamma})$ of punctured discs in $\widehat{W}^{sc}$ with boundary on $\widehat{\mathbf{C}}$ that satisfies the following properties:
\begin{itemize}
\item 
  there is a unique positive boundary puncture at the double point $c_+$ (with respect to our choice of Legendrian lift), while the remaining boundary puncture asymptotics $\boldsymbol{\gamma}$ are allowed to be either double points of $\widehat{\mathbf{C}}$ or Reeb chords on $\boldsymbol{\Lambda}$;
\item the expected dimension of the moduli space of pseudo-holomorphic discs in $\mathcal{D}$ is zero;
 \item All Reeb chord asymptotics are short chords on $\boldsymbol{\Lambda}$; and
\item there exists no nodal disc with boundary on $\widetilde{\mathbf{C}}$  which can be smoothed to a disc in ${\mathcal D}(c_+, \widetilde{\bs{\gamma}})$  and for which all components have positive energy and expected dimensions at least $-1$.
\end{itemize}
Then, there is an equality of signed counts
$$\#\mathcal{M}_{\widehat{\mathbf{C}},J}(c_+;\boldsymbol{\gamma})=\#\mathcal{T}_{\widetilde{\mathbf{C}}}(c_+;\widetilde{\boldsymbol{\gamma}})$$
of $J$-holomorphic discs $\mathcal{M}_{\widehat{\mathbf{C}},J}(c_+;\boldsymbol{\gamma}) \subset \mathcal{D}$ of the speficied topological type, for any generic almost complex structure $J$ that is cylindrical outside of a compact subset, and rigid gradient flow trees $\mathcal{T}_{\widetilde{\mathbf{C}}}(c_+;\widetilde{\boldsymbol{\gamma}})$ on the Morsification $\widetilde{\mathbf{C}}$ with the correponsing topological type, i.e.~where the asytmptotics to a Reeb chord $\gamma$ has been replaced by an asymptotic to the corresponding double point $\tilde{\gamma}$.
\end{thm}
\begin{rem}
In \cite[Section 4.3]{Ekholm_FloerlagCOnt} a conjectural extension of the theory of gradient flow-trees directly to the setting of non-compact exact Lagrangian cobordisms in the SFT-sense has been proposed, i.e.~without passing to Morse cobordisms.
\end{rem}

\begin{proof}[Proof of Theorem \ref{thm:morse}]
First, there is a bijective correspondence between the rigid gradient flow-trees in $\mathcal{T}_{\widetilde{\mathbf{C}}}(c_+;\widetilde{\boldsymbol{\gamma}})$ and the rigid  pseudoholomorphic discs in ${\mathcal D}(c_+, \widetilde{\bs{\gamma}})$ by the standard theory of gradient flow-trees \cite{MorseFlow}, \cite[Section 5]{Ekhoka} for a very particular choice of almost complex structure. Then we stretching the neck along the hypersurface
$$\{t=0\} \times \partial W^{sc} \subset (-\infty,0] \times \partial W^{sc} \subset \widehat{W},$$
near which $\widetilde{\mathbf{C}}$ is cylindrical.
At this point we make heavy use of the assuption in the last bullet point of Theorem \ref{thm:morse}, namely that the rigid discs that we count cannot degenerate to a nodal disc when the almost complex structure is deformed. This means that there is a cobordism between moduli spaces for different choices of almost complex structure, and in particular the signed counts remain invariant.

In the limit each rigid pseudoholomorphic disc with at least one puncture at a double point of $\mathbf{C} \subset \widehat{\mathbf{C}}$, and boundary on the Morsified cobordism, necessarily breaks into a pseudoholomorphic building of precisely the two following levels.

\begin{itemize}
   \item {\bf Top level:} A single rigid punctured pseudoholomorphic disc $u_{top}$ contained in $\widehat{W}^{crit}$ with boundary on $\widehat{\mathbf{C}}$ and negative Reeb chords asymptotic to the Reeb chords $\gamma_1,\ldots,\gamma_k$ on $\mathbf{\Lambda}$.
    \item {\bf Bottom level:} A number $k$ of pseudoholomorphic strips in $\R \times \partial W^{sc}$ with boundary on the completion of the exact immersed Lagrangian cylinder
    $$ \widetilde{\mathbf{C}} \cap (-\infty,0] \times \partial W^{sc},$$
where the $i$:th strip has precisely two punctures; one is asymptotic to $\gamma_i$ at the positive end, and one maps to the double poit $\tilde{\gamma}_i$.
\end{itemize}
Note that all middle symplectisation levels have to be empty because of ridigity and additivity of the index.

    Finally, we claim that the count of discs in the top level of this building, where the boundary condition is on the original cobordism $\widehat{\mathbf{C}}$, gives the cardinality of the moduli space we are interested in.
   
    This follows by gluing the broken configuration to obtain the configurations on the Morse cobordism. Note that there is a unique rigid pseudoholomorphic strip that connects $e_{i,j}$ and $\tilde{e}_{i,j}$ for a suitable cylindrical almost complex structure, as shown in e.g.~\cite[Lemma 8.3(1)]{LiftingPseudoholomorphic}. The reason is that a rigid pseudoholomorphic disc must project to a disc with boundary on the Lagrangian projection of $\boldsymbol{\Lambda}$ that is of negative index under $\R \times J^1S^{n-1} \to T^*S^{n-1}$, and must hence be constant. Here the almost complex structure must be suitably chosen, so that the latter projection becomes holomorphic.
\end{proof}

    \subsection{Counting discs on the Morsification}
    In order to find the rigid gradient flow trees, we here provide some useful restrictions on their behaviour.
\begin{lemma}
  A rigid gradient flow-tree for $\widetilde{\mathbf{C}}$ with precisely one positive puncture, that moreover is asymptotic to a double point in $\mathbf{C} \subset \widetilde{\mathbf{C}}$, satisfies the property that all of its edges contained inside $(-\infty,0]_t \times \partial W^{sc}$ have a tangent vector with a non-zero $\partial_t$-component.
\end{lemma}
\begin{proof}
  In our setting, the only possible vertices are $1$ and $2$-valent punctures and $3$-valent vertices of type $Y_0$, in which two gradient flow-lines of $\hat f_i-\hat f_j$ and $\hat f_j- \hat f_k$ join at a flow-line of $\hat f_i-\hat f_k$.

 Near the hypersurface $\{t=0\} \subset \widehat{C}\cong \R^n$, the gradient of the differences $f_i-f_j$ all have a non-vanishing $\partial_t$-component. Every edge that passes through this hypersurface will thus have a positive $\partial_t$-component of its tangent vector there. It is clear that the same property thus holds along the interior of the entire edge intersected with $(-\infty,0] \times \partial W^{sc}$. If the edge terminates at a 1-valent vertex, the property holds for the entire edge. In the other cases, one can check the possible behaviour near the possible 2 and 3-valent vertices to show that this property must hold for all remaining edges connected to the vertex as well.
\end{proof}
As a consequence, we get the following technique for finding the gradient flow trees directly on the cobordism:
\begin{lemma}[Section 5 in \cite{Ekhoka}] 
\label{lemma:trees}
The gradient flow-trees for the Morsification $\widetilde{\mathbf{C}}$ of the types considered here are in bijective correspondence to gradient flow-trees for $\widehat{\mathbf{C}} \subset T^*_{<\epsilon}\widehat{C} \subset T^*\R^n$ with non-compact ends for a suitable choice of metric. An asymptotic constraint for an infinite edge at the orbit $\gamma$ for the flow-tree for $\widehat{\mathbf{C}}$ corresponds to a puncture at the double point $\tilde{\gamma}$ of the Morse cobordism $\widetilde{\mathbf{C}}$.
\end{lemma}
The flow-trees of the above type were called \textbf{Long Conical flow-trees (LC flow-trees for short)} in \cite{Ekhoka}.

\section{SFT-curves on cylinders over multiply-copy Legendrians} \label{appendix: uffa}
In this appendix  we relate the count of holomorphic discs with boundary on 
   a Lagrangian cylinder $\R \times \Lambda$ to the count of certain holomorphic discs with boundary on a Lagrangian cylinder $\R \times \bs{\Lambda}_\epsilon$ consisting of $k$ small perturbations of $\R \times \Lambda$. In the applications $\Lambda$ will be a connected component $S_\sigma$ of the attaching link $\mathbf{S}$ and $\bs{\Lambda}_\epsilon$ will be the link $\bs{\Lambda}$ of Theorem \ref{thm:maingeometric}. It should be possible to compute the discs of interest by an adiabatic limit argument, taking $\epsilon \to 0$, as proposed in the \cite[Conjectural Lemma 4.10]{Ekholm_FloerlagCOnt}. In this limit the finite energy SFT-curves with boundary on $\R \times \mathbf{\Lambda}_\epsilon$ should degenerate to pseudoholomorphic curves with boundary on the single copy $\R \times \Lambda$ together with possibly non-compact gradient flow-trees on the latter cylinder. On the other hand, it should be possible to glue such degenerate configurations with boundary on $\R \times \Lambda$ to actual holomorphic curves with boundary on $\R \times \bs{\Lambda}_\epsilon$. When the Lagrangian is a compact immersion instead of a cylinder, such results have been obtained in several settings, see e.g.~\cite{Duality_EkholmetAl}. Instead of proving this statement in general we will deduce some special cases of curve counts that would be a direct consequence of the more general conjectural result, if proved.

 We will need a version of SFT compactness for $J$-holomorphic curves with varying cylindrical boundary conditions. This case is not considered in the standard reference about SFT compactness \cite{Bourgeois_&_Compactness} and therefore we sketch here the proof highlighting the places where some extra care is needed. We say that a sequence of Lagrangian submanifolds $L_n$ converges to a Lagrangian submanifold $L$, and write $L_n \to L$, if there is a sequence of Hamiltonian diffeomorphisms $\varphi_n$ such that $\varphi_n(L)=L_n$ and $\varphi_n \to \op{id}$ in $C^\infty_{\op{loc}}$. Similarly, we say that a sequence of Legendrian submanifolds $\Lambda_n$ converges to a Legendrian submanifold $\Lambda$, and write $\Lambda_n \to \Lambda$, if there is a sequence of contactomorphisms $\varphi_n$ such that $\varphi_n(\Lambda)=\Lambda_n$ and $\varphi_n \to \op{id}$ in $C^\infty_{\op{loc}}$. Since a contactomorphism lifts to a Hamiltonian diffeomorphism of the symplectisation preserving the $\R$-direction, if $\Lambda_n \to \Lambda$, then $\R \times \Lambda_n \to \R \times \Lambda$ holds for the corresponding Lagrangian cylinders as well.

    Our definition of SFT convergence is very similar to the usual one from \cite[Section 8.2]{Bourgeois_&_Compactness} but not exactly equal, so we will call it {\em partial SFT convergence} to distinguish it from the usual one. The difference is that we will ignore the convergence to gradient trajectories in the case of Morse Bott degenerations of the boundary components and retain only the holomorphic part of the building. 

  Let $\Delta$ be a nodal punctured disc with the nodes removed (i.e.\ a disjoint union of punctured discs with a matching between certain punctures, which correspond to the nodes). Punctures can be either interior or on the boundary; one unmatched boundary puncture is labelled positive and all  other unmatched punctures are labelled negative. We also allow marked points, both in the interior and in the boundary, which will be used in the proof of the compactness theorem as in \cite{Bourgeois_&_Compactness}. 
    We denote by $\Delta^{(1)}, \ldots, \Delta^{(k)}$ the connected components of $\Delta$, which we will also call irreducible components. To a nodal punctured disc $\Delta$ we associate a rooted tree\footnote{The fact that $\Gamma(\Delta)$ is a a tree part of the definition} $\Gamma(\Delta)$  whose vertices are  the irreducible components of $\Delta$, the root is the irreducible component containing the positive puncture,  and the edges are the nodes. Given two irreducible components $\Delta^{(i)}$ and $\Delta^{(j)}$ we say that $\Delta^{(i)} \prec \Delta^{(j)}$ if the shortest path in $\Gamma(\Delta)$ from the root to $\Delta^{(j)}$ passes through $\Delta^{(i)}$. If $\Delta_0$ and $\Delta_1$ are punctured nodal discs, we write $\Delta_0 \subseteq \Delta_1$ if
    \begin{itemize}
    \item the irreducible components of $\Delta_0$ are irreducible components of $\Delta_1$, 
    \item $\Gamma(\Delta_0)$ is a subgraph of $\Gamma(\Delta_1)$, and
    \item the root of $\Gamma(\Delta_0)$ coincides with the root of  $\Gamma(\Delta_1)$.
    \end{itemize}
    A $J$-holomorphic building $v \colon \Delta \to \R \times M$ is a collection of finite energy  $J$-holomorphic maps $v^{(i)} \colon \Delta^{(i)} \to \R \times M$ with boundary on Legendrian cylinders over Lagrangian submanifold, a positive end asymptotic to a nondegenerate Reeb chord (or orbits) at the positive puncture, negative ends asymptotic to  nondegenerate Reeb chords (or orbits) at the negative puncturs, such that matched punctures coming from a node are asymptotic to the same limit, but one positively and the other one negatively. See \cite[Section 7]{Bourgeois_&_Compactness} for the details in the case without boundary. When we will talk about the ends of a $J$-holomorphic buildings without any other specification, we will always mean the unmatched ends of its irreducible components.

    Let $\Delta_n$ be a sequence of punctured discs and $u_n \to \R \times M$ be a sequence of finite energy $J$-holomorphic maps. We say that $u_n$ converges in the partial SFT sense to a $J$-holomorphic building $u_\infty \colon \Delta_\infty \to \R \times M$ if, after adding marked points to each $\Delta_n$ and $\Delta_\infty$, we have:
    \begin{itemize}
    \item $\Delta_n$ converge in the Deligne-Mumford topology to a nodal disc $\Delta_\infty'$ such that $\Delta_\infty \subseteq \Delta_\infty'$,
    \item if $\Delta_\infty^{(i)}$ is an irreducible component of $\Delta_\infty$, then $u_\infty^{(i)}$ is the limit of reparametrisations of  restrictions of $u_n$ according to conditions {\bf CHC1} and {\bf CHC2} of  \cite[Section 7.3]{Bourgeois_&_Compactness}, and
    \item $\sum \limits_{i=1}^k \int_{\Delta_\infty} (u_\infty^{(i)})^* d \alpha = \lim \limits_{n \to \ \infty} \int_{\Delta_n}u_n^* d \alpha$.
    \end{itemize} 
The last condition implies that the procedure for extracting the limit produces constant maps on the irresucible components of $\Delta_\infty'$ which are not irreducible components of $\Delta_\infty$.

    We start with the following local lemma, which extends (a local version of) \cite[Theorem 4.4.1]{McDuff_Salamon_Hol_Curv} to the case of varying boundary conditions.
    \begin{lemma}\label{lemma: compactness with gradient bound}
      Let $(X, \omega)$ be a symplectic manifold, $L \subset X$ a Lagrangian submanifold, $J$ and almost complex structure on $X$ compatible with $\Omega$ and $L_n \to L$ a sequence of converging Lagrangian submanifolds. 
      We denote by $\overline{\HH}$ the closed upper half plane in $\C$. If $\Omega \subset \overline{\HH}$ is an open set, $C>0$ a real number and $u_n \colon \Omega \to X$ a seqence of $J$-holomorphic maps such that $u_n(\partial \overline{\HH} \cap \Omega) \subset L_n$ and $\lVert du_n \rVert_{L^\infty(\Omega)} \le C$ for every $n$, then there is a subsequence $u_{n_k}$ which converges in $C^\infty_{\op{loc}}$ to a $J$-holomorphic map $u \colon \Omega \to X$ such that $u(\partial \overline{\HH} \cap \Omega) \subset L$.
      \end{lemma} 
      \begin{proof}
We define $\tilde{u}_n = \varphi_n^{-1} \circ u_n$ and $\tilde{J}_n = d \varphi_n^{-1} \circ J \circ d \varphi_n$ and apply \cite[Theorem 4.4.1]{McDuff_Salamon_Hol_Curv} to the sequence of $\tilde{J}_n$-holomorphic maps $\tilde{u}_n$.
\end{proof}
The next lemma describes the behaviour of strips with small energy and will play the role of \cite[Proposition 5.7]{Bourgeois_&_Compactness}. Note that our proof will be significantly simpler because we will assume a gradient bound which is verified in situations where we will apply the lemma, but is not assumed in \cite[Proposition 5.7]{Bourgeois_&_Compactness}.
\begin{lemma}\label{lemma: thin strips}
  Let $\Lambda_{0, n} \to \Lambda_{0, \infty}$ and $\Lambda_{1, n} \to \Lambda_{1, \infty}$ be Legendrian submanifolds such that $\Lambda_{0, \infty}$ and $\Lambda_{1, \infty}$ are either equal or disjoint, $N_n \to + \infty$ real numbers and $f_n \colon [-N_n, N_n] \times [0,1] \to \R \times M$ $J$-holomoprphic maps such that
  \begin{enumerate}
  \item $f_n ([-N_n, N_n] \times \{ 0 \}) \subset \R \times \Lambda_{0,n}$ and $f_n ([-N_n, N_n] \times \{ 1 \}) \subset \R \times \Lambda_{1,n}$,
  \item $\| f_n \| _{L^\infty([-N_n, N_n] \times [0,1])} \le C$ for some constant $C$ independent of $n$,
  \item $\lim \limits_{n \to \infty} \int_{[-N_n, N_n] \times [0,1]}f_n^*(d \alpha)=0$,
  \item $\lim \limits_{n \to \infty} \int_{\{ N_n \} \times [0,1]} f_n^*\alpha = \mathfrak{a} >0$, and
  \item all Reeb chords from $\Lambda_{0, \infty}$ to $\Lambda_{1, \infty}$ with action $\mathfrak{a}$ are nondegenerate.
  \end{enumerate}
Then for every $\varepsilon$ sufficiently small there exists $m \in \N$ such that for every $n \ge m$ the image of $f_n$ is contained in an $\varepsilon$-neighbourhood of a trivial strip over a Reeb chord of action $\mathfrak{a}$.
\end{lemma}
\begin{proof}
  Let $\varepsilon >0$ be small enough so that the $\varepsilon$-neighbourhoods of the trivial strips over the Reeb chords from $\Lambda_{0, \infty}$ to $\Lambda_{1, \infty}$ with action $\mathfrak{a}$ are pairwise disjoint. It is possible to find such an $\varepsilon$ because those Reeb chords are nondegenerate and therefore isolated. We call ${\mathcal U}_\varepsilon$ the union of such neighbourhoods. Now suppose the conclusion of the lemma does not hold: this means that for every $n$ there are points $(s_n, t_n) \in [-N_n, N_n] \times [0,1]$ such that $f_n(s_n, t_n) \not \in {\mathcal U}_\varepsilon$. Up to translations in the target, we can assume that $f_n(s_n, t_n) \in \{ 0 \} \times M$. We define $c_+= \min_{n \in \N} \{ N_n-s_n, 1 \}$ and $c_- = \max_{n \in \N} \{ -N_n-s_n, -1 \}$. It is clear that $c_- \le 0 \le c_+$ and $c_+ - c_- \ge 1$. We define $J$-holomorphic maps $\tilde{f}_n \colon [c_-, c_+] \times [0,1] \to \R \times M$ by $\tilde{f}_n(s, t) = f_n(s+s_n, t)$. By Lemma \ref{lemma: compactness with gradient bound} the maps $\tilde{f}_n$ converge uniformly with all derivatives to a $J$-holomorphic map $\tilde{f}_\infty \colon [c_-, c_+] \times [0,1] \to \R \times M$ which satisfies $\int_{[c_-, c_+] \times [0,1]}\tilde{f}_\infty^*(d \alpha)=0$ and $\int_{\{ 0 \} \times [0,1]}\tilde{f}_\infty^* \alpha = \mathfrak{a}$. Then $\tilde{f}_\infty$ is a portion of a trivial strip over a Reeb chord of action $\mathfrak{a}$, which is a contradiction because, up to passing to a subsequence, $t_n \to t_\infty$ and $\tilde{f}(0, t_\infty) \not \in {\mathcal U}_\varepsilon$. Then $f_n([-N_n, N_n] \times [0,1]) \subset {\mathcal U}_\varepsilon$ for $n$ large enough. Since $f_n([-N_n, N_n] \times [0,1])$ is connected, by the choiche of $\varepsilon$ it is contained in the $\varepsilon$-neighbourhood of a Reeb chord. 
\end{proof}

\begin{thm}\label{thm: SFT compactness with moving boundary conditions}
  Let $(M, \alpha)$ be a contact manifold and let $\Lambda_{1,n}, \ldots, \Lambda_{r,n}$, for $n=1, \ldots, \infty$, be embedded closed Legendrian submanifolds such that
  \begin{itemize}
  \item[(i)] for every $n=1, \ldots, \infty$ any pair of Legendrians $\Lambda_{i,n}$ and $\Lambda_{j,n}$ are either disjoint or equal,
  \item[(ii)] $\Lambda_{i,n} \xrightarrow{n \to \infty} \Lambda_{i, \infty}$ for each $i=1, \ldots, r$, and
  \item[(iii)] there are constants $0<Q'<Q$ such that
    the Reeb chords of the Legendrian link $\Lambda_{1,\infty} \cup \ldots \cup \Lambda_{r,\infty}$ with action in $[Q', Q]$ are nondegenerate and those with action in $(0,Q')$ are either nondegenerate or Morse-Bott, in which case they  are mixed.
  \end{itemize}
  For every $n \ne \infty$ let  $a_{0,n}$ be a nondegenerate Reeb chord from $\Lambda_{1,n}$ to $\Lambda_{r,n}$ and $a_{i,n}$, for $i=1, \ldots, r-1$ a nondegenerate  Reeb chord from $\Lambda_{i,n}$ to $\Lambda_{i+1,n}$.  Assume moreover that, for every $i=0, \ldots, r-1$, $a_{i, n} \to a_{i, \infty}$ where $a_{0, \infty}$ is a nondegenerate Reeb chord from $\Lambda_{0, \infty}$ to $\Lambda_{r, \infty}$ with action in $(Q', Q)$ and $a_{i, \infty}$, for $i=1, \ldots, r-1$, is either a nondegenerate Reeb chord from $\Lambda_{i,\infty}$ to $\Lambda_{i+1, \infty}$ with action in $(Q', Q)$ or a single point, where the latter case only can happen if $\Lambda_{i, n} \ne \Lambda_{i+1, n}$ for all $n \ne \infty$ but  $\Lambda_{i,\infty}= \Lambda_{i+1, \infty}$. Then every sequence  of punctured $J$-holomorphic discs $u_n \colon \Delta_n \to \R \times M$  with boundary on $\Lambda_{0,n} \cup \ldots \cup \Lambda_{r, n}$ which are positively asymptotic to $a_{0,n}$ and negatively asymptotic to $a_{1,n}, \ldots, a_{r-1, n}$ has a subsequence which converges  in the partial SFT sense to a $J$-holomorphic building with boundary on $\Lambda_{0,\infty} \cup \ldots \cup \Lambda_{r, \infty}$, a positive asymptotic to $a_{0, \infty}$, and negatively asymptotic to those among the  $a_{1,\infty}, \ldots, a_{r-1, \infty}$ that are Reeb chords. 
\end{thm}

Before proving the theorem, the following remark is in place:
\begin{rem}
\begin{itemize}
\item Both Lemma \ref{lemma: compactness with gradient bound} and Theorem \ref{thm: SFT compactness with moving boundary conditions} hold if we replace the fixed almost complex structure $J$ with a sequence of cylindrical almost complex structures which converges in $C^\infty_{\op{loc}}$;
\item When also the negative asymptotic Reeb chords of the discs $u_n$ have lenghs uniformly bounded away from zero, then the SFT convergence of Theorem \ref{thm: SFT compactness with moving boundary conditions} coincides with the classical notion in \cite{Bourgeois_&_Compactness}; 
\item The theorem is suboptimal in many ways: for example its hypotheses are designed so that brakings at Morse-Bott orbits cannot happen, and the limit ignores the gradient flow trajectory that are expected to start at those among the $a_{1,\infty}, \ldots, a_{r-1, \infty}$ that are points.
\end{itemize}
\end{rem}
In the proof we will use the following lemma.
\begin{lemma}\label{no Morse-Bott breaking}
  In the hypotheses of Theorem \ref{thm: SFT compactness with moving boundary conditions}, if $\gamma_n$ is a sequence of properly  embedded arcs in $\Delta_n$ which are homotopic through properly embeded arcs, separate $\Delta_n$ in two components each containing a non-empty set of boundary punctures, and are oriented as the boundary of the connected component $\Delta_n^-$ of $\Delta_n \setminus \gamma_n$ which does not contain the positive puncture, and $\liminf \int_{\gamma_n} u_n^* \alpha < Q'$, then the punctures in $\Delta_n^-$ are all asymptotic to chords converging to points. 
\end{lemma}
\begin{proof} Choose $\varepsilon$ such that $0 < \varepsilon < Q' - \liminf \limits_{n \to \infty} \int_{\gamma_n} u_n^* \alpha$. By Stokes theorem and the positivity of the $d\alpha$-area on $J$-holomorphic curves,  $\int_{\gamma_n} u_n^* \alpha$ is larger than the sum of the actions of the Reeb chords at the ends of $u_n|_{\Delta_n^-}$. If at least one of the punctures of $\Delta_n^-$ does not converge to a point, for some $n$ large enough it converges to a chord with action larger than $Q'- \frac{\varepsilon}{2}$ and this is a contradiction.
\end{proof}

\begin{proof}[Proof of Theorem \ref{thm: SFT compactness with moving boundary conditions}]
 The first step of the proof of SFT compactness is to obtain gradient bounds modulo bubbling (\cite[Section 10.2.1]{Bourgeois_&_Compactness}). Since the bubbling analysis is performed in smaller and smaller charts around the bubbling point, the fact that several Legendrian submanifolds may converge to the same one does not matter because the bubbling charts can intersect at most one boundary component. In order to obtain convergence to the bubble we use Lemma \ref{lemma: compactness with gradient bound} and to obtain convergence to a Reeb chord for boundary bubbles we use \cite{Abbas_Chord_Asy}, which can be generalised to arbitrary dimensions (c.f.~the treatment of the case of periodic orbits in \cite{Hofer_&_Symplectisation1}).  Since Morse-Bott chords with action less than $Q'$ are mixed, they do not matter here.

 The second step is the convergence of the Riemann surfaces and of the holomorphic maps on the thick parts (\cite[Section 10.2.2]{Bourgeois_&_Compactness}). The punctured discs $\Delta_n$ converge, possibly after adding extra marked points to control bubbling as in 
 \cite[Lemma 10.7]{Hofer_&_Symplectisation1}, to a nodal punctured disc $\Delta_\infty'$.

 The convergence of the holomorphic maps on the thick part is a consequence of the gradient bound obtained in Step 1 and of Lemma \ref{lemma: compactness with gradient bound}. This step gives $J$-holomorphic maps $u_\infty^{(i)} \colon \Delta_\infty^{(i)} \to \R \times M$.  Observe that some of the $u_\infty^{(i)}$ could be constant. If this happens, then the boundary conditions on that component came from distinct Legendrian submanifolds collapsing to the same in the limit. By Stokes theorem, if $u_\infty^{(i)}$ is constant, then $u_\infty^{(j)}$ is also constant for all $j$ such that $\Delta_\infty^{(i)} \prec \Delta_\infty^{(j)}$, and therefore the irreducible components of $\Delta_\infty'$ on which the limit is not constant form a nodal punctured disc $\Delta_\infty \subseteq \Delta_\infty'$. The irreducible components of $\Delta_\infty'$ on which the limit is constant will be disregarded from the analysis in the following. To that end, observe that any puncture of some limit component that arises as a node which is connected to a constant component must be a removable singularity. (There may also be additional removable singularities arising in the limit that are not nodes, but simply boundary punctures that are asymptotic to Reeb chords whose lengths become zero in the limit.)

 The third step is convergence in the thin part (\cite[Section 10.2.3]{Bourgeois_&_Compactness}). For interior nodes the argument is exactly the same the same as in \cite{Bourgeois_&_Compactness}. For boundary nodes it follows the same lines, with Lemma \ref{lemma: thin strips} replacing \cite[Proposition 5.7]{Bourgeois_&_Compactness}, but slightly more care must be taken to rule out breaking at Morse-Bott chords and in dealing with shrinking Reeb chords. Lemma \ref{no Morse-Bott breaking} implies that a breaking at a Morse-Bott chord can happen only if that chord is pure, but all Morse-Bott chords are mixed by hypothesis. Removable punctures arising from different Legendrian boundary conditions converging to the same one also do not change the proof in a significantg way, because when one arises, by Stokes theorem, the limit becomes constant in all irreducible components before that.

The fourth step is the analysis of the level structure of the limit, which is unaffected by the varying boundary conditions.
\end{proof}

Let  $\Lambda$ be a Legendrian sphere in $(Y, \alpha)$ and $\mathbf{\Lambda}_\epsilon$, for every $\epsilon>0$,  an embedded Legendrian $k$-copy link that is constructed in the same manner as described in Section \ref{sec: cap algebra}, i.e.\ inside a standard Legendrian neghbourhood $U_\Lambda$ of $\Lambda$, which is strictly contactomorphic to a neighbourhood of the $0$-section of the jet space $J^1\Lambda$, by using one-jets of functions $\epsilon  \tilde{f}_j \colon \Lambda \to \R$ for $j=1, \ldots, k$ where $0<\tilde{f}_1 < \ldots <\tilde{f}_k$ and each $\tilde{f}_j$ as well as the differences $\tilde{f}_i-\tilde{f}_j$ are Morse for $i \neq j$.  We will further assume that each $\tilde{f}_j$ is a small perturbation of the constant function $jc>0$, and will write $\tilde{f}_j=jc+\eta \tilde{h}_j$ for some small $\eta>0$. When emphasising the dependence on both $\epsilon,\eta>0$ we will write $\mathbf{\Lambda}_{\epsilon,\eta}$. We will also consider the immersed link $\bs{\Lambda}_0$ consisting of $k$ coinciding copies of $\Lambda$.

For all $Q>0$ there is $\epsilon(Q) >0$ such that, for all $0<\epsilon\le\epsilon(Q)$ and $\eta >0$ sufficiently small compared to $\epsilon$, the Reeb chords of $\mathbf{\Lambda}_{\epsilon, \eta}$  of action less than $Q$ consist of precisely of the Reeb chords $\{a_{i,j}\}$ from the $i$:th to the $j$:th component, for any $i,j \in \{1,\ldots,k\}$ and Reeb chord $a$ on $\Lambda$, (long chords) together with the chords correspondig to the critical points of $\tilde{f}_i-\tilde{f}_j$ (short chords) from the $i$:th to the $j$:th component for each $i<j$; see Lemma \ref{short, long, and very long}. For $\eta=0$ the long chords are as above, while the short chords form Morse-Bott families diffeomorphic to $\Lambda$.  Note that there are no short chords from the $j$:th to the $i$:th component when $i<j$.  We will  always assume that there is a unique minimum-type chord $e_{i,j}$ and maximum-type chord $m_{i,j}$ from $i$:th to the $j$:th sheet for any $i<j$. The short chords all have action bounded by $C\epsilon$ for $C=\max\tilde{f}_k$. The long chords $a_{ij}$ are geometrically close to $a$ and have action  that satisfies $|\mathfrak{a}(a_{i,j})-\mathfrak{a}(a)|\le 2C\epsilon$.

In the following when talking about Reeb chords on $\mathbf{\Lambda}_{\epsilon, \eta}$, or pseudoholmorphic discs with boundary on $\R \times \mathbf{\Lambda}_{\epsilon, \eta}$ with Reeb chord asymptotics, we will always implicitly assume that they all are of action less than $Q>0$, and that $0<\epsilon<\epsilon(Q)$ has been chosen.

Given a string $\tilde{a}_1, \ldots, \tilde{a}_l$ of chords of $\bs{\Lambda}_{\epsilon, \eta}$ which only Reeb chords $a_{i,j}$ and short chords corresponding to minima, we denote by $d(\tilde{a}_1, \ldots, \tilde{a}_l)$ the number of short chords in $\tilde{a}_1, \ldots, \tilde{a}_l$ and by $\Pi(\tilde{a}_1, \ldots, \tilde{a}_l)$ the string of chords of $\Lambda$ obtained by first erasing all short chords, and then replacing each remaining chord, which is long, with the corresponding chord of $\Lambda$. Finally, for the moduli spaces we will use the same notation as in Section \ref{ChekAlg} and will assume that the almost complex structure is chosen generically so that the moduli spaces are regular. With this notation set, we have the following propositions.

\begin{lemma}\label{lemma: constraints on holomorphic curves}
  If $\tilde{a}_0, \ldots, \tilde{a}_l$ is a composable string of chords of $\bs{\Lambda}_{\epsilon, \eta}$ such that all short chords are of minimum type, then for all $\epsilon$ and $\eta$ sufficiently small $\widetilde{\mathcal M}^1_{\bs{\Lambda}_{\epsilon, \eta}}(\tilde{a}_0; \tilde{a}_1, \ldots, \tilde{a}_l) = \emptyset$ unless one of the following conditions hold:
  \begin{enumerate}
  \item all chords $\tilde{a}_1, \ldots, \tilde{a}_l$ are long,
  \item $\widetilde{\mathcal M}^1_{\bs{\Lambda}_{\epsilon, \eta}}(\tilde{a}_0; \tilde{a}_1, \ldots, \tilde{a}_l) = \widetilde{\mathcal M}^1_{\bs{\Lambda}_{\epsilon, \eta}}(a_{i,j}; a_{i,h}, e_{h,j})$ or  $\widetilde{\mathcal M}^1_{\bs{\Lambda}_{\epsilon, \eta}}(\tilde{a}_0; \tilde{a}_1, \ldots, \tilde{a}_l) = \widetilde{\mathcal M}^1_{\bs{\Lambda}_{\epsilon, \eta}}(a_{i,j}; e_{i,h}, a_{h,j})$, where $h<j$ or $i<h$ respectively, and $a_{i,j}$, $a_{i,h}$ and $a_{h,j}$ are long chords corresponding to the same chord $a$ of $\Lambda$, or
  \item $\widetilde{\mathcal M}^1_{\bs{\Lambda}_{\epsilon, \eta}}(\tilde{a}_0; \tilde{a}_1, \ldots, \tilde{a}_l) = \widetilde{\mathcal M}^1_{\bs{\Lambda}_{\epsilon, \eta}}(e_{i,j}; e_{i,h}, e_{h,j})$ where $i<h<j$.
 \end{enumerate}
\end{lemma}
\begin{proof}
  We prove the lemma assuming that $\Lambda$ is graded; the general case is similar but requires more {\em ad hoc} computations of the indices, which are left to the reader. We choose a Maslov potential on $\bs{\Lambda}_{\epsilon, \eta}$ such that all chords $e_{i,j}$ has degree $-1$.

  Denote $d=d(\tilde{a}_1, \ldots, \tilde{a}_l)$. First we assume that $\tilde{a}_0$ is a long chord. If $\widetilde{\mathcal M}^1_{\bs{\Lambda}_{\epsilon, \eta}}(\tilde{a}_0; \tilde{a}_1, \ldots, \tilde{a}_l) \ne \emptyset$ for $(\epsilon, \eta) \to 0$, then by Theorem \ref{thm: SFT compactness with moving boundary conditions} there are sequences $(\epsilon_n, \eta_n) \to 0$ and $u_n \in \widetilde{\mathcal M}^1_{\bs{\Lambda}_{\epsilon, \eta}}(\tilde{a}_0; \tilde{a}_1, \ldots, \tilde{a}_l)$ such that $u_n$ converge in the partial SFT sense to a holomorphic building $u_n$ with a positive end at $a_0$ and negative ends at $\Pi(\tilde{a}_1, \ldots, \tilde{a}_l)$. The index of $u_\infty$ is $\op{ind}(u_\infty) = \op{ind}(u_n)-d$ because the minimum type chords have index $-1$. Since $\op{ind}(u_n)=1$ and $\op{ind}(u_\infty) \ge 0$ because $J$ is generic, we have $d \in \{ 0, 1 \}$. If $d=0$ we are in Case (1). If $d=1$, then $\op{ind}(u_\infty)=0$, and therefore  $u_\infty$ is a trivial strip over a Reeb chord $a$. This implies that we are in Case (2).

  If $\tilde{a}_0$ is a short chord, then $\tilde{a}_1, \ldots, \tilde{a}_l$ are also short chords by action condiderations.  Since short chords of minimum-type have degree -1, the only possibility to have a curve of index one with all ends of minimum type is if $l=2$, so we are in Case (3). 
\end{proof}
The next proposition analyses Case (1) of Lemma \ref{lemma: constraints on holomorphic curves}
\begin{prop}\label{huey}
  If $\tilde{a}_0, \ldots, \tilde{a}_l$ is a composable string of long chords of $\bs{\Lambda}_{\epsilon, \eta}$ and $a_0, \ldots, a_l$ is the corresponding string of chords of $\Lambda$, then for all $\epsilon >0$ sufficiently small and $\eta >0$ sufficiently small compared to $\epsilon$, the signed count of elements of $\widetilde{\mathcal M}^1_{\bs{\Lambda}_{\epsilon, \eta}}(\tilde{a}_0; \tilde{a}_1, \ldots, \tilde{a}_l)$ and $\widetilde{\mathcal M}_\Lambda^1(a_0; a_1, \ldots, a_l)$ agree. 
\end{prop}
\begin{proof}
  We prove the proposition in two steps: first we show that
  \begin{equation} \label{first step}
    \# \widetilde{\mathcal M}^1_{\bs{\Lambda}_{\epsilon, 0}}(\tilde{a}_0; \tilde{a}_1, \ldots, \tilde{a}_l) = \# \widetilde{\mathcal M}_\Lambda^1(a_0; a_1, \ldots, a_l)
  \end{equation}
  for all $\epsilon>0$ sufficiently small, then we prove that, for every such $\epsilon$,
  \begin{equation} \label{second step} \# \widetilde{\mathcal M}^1_{\bs{\Lambda}_{\epsilon, \eta}}(\tilde{a}_0; \tilde{a}_1, \ldots, \tilde{a}_l) = \# \widetilde{\mathcal M}^1_{\bs{\Lambda}_{\epsilon, 0}}(\tilde{a}_0; \tilde{a}_1, \ldots, \tilde{a}_l)
  \end{equation}
  if $\eta>0$ is sufficiently small.

   The two steps are proved in similar ways; since the first one is the the less standard, we will focus on its proof. We denote by ${\mathcal M}^1_{[0, \epsilon_0]}(\tilde{a}_0; \tilde{a}_1, \ldots, \tilde{a}_l)$ the set of pairs $(u,\epsilon)$ such that $\epsilon \in [0, \epsilon_0]$, and $u \in {\mathcal M}_\Lambda^1(a_0; a_1, \ldots, a_l)$ if $\epsilon=0$ or $u \in {\mathcal M}^1_{\bs{\Lambda}_{\epsilon, 0}}(\tilde{a}_0; \tilde{a}_1, \ldots, \tilde{a}_l)$ if $\epsilon >0$. After dividing by $\R$-translations in the target we obtain $\widetilde{\mathcal M}^1_{[0, \epsilon_0]}(\tilde{a}_0; \tilde{a}_1, \ldots, \tilde{a}_l)$. We will prove that for $\epsilon_0$ small enough $\widetilde{\mathcal M}^1_{[0, \epsilon_0]}(\tilde{a}_0; \tilde{a}_1, \ldots, \tilde{a}_l)$ is a compact manifold with boundary
   $$ \partial \widetilde{\mathcal M}^1_{[0, \epsilon_0]}(\tilde{a}_0; \tilde{a}_1, \ldots, \tilde{a}_l) = \widetilde{\mathcal M}_\Lambda^1(a_0; a_1, \ldots, a_l) \sqcup \widetilde{\mathcal M}^1_{\bs{\Lambda}_{\epsilon_0,0}}(\tilde{a}_0; \tilde{a}_1, \ldots, \tilde{a}_l).$$
   This will prove Equation \eqref{first step}.\footnote{In fact we will prove more, namely that the projection $\widetilde{\mathcal M}^1_{[0, \epsilon_0]}(\tilde{a}_0; \tilde{a}_1, \ldots, \tilde{a}_l) \to [0, \varepsilon_0]$ has no critical points, and therefore there is a bijection between $\widetilde{\mathcal M}_\Lambda^1(a_0; a_1, \ldots, a_l)$ and $\widetilde{\mathcal M}^1_{\bs{\Lambda}_{\epsilon, 0}}(\tilde{a}_0; \tilde{a}_1, \ldots, \tilde{a}_l)$ for all $\epsilon \in [0,\epsilon_0]$.}

 First  we study the regularity of the parametrised moduli space.
The new phenomenon to address is the varying boundary conditions, and in particular the possibility that distinct  Lagrangian boundary conditions for $\epsilon >0$ become the same at $\epsilon=0$. We assume for simplicity that there are no interior punctures and that the conformal structure of the domain is fixed.  Although the actual Fredholm problem we are interested in is slightly different,  the extra difficulties which are brought in by the variation of the conformal structure on the domain are independent from those which are introduced by the varying boundary conditions. In the more general case we need to consider Teichm{\"u}ller slices for the domains; see for example \cite[Section 9h]{Seidel_Fukaya} or \cite[Lecture 7]{Wendl:SFT}.

We start by chosing a Riemannian metric on $M$ for which $\Lambda$ is totally geodesic and is invariant under the Reeb flow in a neighbourhood of $\Lambda$ which is large enough to contain  $\bs{\Lambda}_{\epsilon,0}$ for all $\epsilon \in [0,\epsilon_0]$. This implies in particular that $\bs{\Lambda}_{\epsilon, 0}$ is totally geodesic for every $\epsilon \in [0, \epsilon_0]$.\footnote{This is the reason why we prove the lemma in two steps. For the perturbation from $\bs{\Lambda}_{\epsilon, 0}$ to $\bs{\Lambda}_{\epsilon, \eta}$ one can work instead with a family of Riemannian metrics which depend on $\eta$, but trying to construct a family of Riemannian metrics for which all $\bs{\Lambda}_{\epsilon, \eta}$ are totally geodesic, including for $\epsilon=0$, seems unnecessarily complicated due to the fact that different Legendrian submanifold can go to the same as $\epsilon \to 0$.} We extend this metric on $M$ to a product metric on $\R \times M$ which is the standard metric on the $\R$ factor. The exponential map on $\R \times M$ will always be induced by such a metric.

Let $\Delta$ be the domain of the holomorphic curves. Around each puncture of $\Delta$ we choose strip-like ends with coordinates $\sigma_i, \tau_i$ such that  $(\sigma_0, \tau_0) \in (0, + \infty) \times [0,1]$ and $(\sigma_i, \tau_i) \in (- \infty, 0) \times [0,1]$ for $i=1, \ldots, l$.  From now on we write $\tilde{a}_i^\epsilon$ to  stress the dependence of $\tilde{a}_i$ on $\epsilon$. For every $\epsilon \in [0, \epsilon_0]$ we denote by $T_i(\epsilon)$ the action of the Reeb chord $\tilde{a}_i^\epsilon$ of $\bs{\Lambda}_{\epsilon,0}$; i.e.\ $\tilde{a}_i ^\epsilon \colon [0,T_i(\epsilon)] \to M$. We also denote by $\lambda_0(\epsilon)$ the absolute value of the largest negative eigenvalue of the asymptotic operator of $\tilde{a}_0^\epsilon$ and by $\lambda_i(\epsilon)$ the smallest positive eigenvalue of the asymptotic operator of $\tilde{a}_i^\epsilon$ for $i=1, \ldots, l$. Since the Reeb chords $\tilde{a}_i^\epsilon$ remain nondegenerate for all $\epsilon \in [0, \epsilon_0]$, the functions $T_i(\epsilon)$ and $\lambda_i(\epsilon)$ are uniformely bounded away from zero.
We fix $p>2$ and $\delta$ such that  $0 < \delta< \lambda_i(\epsilon)$ for all $\epsilon \in [0,\epsilon_0]$ and $i=0, \ldots, l$. Finally we denote by ${\mathcal B}_{[0, \epsilon_0]}^{p,\delta}$ the set of pairs $(\exp_{v}(\xi), \epsilon)$ such that
\begin{enumerate}
\item $\epsilon \in [0, \epsilon_0]$,
\item $v \colon \Delta \to \R \times M$ is a smooth function such that
 $v(\partial \Delta) \subset \bs{\Lambda}_{\epsilon, 0}$ and $v(\sigma_i, \tau_i, \epsilon)=(T_i(\epsilon) \sigma_i + \kappa_i, \tilde{a}_i^\epsilon(T_i(\epsilon)\tau_i))$ for all $\epsilon \in [0,\epsilon_0]$ where $\kappa_i \in \R$, and 
  \item $\xi$ is a section of $v^*(\R \times M)$ of weighted Sobolev class $W^{1,p,\delta}$ such that $\xi|_{\partial \Delta}$ takes values in $\R \times T \bs{\Lambda}_{\epsilon,0}$ (this condition makes sense because sections of Sobolev class $W^{1,p,\delta}$ are continuous).
  \end{enumerate}
  The meaning of this definition is the following: the maps $v$ are smooth maps with the same boundary conditions and the same asymptotics as the maps in ${\mathcal M}_{[0, \epsilon_0]}(\tilde{a}_0; \tilde{a}_1, \ldots, \tilde{a}_l)$ ands which coincide with a trivial strip on the strip-like ends of $\Delta$, while $\xi$ is a perturbation of $v$ which preserves the boundary conditions and the asymptotics, and decreases exponentially fast in the strip-like ends. The asymptotic estimates for holomorphic maps imply that ${\mathcal M}_{[0, \epsilon_0]}(\tilde{a}_0; \tilde{a}_1, \ldots, \tilde{a}_l) \subset {\mathcal B}_{[0, \epsilon_0]}^{p,\delta}$; see \cite[Lecture~7.2]{Wendl:SFT}.
  
  There is an obvious projection ${\mathcal B}_{[0, \epsilon_0]}^{p, \delta} \to [0, \epsilon_0]$ defined by $(u, \epsilon) \mapsto \epsilon$ and we denote by ${\mathcal B}_\epsilon^{p, \delta}$ the preimage over $\epsilon$. It is a folklore result that ${\mathcal B}_\epsilon^{p, \delta}$ is a Banach manifold, and 
 we will show that the same argument applies to ${\mathcal B}_{[0, \varepsilon]}^{p,\delta}$, which is therefore a Banach manifold with boundary. 

 We fix functions $\chi_i \colon \Delta \to [0,1]$ supported in the $i$-th strip-like end such that $\chi_0(\sigma_0)=1$ if $\sigma_0 \ge 1$, and $\chi_i(\sigma_i)=1$ if $i=1, \ldots, l$ and $\sigma_i \le -1$. Given a smooth function $\tilde{v} \colon \Delta \times [0, \epsilon_0] \to \R \times M$  such that, for every $\epsilon \in [0, \epsilon_0]$, the restriction $\tilde{v}_\epsilon = \tilde{v}|_{\Delta \times \{ \epsilon \}}$ satisfies the conditions of item (2), we can choose a trivialisation
 \begin{equation} \label{tellaro}
   \tilde{v}^* T(\R \times M) \cong \C^n \times \Delta \times [0, \epsilon_0]
 \end{equation}
 such that there is a real subbundle $F \subset \C^n \times \Delta$ with the property that $(\tilde{v}_\epsilon|_{\partial \Delta})^*T(\R \times \bs{\Lambda}_\epsilon)$ is identified with $F \times \{ \epsilon \}$ for every $\epsilon \in [0, \epsilon_0]$.  Thus we can identify the section $\xi$ in item (3) with maps in the weighted Sobolev space $W^{1,p,\delta}(\Delta, \C^n)$ with values in $F$ along $\partial \Delta$, and moreover the identification is smooth in $\epsilon$. We will denote 
the space of such maps by $W^{1,p,\delta}(\Delta, \C^n; F)$.

 We call $s$ the coordinate on $\R$. For every smooth function $\tilde{v}$ as above, we define a map  $\tilde{\psi} \colon \R^{l+1} \times W^{1,p,\delta}(\Delta, \C^n; F)  \times [0, \epsilon_0] \to {\mathcal B}^{p, \delta}_{[0, \epsilon_0]}$ by
 $$\psi(\nu_0, \ldots,  \nu_l, \xi, \epsilon) = (\exp_{\tilde{v}_\epsilon}((\nu_0 \chi_0 + \ldots + \nu_l \chi_l) \partial_s + \xi), \epsilon),$$
 which is a homeomorphism with its image when restricted to ${\mathcal U} \times [0, \epsilon_0]$ where ${\mathcal U}$ is some neighbourhood of the origin in $\R^{l+1} \times W^{1,p,\delta}(\Delta, \C^n; F)$.

To show that these maps induce the structure of a Banach manifold with boundary on ${\mathcal B}^{p, \delta}_{[0, \epsilon_0]}$ it is enough to show that their images cover the whole ${\mathcal B}^{p, \delta}_{[0, \epsilon_0]}$, and that  for any two such maps $\psi_0 \colon {\mathcal U}_0 \times [0, \epsilon_0] \to {\mathcal B}_{[0, \epsilon_0]}^{p,\delta}$ and $\psi_1 \colon {\mathcal U}_1 \times [0, \epsilon_0] \to {\mathcal B}_{[0, \epsilon_0]}^{p,\delta}$  with partially overlapping images, the composition $\psi_1^{-1} \circ \psi_0$, where defined, is a smooth map. To show that the images cover the whole ${\mathcal B}^{p, \delta}_{[0, \epsilon_0]}$ one only needs to observe that for every $\epsilon$ and every $v$ satisfying item (2) there is a map $\tilde{v}$ as above such that $\tilde{v}_\epsilon=v$.

To show that $\psi_1^{-1} \circ \psi_0$ is differentiable, we look more in details how this map behaves.  The parametrisations $\psi_j$ are centred at maps $\tilde{v}^j$ which, in the cylindrical ends, have the form $\tilde{v}^j(\sigma_i, \tau_i, \epsilon)= (T_i(\epsilon) s_i + \kappa_i^j, \tilde{a}^\epsilon_i(T_i(\epsilon) \tau_i)$. Thus $\psi_1^{-1} \circ \psi_0$ has the form
$$(\nu_1, \ldots, \nu_l, \xi, \epsilon) \mapsto (\kappa_1^0- \kappa_1^1 + \nu_1, \ldots, \kappa_l^0- \kappa_l^1 + \nu_l, G_*(\xi, \epsilon), \epsilon)$$

where $G_*$ is the operator defined by composition with a function $G \colon \Delta \times \C^n \times [0,\epsilon_0] \to \C^n$, i.e.\ for every $z \in \Delta$, 
$$G_*(\xi, \epsilon)(z)= G(z, \xi(z), \epsilon).$$
We observe that $G$ is essentially $\exp_{\tilde{v}_1}^{-1} \circ \exp_{\tilde{v}_0}$, and therefore it is smooth, and in the strip-like ends $G(\sigma_i, \tau_i, \xi, \epsilon)= \xi$. These properties will be crucial to the proof that $G_*$ is a smooth operator, that we postpone to Lemma \ref{Why should we do all the dirty work?}. 

The Cauchy-Riemann operator defines a Fredholm section of a Banach bundle ${\mathcal E}^{p, \delta}_{[0, \epsilon_0]} \to {\mathcal B}^{p, \delta}_{[0, \epsilon_0]}$. For a generic almost complex structure $J$ its linearisation is transverse to the zero section on ${\mathcal B}^{p, \delta}_0$ and therefore, since transversality is an open condition, we can assume that it is transverse to the zero section on   ${\mathcal B}^{p, \delta}_{[0, \epsilon_0]}$ up to making $\epsilon_0$ smaller. Thus for a generic $J$, possibly up to making $\epsilon_0$ smaller, $\widetilde{\mathcal M}^1_{[0, \epsilon_0]}(\tilde{a}_0; \tilde{a}_1, \ldots, \tilde{a}_l)$ is a one-dimensional manifold with boundary and the projection to $[0, \epsilon_0]$ has no critical points. It is also compact because, by Theorem \ref{thm: SFT compactness with moving boundary conditions}, every sequence in $\widetilde{\mathcal M}^1_{[0, \epsilon_0]}(\tilde{a}_0; \tilde{a}_1, \ldots, \tilde{a}_l)$ has a subsequence converging to a $J$-holomorphic building with all components having asymptotics only at long and non-degenerate Reeb chords. However, by the additivity of the index, if the building has more than one level, one must have negative expected dimension, which is not possible because $J$ is generic.
\end{proof}
\begin{rem}
The key points of the proof are the existence of a weight $\delta$ working for every $\epsilon$ and the existence of the trivialisation \eqref{tellaro}.
\end{rem}

Now we pay the anlytical debt and prove that the operator $G_*$ is smooth. We fix some notation first. We denote by ${\mathcal E}$ the union of the strip-like ends of $\Delta$ (i.e. where the coordinates $(\sigma_i, \tau_i)$ are defined). We define a function $\sigma \colon \Delta \to [0, +\infty)$ such that $\sigma = |\sigma_i|$ where $|\sigma_i| \ge 1$ and $\sigma \equiv 0$ outside of ${\mathcal E}$. We define also a measure $dx$ on $\Delta$ which restricts to $d\sigma_i d\tau_i$ on the positive ends by choosing a suitable embedding of $\Delta$ in $\C$. We recall that $W^{1,p,\delta}(\Delta, \R^m)$ is the space of functions $\xi \colon \Delta \to \R^m$ such that $e^{\delta \sigma}\xi \in W^{1,p}(\Delta, \R^m)$ and $\| \xi \|= \| \xi \|_{1,p, \delta} = \| e^{\delta \sigma} \xi \|_{1,p}$. We recall that, for $p>2$, there is a continuous embedding $W^{1,p}(\Delta, \R^m) \hookrightarrow C^0(\Delta, \R^m)$, which implies that $\xi$ is continuous and $\lim \limits_{\sigma_i \to \pm \infty} \xi(\sigma_i, \tau_i)=0$ if $\xi \in W^{1,p,\delta}$, and moreover there are inequalities
$$\| \xi \|_{C^0} \le \| e^{\delta \sigma} \xi \|_{C^0} \le C \| e^{\delta \sigma} \xi \|_{1,p} = C  \| \xi \|_{1,p, \delta}.$$
The first inequality follows from $e^{\delta \sigma} \ge 1$ and the second one from the Sobolev embedding theorem.

Given a function $G \colon \Delta \times \R^m \times [0, \epsilon_0] \to \R^k$ we will denote by $\frac{\partial G}{\partial x}$ the Jacobian matrix of $G$ with respect to the variables in $\Delta$, by $\frac{\partial G}{\partial v}$  the Jacobian matrix of $G$ with respect to the variable in $\R^m$, and by $\frac{\partial G}{\partial \epsilon}$  the derivative of $G$ with respect to the variable in $[0, \epsilon_0]$.  Finally we will denote by $| v |$ the Euclidean norm in any finte dimensional vector space.

\begin{lemma}\label{Why should we do all the dirty work?}
  Let $G \colon \Delta \times \R^m \times [0, \epsilon_0] \to \R^m$ be a smooth function such that, in the strip-like ends,  $G(x,v,\epsilon) =v$.  Then the map
  $$G_* \colon W^{1, p, \delta}(\Delta, \R^m) \times [0, \epsilon_0] \to W^{1, p, \delta}(\Delta, \R^m)$$
  defined by $G_*(\xi, \epsilon)(x) = G(x, \xi(x), \epsilon)$ is smooth.
\end{lemma}
\begin{proof}
  We start by verifying that $G_*(\xi, \epsilon) \in W^{1,p,\delta}(\Delta, \R^m)$ if $\xi \in W^{1,p,\delta}(\Delta, \R^m)$. We need to verify the two inequalities
  \begin{equation}\label{sobolev 1}
    \int_{\Delta} e^{p\delta \sigma(x)}|G(x, \xi(x), \epsilon)|^p dx < + \infty,
  \end{equation}
  \begin{equation}\label{sobolev 2}
    \int_\Delta e^{p\delta \sigma(x)} \left | \frac{\partial}{\partial x}(G(x, \xi(x), \epsilon)) \right |^p dx < + \infty.
  \end{equation}
  We consider Equation \eqref{sobolev 1}. Using the properties of $G$ and $\sigma$ we have
  $$ \int_{\Delta} e^{p\delta \sigma(x)}|G(x, \xi(x), \epsilon)|^p dx = \int_{\Delta \setminus {\mathcal E}} |G(x, \xi(x), \epsilon |^p dx + \int_{\mathcal E} e^{p\delta \sigma(x)} |\xi(x)|^p dx.$$
  The first term is finite because $\Delta \setminus {\mathcal E}$ is compact and the integrand is continuous, and the secon term is bounded above by $\|\xi\|^p$,
  Now we consider Equation \eqref{sobolev 2}. We decompose the domain of integration and apply the chain rule to obtain
  \begin{align*}
 & \int_\Delta e^{p\delta \sigma(x)} \left | \frac{\partial}{\partial x}(G(x, \xi(x), \epsilon)) \right |^p dx = \\ & \int_{\Delta \setminus {\mathcal E}} \left |\frac{\partial G}{\partial x}(x, \xi(x), \epsilon) + \frac{\partial G}{\partial v}(x, \xi(x), \epsilon) \frac{\partial \xi}{\partial x} \right |^p dx +  \int_{\mathcal E}e^{p\delta \sigma(x)} \left | \frac{\partial \xi}{\partial x} \right |^p dx.
  \end{align*}
  For the first integral we have
  \begin{align*}
    & \int_{\Delta \setminus {\mathcal E}} \left |\frac{\partial G}{\partial x}(x, \xi(x), \epsilon) + \frac{\partial G}{\partial v}(x, \xi(x), \epsilon) \frac{\partial \xi}{\partial x} \right |^p dx \le \\
   & C_0 + C_1 \int_{\Delta \setminus {\mathcal E}} \left | \frac{\partial \xi}{\partial x} \right |^p dx \le C_0 + C_1 \| \xi \|
  \end{align*}
  because $\frac{\partial G}{\partial x}(x, \xi(x), \epsilon)$ and $\frac{\partial G}{\partial v}(x, \xi(x), \epsilon)$ are continuous and $\Delta \setminus {\mathcal E}$ is compact.  The first integral is bounded above by $\| \xi \|$. This ends the proof of Equation \eqref{sobolev 2}, and therefore of the claim that $G_*(\xi, \epsilon) \in W^{1,p, \delta}(\Delta, \R^m)$.

  The next step is to prove that $G_*$ is continuous. Given $v, v' \in \R^m$ and $\epsilon, \epsilon' \in [0, \epsilon_0]$, by Hadamard's lemma we can write
  $$G(x, v', \epsilon')- G(x, v, \epsilon) = R_0(x, v',v,\epsilon', \epsilon)(v'-v) + S_0(x, v',v,\epsilon', \epsilon)(\epsilon' - \epsilon)$$
  where $R_0$ and $S_0$ are smooth functions (matrix and vector valued respectively) such that $R_0$ is constantly equal to the identity and $S_0$ vanishes for $x \in {\mathcal E}$.  Since the $W^{1,p, \delta}$ norm bounds the $C^0$ norm, for every constant $c>0$ there is a constant $C>0$ such that, for all $\xi, \xi' \in W^{1,p, \delta}(\Delta, \R^m)$ with $\| \xi' \|, \| \xi \| \le c$,

 $$|R_0(x, \xi'(x), \xi(x),\epsilon', \epsilon)|<C, \quad |S_0(x, \xi'(x), \xi(x),\epsilon', \epsilon)|  < C.$$
  Then $\|G_*(\xi', \epsilon')-G_*(\xi, \epsilon) \| \le C \| \xi' - \xi \|$, and therefore $G_*$ is locally Lipschitz, and in particular continuous.

  Now we prove that $G_*$ is differentiable.  We denote by
  $$L(\xi, \epsilon) \colon W^{1,p, \delta}(\Delta, \R^m) \oplus \R \to W^{1,p, \delta}(\Delta, \R^m)$$
 the linear map given by
  $$(L(\xi, \epsilon)(\eta, \dot{\epsilon}))(x) = \frac{\partial G}{\partial v}(x, \xi(x), \epsilon)\eta(x) + \frac{\partial G}{\partial \epsilon}(x, \xi(x), \epsilon) \dot{\epsilon}.$$

  The functions $\frac{\partial G}{\partial v}(x, \xi(x), \epsilon)$ and $\frac{\partial G}{\partial \epsilon}(x, \xi(x), \epsilon)$ are uniformely bounded, and therefore there is a constant $C$ (depending on $\xi$ and $\epsilon$, but not on $\eta$ and $\dot{\epsilon}$) such that $\| L(\xi, \epsilon)(\eta, \dot{\epsilon}) \| \le C (\| \eta \| + | \dot{\epsilon}|)$, i.e. $L(\xi, \epsilon)$ is a bounded linear operator.

  To prove that $L(\xi, \epsilon)$ is the differential of $G_*$ at $(\xi, \epsilon)$ it is enough to show that for all $\xi, \xi' \in W^{1,p, \delta}(\Delta, \R^m)$ and $\epsilon, \epsilon' \in [0, \epsilon_0]$ with $\xi'$ and $\xi$ close enough there is a constant $C>0$ such that 
$$\| G_*(\xi', \epsilon') - G_*(\xi, \epsilon) - L(\xi, \epsilon)(\xi'-\xi, \epsilon' - \epsilon) \| \le C(\| \xi'-\xi\|^2+(\epsilon'-\epsilon)^2).$$
  A further application of Hadamard's lemma yields
  \begin{align*}
  &  G(x, v', \epsilon') - G(x, v, \epsilon) - \frac{\partial G}{\partial v}(x, v, \epsilon)(v' -v) -  \frac{\partial G}{\partial \epsilon}(x, v, \epsilon)(\epsilon' -\epsilon) = \\
    & R_1(x, v', v, \epsilon', \epsilon)  (v'-v)^2 + S_1(x, v', v, \epsilon', \epsilon)(\epsilon' - \epsilon)^2
  \end{align*}
  where $R_1$ and $S_1$ are smooth functions which vanish for $x \in {\mathcal E}$. Note that in this formula $R_1$ is a ``cubic matrix'' and $(v' -v)^2$ is the matrix whose $(i,j)$ entry is the product between the $i$th and the $j$th entries of $v'-v$. As before, for every $\xi$ and every $\xi'$ close enough to $\xi$ there is a constant $C$ such that
  $$|R_1(x, \xi'(x), \xi(x), \epsilon', \epsilon)| \le C, \quad |S_1(x, \xi'(x), \xi(x), \epsilon', \epsilon)| \le C.$$
  Then, for every $\xi$ and every $\xi'$ close enough to $\xi$, we have
  $$\| G_*(\xi', \epsilon') - G_*(\xi, \epsilon) - L(\xi, \epsilon)(\xi'-\xi, \epsilon' - \epsilon) \| \le C (\| (\xi' - \xi)^2 \| + (\epsilon' - \epsilon)^2).$$
  In order to estimate the right hand side, it is enough to estimate $\| (\xi' - \xi)^2\|$.  from  the continuous embedding of $W^{1,p. \delta}(\Delta, \R^m)$ into $C^0(\Delta, \R^m)$ we obtain
  $$\| (\xi' - \xi)^2\| \le C \| \xi' - \xi\|_{C^0} \| \xi'-\xi \| \le C \| \xi' - \xi\|^2$$
  where the constant changes from the first inequality to the second one. This ends the proof that $G_*$ is differentiable.

  Next we prove that $L$ is continuous as an operator valued function. We recall that the operator norm has the property that if $A$ is a linear operator such that $\|A\eta \| \le c \| \eta\|$ for every $\eta$, then $\|A \|_{op} \le c$. Thus in order to prove that $L$ is a continuous map, it is enough to show that for every $\xi$ there exists a constant $C>0$ such that for every $\xi'$ close enough to $\xi$, every $\epsilon$ and $\epsilon'$, and every  $\dot{\epsilon}$ and $\eta$ the following inequality holds:
  $$\|L(\xi', \epsilon') (\eta, \dot{\epsilon}) - L(\xi, \epsilon) (\eta, \dot{\epsilon}) \| \le C(\|\xi'-\xi\|+|\epsilon' - \epsilon|)( \|\eta\| +|\dot{\epsilon}|).$$
  The proof of this inequality is exactly the same as the proof of the inequality implying the continuity of $G_*$.
  
  In order to prove that $G_*$ is $C^\infty$ we must keep differentiating: at the $k$th step we obtain a map from $W^{1,p, \delta} \times [0, \epsilon_0]$ to the space of bounded $k$-linear operators $\underbrace{W^{1, p, \delta}(\Delta, \R^m) \times \ldots \times W^{1, p, \delta}(\Delta, \R^m)}_k \to W^{1, p, \delta}(\Delta, \R^m)$ and must prove that it is continuous and differentiable with respect to the operator norm. This is tedious, but the necessary verifications are not dissimilar from what we have done so far.   
\end{proof}

The next proposition analyses Case (2) of Lemma \ref{lemma: constraints on holomorphic curves}
\begin{prop}\label{louie}
For every $\epsilon, \eta$ sufficiently small we have
$$\# \widetilde{\mathcal M}^1_{\bs{\Lambda}}(a_{i,j}; a_{i,h}, e_{h,j}) = \# \widetilde{\mathcal M}^1_{\bs{\Lambda}}(a_{i,j}; e_{i,h}, a_{h,j})=1$$
for every $h<j$ or $i<h$, respectively, where $a_{i,j}$ and $a_{i,h}$, $a_{h,j}$ are the long chords corresponding to $a$.
\end{prop}
  \begin{proof}
    We focus on $\widetilde{\mathcal M}^1_{\bs{\Lambda}_{\epsilon,eta}}(a_{i,j}; a_{i,h}, e_{h,j})$. Recall that $\bs{\Lambda}_{\epsilon,\eta}$ is defined by perturbing $\bs{\Lambda}_{\epsilon,0}$ by using a sequence of Morse functions $\epsilon\tilde{f}_j=\epsilon(jc+\eta\tilde{h}_j)$. We choose a different perturbation $\widetilde{\bs{\Lambda}}_{\epsilon, \eta}$ of $\bs{\Lambda}_{\epsilon, 0}$ that is induced by a sequence of Morse functions $\epsilon(jc+\eta\tilde{k}_j)$, for which the chords 
$a_{i,h}$ and  $e_{h,j}$ are contained inside the chord $a_{i,j}$. Then there is a holomorphic disc with a positive puncture at $a_{i,j}$ and negative ends at $a_{i,h}$ and $e_{h,j}$ whose image of this disc is contained inside the trivial strip over $a_{i,j}$. By action considerations this disc is the only element of $\widetilde{\mathcal M}^1_{\widetilde{\bs{\Lambda}}_{\epsilon,\eta}}(a_{i,j}; a_{i,h}, e_{h,j})$. Even if $\widetilde{\bs{\Lambda}}_{\epsilon, \eta}$ is not generic, the regularity of this disc can be checked by hand. The reason is that the linearised $\overline{\partial}$-operator at this solution splits into an operator on the normal bundle of the solution, i.e.~the contact planes, and a linear $\overline{\partial}$ operator in a one-dimensional complex situation which is obviously surjective. The claim of the regularity follows since the normal operator has index zero and empty kernel; see \cite[Lemma 8.3(1)]{LiftingPseudoholomorphic} for a similar computation. We relate $\widetilde{\mathcal M}^1_{\widetilde{\bs{\Lambda}}_{\epsilon, \eta}}(a_{i,j}; a_{i,h}, e_{h,j})$ to $\widetilde{\mathcal M}^1_{\bs{\Lambda}_{\epsilon, \eta}}(a_{i,j}; a_{i,h}, e_{h,j})$ by a continuation argument. To that end, we choose a generic path of Legendrian submanifolds starting at $\widetilde{\bs{\Lambda}}_{\epsilon,\eta}$ and ending at $\bs{\Lambda}_{\epsilon,\eta}$, which is indued by a suitable interpolation between the functions $\tilde{k}_j$ and $\tilde{h}_j$ in the above construction (which can be taken sufficiently close to a convex interpolation). First we observe that, by energy considerations, the only possible degeneration of the holomorphic curve as the boundary conditions are deformed along this path is a breaking into a building consisting in a holomorphic disc with a positive puncture asymptotic to $a_{i,j}$ and negative punctures asymptotic to $a_{i,h}$ and $m_{h,j}$ followed by a holomorphic strip positively asymptotic to $m_{h,j}$ and negatively asymptotic to $e_{h,j}$. The first holomorphic disc has index $1-\dim \Lambda$ by Equation \eqref{gradi} (recall that $\dim \Lambda = n-1$), and therefore it can appear in a one-dimensional parameter of Legendrian submanifolds only if $\dim \Lambda =1$. Hence, if $\dim \Lambda >1$ we have $\# \widetilde{\mathcal M}^1_{\widetilde{\bs{\Lambda}}_{\epsilon, \eta}}(a_{i,j}; a_{i,h}, e_{h,j}) = \# \widetilde{\mathcal M}^1_{\bs{\Lambda}_{\epsilon, \eta}}(a_{i,j}; a_{i,h}, e_{h,j})$. If $\dim \Lambda =1$, then there are two strips from $m_{h,j}$ to $e_{h,j}$, so if the moduli spaces degenerate as above, we can glue the holomorphic disc with a positive puncture asymptotic to $a_{i,j}$ and negative punctures asymptotic to $a_{i,h}$ and $m_{h,j}$ to the other strip from $m_{h,j}$ to $e_{h,j}$ and restart the moduli space. Thus, also in the case $\dim \Lambda =1$, we have $\# \widetilde{\mathcal M}^1_{\widetilde{\bs{\Lambda}}_{\epsilon, \eta}}(a_{i,j}; a_{i,h}, e_{h,j}) = \# \widetilde{\mathcal M}^1_{\bs{\Lambda}_{\epsilon, \eta}}(a_{i,j}; a_{i,h}, e_{h,j})$.
\end{proof}

  \begin{prop}\label{louie2}
    For every $\epsilon, \eta$ small enough we have
$$\# \widetilde{\mathcal M}^1_{\bs{\Lambda}_{\epsilon, \eta}}(e_{i,j}; e_{i,h}, e_{h,j}) = 1$$
for every $i<h<j$.
\end{prop}
\begin{proof}
 We choose a perturbation of $\bs{\Lambda}_{\epsilon, 0}$ such that the chords 
$e_{i,h}$ and  $e_{h,j}$ are contained inside the chord $e_{i,j}$, and from here we proceed as in the proof of Lemma \ref{louie}.
\end{proof}
 
\bibliographystyle{plain}
\bibliography{Bibliographie_en}
\end{document}